\documentclass{amsart}

\title[Network parameterizations for the  Grassmannian]
{Network parameterizations for the Grassmannian}

\author{Kelli Talaska and Lauren Williams}
\date{\today}
\thanks{The first author was partially
supported by NSF Grant DMS-1004532.  The second author was
partially supported by an NSF CAREER award and an
Alfred Sloan Fellowship.}
\address{Department of Mathematics, University of California,
Berkeley, CA 94720-3840}
\email{talaska@math.berkeley.edu}
\address{Department of Mathematics, University of California,
Berkeley, CA 94720-3840}
\email{williams@math.berkeley.edu}
\subjclass[2000]{}

\usepackage{rotating}
\newcommand{\rotxc}[1]{\begin{sideways}#1\end{sideways}}
\newcommand{\invert}[1]{\rotxc{\rotxc{#1}}}

\def\Le{\hbox{\invert{$\Gamma$}}}
\countdef\x=23
\countdef\y=24
\countdef\z=25
\countdef\t=26

\def\tbox(#1,#2)#3{
\x=#1 \y=#2
\multiply\x by 12
\multiply\y by 12
\z=\x \t=\y
\advance\z by 12
\advance\t by 12
\psline(\x,\y)(\x,\t)(\z,\t)(\z,\y)(\x,\y)
\advance\x by 6
\advance\y by 6
\rput(\x,\y){{\bf #3}}}

\usepackage{pstricks, pst-node} 
\usepackage{amssymb}
\usepackage{graphicx}
\usepackage{epstopdf}
\usepackage{amscd}
\usepackage{graphics}
\def\proof{\par{\it Proof}. \ignorespaces}
\def\endproof{{\ \vbox{\hrule\hbox{%
     \vrule height1.3ex\hskip0.8ex\vrule}\hrule }}\par}

\theoremstyle{definition}

\theoremstyle{remark}

\numberwithin{equation}{section}

\let\trueint=\int
\let\truesum=\sum
\def\int{\mathop{\textstyle\trueint}\limits}
\def\sum{\mathop{\textstyle\truesum}\limits}
\let\<=\langle
\let\>=\rangle

\def\K{{\mathbb{K}}}
\def\F{{\mathbb{F}}}
\def\P{{\mathcal P}}

\def\SL{\mathrm{SL}}
\def\Sym{{\mathfrak S}}

\def\v{\mathbf{v}}
\def\w{\mathbf{w}}

\usepackage{amsmath, amsthm, amssymb, amsbsy,amsfonts,amscd}
\usepackage[mathscr]{eucal}
\usepackage{epsfig}
\usepackage{young}
\newtheorem{theorem}{Theorem}[section]
\newtheorem*{theorem*}{Theorem}

\newtheorem{definition}[theorem]{Definition}
\newtheorem{proposition}[theorem]{Proposition}
\newtheorem{lemma}[theorem]{Lemma}

\newtheorem{example}[theorem]{Example}
\newtheorem{corollary}[theorem]{Corollary}

\newtheorem{remark}[theorem]{Remark}

\usepackage{amsfonts}
\usepackage{xy}
\usepackage{amssymb}
\usepackage{amsmath}
\newcommand{\inv}{^{-1}}
\newcommand{\To}{\longrightarrow}

\newcommand{\R}{\mathbb R}
\newcommand{\RR}{\mathcal R}

\newcommand{\C}{\mathbb C}

\newcommand{\B}{\mathcal{B}}

\newcommand{\Black}{\raisebox{0.12cm}{\hskip0.14cm\circle*{7}\hskip-0.1cm}}
\newcommand{\White}{\raisebox{0.12cm}{\hskip0.14cm\circle{7}\hskip-0.1cm}}
\newcommand{\smallblack}{\raisebox{0.12cm}{\hskip0.15cm\circle*{4}}\hskip-0.1cm}

\DeclareMathOperator{\In}{in}
\DeclareMathOperator{\Out}{out}

\DeclareMathOperator{\M}{\mathcal M}
\DeclareMathOperator{\I}{\mathcal I}

\newcommand{\thmrefer}[1]{\renewcommand\thetheorem
  {\protect\ref{#1}}\addtocounter{theorem}{-1}}

\xyoption{all}
\CompileMatrices

\begin{document}

\begin{abstract}
Deodhar introduced his decomposition of partial flag varieties
as a tool for understanding Kazhdan-Lusztig polynomials.
The Deodhar decomposition of the Grassmannian is
also useful in the context of soliton
solutions to the KP equation, as shown by Kodama and the second author.
Deodhar components $\mathcal{R}_D$ of the
Grassmannian are in bijection with certain tableaux $D$ called
\emph{Go-diagrams}, and each component
is isomorphic to $(\K^*)^a \times (\K)^b$ for some non-negative
integers $a$ and $b$.

Our main result is
an explicit parameterization of
each Deodhar component in the Grassmannian in terms of networks.
More specifically, from a Go-diagram $D$ we construct a
weighted network $N_D$ and its  \emph{weight matrix} $W_D$,
whose entries enumerate directed paths in $N_D$.  By letting the
weights in the network vary over $\K$ or $\K^*$ as appropriate,
one gets a parameterization of the Deodhar component $\mathcal{R}_D$.
One application of such a parameterization is that one may
immediately determine which Pl\"ucker coordinates are
vanishing and nonvanishing, by using the Lindstrom-Gessel-Viennot Lemma.
We also give a (minimal) characterization of
each Deodhar component in terms of Pl\"ucker coordinates.
A main tool for us is the work of Marsh and Rietsch \cite{MR}
on Deodhar components in the flag variety.
\end{abstract}

\maketitle

\setcounter{tocdepth}{1}
\tableofcontents

\section{Introduction}

There is a remarkable subset of the
real Grassmannian $Gr_{k,n}(\R)$
called its \emph{totally non-negative part} $(Gr_{k,n})_{\geq 0}$
\cite{Lusztig2, Postnikov}, which may be defined as the subset of the
real Grassmannian where all Pl\"ucker coordinates have the same sign.
Postnikov showed that $(Gr_{k,n})_{\geq 0}$ has a decomposition into
\emph{positroid cells},
which are indexed by certain tableaux called $\Le$-diagrams.
He also gave explicit parameterizations of each cell.  In particular, he showed
that from each $\Le$-diagram one can produce a planar \emph{network}, and that
one can write down a parameterization of the corresponding cell using the
\emph{weight matrix} of that network.  This parameterization
shows that the cell
is isomorphic to $\R_{>0}^d$ for some $d$.
Such a parameterization is convenient, because for example,
one may read off formulas for Pl\"ucker coordinates from non-intersecting
paths in the network, using
the Lindstrom-Gessel-Viennot Lemma.

A natural question is whether these network parameterizations
for positroid cells can be extended from
$(Gr_{k,n})_{\geq 0}$ to the entire real Grassmannian $Gr_{k,n}(\R)$.
In this paper we give an affirmative answer to this question,
by replacing the positroid cell decomposition with the Deodhar
decomposition of the Grassmannian $Gr_{k,n}(\K)$ (here $\K$
is an arbitrary field).

The components of the Deodhar decomposition are
not in general cells, but nevertheless have
a simple topology: by \cite{Deodhar, Deodhar2}, each
one is isomorphic to
$(\K^*)^a \times (\K)^b$.  The relation of the Deodhar decomposition
of $Gr_{k,n}(\R)$ to Postnikov's cell decomposition of $(Gr_{k,n})_{\geq 0}$
is as follows: the intersection of a Deodhar component
$\mathcal{R}_D \cong (\R^*)^a \times (\R)^b$
with $(Gr_{k,n})_{\geq 0}$ is precisely one positroid cell isomorphic
to $(\R_{>0})^a$
if $b=0$, and is empty otherwise.  In particular, when one
intersects the Deodhar decomposition with $(Gr_{k,n})_{\geq 0}$,
one obtains the positroid cell decomposition of $(Gr_{k,n})_{\geq 0}$.
There is a related \emph{positroid stratification} of the real
Grassmannian, and each positroid stratum is a union of Deodhar components.

As for the combinatorics,
components of the Deodhar decomposition are indexed by
\emph{distinguished subexpressions} \cite{Deodhar, Deodhar2}, or
equivalently, by certain tableaux
called \emph{Go-diagrams} \cite{KW2}, which generalize
$\Le$-diagrams.
In this paper we associate a network to each Go-diagram,
and write down a parameterization of the corresponding
Deodhar component using the weight matrix of that network.
Our construction generalizes Postnikov's, but
our networks are no longer planar in general.

Our main results can be summed up as follows.  See Theorems~\ref{th:main} and~\ref{th:characterize} and the constructions preceding them for complete details.
\begin{theorem*} Let $\K$ be an arbitrary field.
\begin{itemize}
\item Every point in $Gr_{k,n}(\K)$ can be realized as the weight matrix of a unique network associated to a Go-diagram, and  we can explicitly construct the corresponding network.  The networks corresponding to points in the same Deodhar component have the same underlying graph, but different weights.
\item Every Deodhar component may be characterized by the vanishing and nonvanishing of certain Pl\"ucker coordinates.  Using this characterization, we can also explicitly construct the network associated to a point given either by a matrix repsresentative or by a list of Pl\"ucker coordinates.
\end{itemize}
\end{theorem*}

To illustrate the main results, we provide a small example here.  More complicated examples may be seen throughout the rest of the paper.

\begin{figure}[ht]
\setlength{\unitlength}{0.7mm}
\begin{center}
\begin{picture}(30,30)
  \put(5,25){\line(1,0){20}}
  \put(5,15){\line(1,0){20}}
  \put(5,5){\line(1,0){20}}
  \put(5,5){\line(0,1){20}}
  \put(15,5){\line(0,1){20}}
  \put(25,5){\line(0,1){20}}

  \put(7,18){\scalebox{1.7}{$+$}}
  \put(17,18){\scalebox{1.7}{$+$}}
  \put(7,8){\scalebox{1.7}{$+$}}
  \put(17,8){\scalebox{1.7}{$+$}}

\put(2,28){$D_1$}

  \end{picture}\quad\psset{unit=.6cm,dotstyle=o,dotsize=5pt 0,linewidth=0.8pt,arrowsize=3pt 2,arrowinset=0.25}
\begin{pspicture*}(1,0.8)(8,5)
\psline{->}(5,2)(4,2) \uput[u](4.5,2){$a_1$}
\psline{->}(4,2)(2,2) \uput[u](2.5,2){$a_2$}
\psline{->}(5,4)(4,4) \uput[u](4.5,4){$a_3$}
\psline{->}(4,4)(2,4) \uput[u](2.5,4){$a_4$}

\psline{->}(2,4)(2,2)
\psline{->}(2,2)(2,1)
\psline{->}(4,4)(4,2)
\psline{->}(4,2)(4,1)

\psdots[dotstyle=*,linecolor=black](2,4)
\psdots[dotstyle=*,linecolor=black](2,2)
\psdots[dotstyle=*,linecolor=black](4,4)
\psdots[dotstyle=*,linecolor=black](4,2)

\psdots[dotstyle=*,linecolor=black](5,4)
\uput[r](5.08,4.12){\large{1}}
\psdots[dotstyle=*,linecolor=black](5,2)
\uput[r](5.08,2.12){\large{2}}
\psdots[dotstyle=*,linecolor=black](4,1)
\uput[r](4.08,1.12){\large{3}}
\psdots[dotstyle=*,linecolor=black](2,1)
\uput[r](2.08,1.12){\large{4}}

\end{pspicture*}\quad\begin{picture}(30,30)
  \put(5,25){\line(1,0){20}}
  \put(5,15){\line(1,0){20}}
  \put(5,5){\line(1,0){20}}
  \put(5,5){\line(0,1){20}}
  \put(15,5){\line(0,1){20}}
  \put(25,5){\line(0,1){20}}

  \put(8,20){\hskip0.15cm\circle*{5}}
  \put(17,18){\scalebox{1.7}{$+$}}
  \put(7,8){\scalebox{1.7}{$+$}}
  \put(18,10){\hskip0.15cm\circle{5}}

\put(2,28){$D_2$}

  \end{picture}\quad\begin{pspicture*}(1,0.8)(6,5)
\psline{->}(5,2)(2,2) \uput[u](2.5,2){$a_2$}
\psline{->}(5,4)(4,4) \uput[u](4.5,4){$a_3$}
\psline{->}(4,4)(2,4) \uput[u](2.5,4){$c_4$}

\psline{->}(2,4)(2,2)
\psline{->}(2,2)(2,1)
\psline{->}(4,4)(4,1)

\psdots[dotsize=9pt 0,dotstyle=*,linecolor=black](2,4)
\psdots[dotstyle=*,linecolor=black](2,2)
\psdots[dotstyle=*,linecolor=black](4,4)

\psdots[dotstyle=*,linecolor=black](5,4)
\uput[r](5.08,4.12){\large{1}}
\psdots[dotstyle=*,linecolor=black](5,2)
\uput[r](5.08,2.12){\large{2}}
\psdots[dotstyle=*,linecolor=black](4,1)
\uput[r](4.08,1.12){\large{3}}
\psdots[dotstyle=*,linecolor=black](2,1)
\uput[r](2.08,1.12){\large{4}}
\end{pspicture*}
\end{center}
\caption{The diagrams and networks associated to $\mathcal{R}_{D_1}$ and $\mathcal{R}_{D_2}$ in Example~\ref{ex:Gr24}.}
\label{fig:Gr24}
\end{figure}

\begin{example}\label{ex:Gr24}
Consider the Grassmannian $Gr_{2,4}$.  The large Schubert cell in this Grassmannian can be characterized as
$$\Omega_{\lambda} = \{A \in Gr_{2,4} \ \vert \ \Delta_{1,2}(A) \neq 0\},$$ where $\Delta_J$ denotes the Pl\"ucker coordinate corresponding to the column set $J$ in a matrix representative of a point in $Gr_{2,4}$.
This Schubert cell contains multiple positroid strata, including
$S_{\I}$, where $\I$ is the Grassmann necklace $\I = (12,23,34,14).$  This positroid stratum can also be characterized by the non-vanishing of certain Pl\"ucker coordinates:
$$S_{\I} = \{A \in Gr_{2,4} \ \vert \ \Delta_{1,2}(A) \neq 0, \
\Delta_{2,3}(A) \neq 0,\ \Delta_{3,4}(A) \neq 0,\ \Delta_{1,4}(A) \neq 0\}.$$

Figure \ref{fig:Gr24} shows two Go-diagrams $D_1$ and $D_2$ and their associated networks. Note that the network on the right is not planar.  The weight matrices associated to these diagrams are
\[
\begin{pmatrix}
1 & 0 & -a_3 & -(a_3 a_4 + a_3 a_2) \\
0 & 1 & a_1 & a_1 a_2
\end{pmatrix}
\text{ and }
\begin{pmatrix}
1 & 0 & -a_3 & -a_3 c_4 \\
0 & 1 & 0 &  a_2
\end{pmatrix}.
\]
The positroid stratum $S_{\I}$ is the disjoint union
of the two corresponding Deodhar components $\mathcal{R}_{D_1}$ and $\mathcal{R}_{D_2}$, which can be characterized in terms of vanishing and nonvanishing of minors as:
$$\mathcal{R}_{D_1} = \{A \in S_{\I} \ \vert \ \Delta_{1,3} \neq 0\} \text{ and }
\mathcal{R}_{D_2} = \{A \in S_{\I} \ \vert \ \Delta_{1,3} = 0\}.$$
Note that if one lets the $a_i$'s range over $\K^*$ and lets $c_4$ range over
$\K$, then we see that $\mathcal{R}_{D_1} \cong (\K^*)^4$ and $\mathcal{R}_{D_2} \cong (\K^*)^2 \times \K.$
\end{example}

There are several applications of our construction.
First, as a special case of our theorem, one may parameterize
all $k \times n$ matrices using
networks.  Second, by applying the Lindstrom-Gessel-Viennot Lemma
to a given network, one may write down explicit formulas
for Pl\"ucker coordinates in terms of collections of non-intersecting
paths in the network.  Third, building upon work
of \cite{KW2}, we obtain (minimal) descriptions
of Deodhar components in the Grassmannian, in terms of vanishing
and nonvanishing of Pl\"ucker coordinates. It follows that
each Deodhar component is a union of matroid strata.

Although less well known than the Schubert decomposition and
matroid stratification,
the Deodhar decomposition is very interesting in its
own right.  Deodhar's original motivation for introducing
his decomposition was the desire to
understand Kazhdan-Lusztig polynomials.  In the flag variety,
one may intersect two opposite Schubert cells, obtaining a
Richardson variety, which Deodhar showed is a union of
Deodhar components.  Each Richardson variety $\RR_{v,w}(q)$ may be defined
over a finite field $\K = \F_q$, and in this case, the number
of points determines the $R$-polynomials
$R_{v,w}(q) = \#(\RR_{v,w}(\F_q))$, introduced by
Kazhdan and Lusztig \cite{KazLus} to give a recursive formula for the
Kazhdan-Lusztig polynomials.
Since each Deodhar component is isomorphic
to $(\F^*_q)^a \times (\F_q)^b$ for some $a$ and $b$,
if one understands the decomposition of a Richardson variety
into Deodhar components, then in principle one may
compute the $R$-polynonomials
and hence Kazhdan-Lusztig polynomials.

Another reason for our interest in the Deodhar decomposition is its
relation to soliton solutions of
the KP equation.  It is well-known that
from each point $A$ in the real Grassmannian, one may construct a soliton
solution $u_A(x,y,t)$ of the KP equation.  It was shown
in recent work of Kodama
and the second author \cite{KW2} that when the time variable $t$
tends to $-\infty$, the combinatorics of the solution
$u_A(x,y,t)$ depends precisely on which Deodhar component $A$ lies in.

One final result of this paper is the verification that two notions of 
total positivity for the Grassmannian coincide.  In 
\cite{Lusztig2}, Lusztig defined the totally non-negative part of 
any partial flag variety in a Lie-theoretic way. He also 
conjectured a cell decomposition for it, which was proved by 
Rietsch \cite{RietschThesis}. Independently 
Postnikov defined the totally non-negative part of the real Grassmannian
in terms of Pl\"ucker coordinates, and gave a cell decomposition of it.  
It is not obvious that Lusztig's
definitions (for $Gr_{k,n}(\R)$) coincide with Postnikov's; 
however, this has been verified by 
Rietsch 
\cite{RietschPC}.
In this paper we give a new proof that the two notions of 
total positivity coincide.
\begin{corollary}\label{cor:coincide}
Lusztig's definition of the totally non-negative part of 
$Gr_{k,n}(\R)$ and its cell decomposition
coincides with Postnikov's definition of the 
totally non-negative part of $Gr_{k,n}(\R)$ and its cell decomposition.
\end{corollary}

The outline of this paper is as follows.  In Section \ref{sec:Gr},
we give some background on the Grassmannian and its decompositions,
including the Schubert decomposition, the positroid stratification,
and the matroid stratification. In Section \ref{sec:main}, we present
our main construction: we explain how to construct
a network from each diagram, then use that network to write
down a parameterization of a subset of the Grassmannian
that we call a network component.
In Section \ref{sec:project}
we define Deodhar's decomposition of the flag variety, and its
projection to the Grassmannian.  We also describe parameterizations
of Deodhar components in the flag variety
which are due to Marsh and Rietsch \cite{MR}.
In Sections \ref{sec:MR-matrix} and \ref{sec:proof} we prove that
after a rational transformation of variables,
our network parameterizations coincide with the projections of the
Marsh-Rietsch parameterizations.
Finally in Section \ref{sec:characterization}
we give a characterization of Deodhar components in terms of the vanishing and nonvanishing of certain
Pl\"ucker coordinates.

\textsc{Acknowledgements:}
We are grateful to Sara Billey for numerous helpful comments on
the first version of this paper, as well as to an anonymous referee
for useful comments.
L.W. is also grateful to Yuji Kodama for their joint work
on soliton solutions of the KP equation, which provided
motivation for this project.


\section{Background on the Grassmannian}\label{sec:Gr}

The \emph{Grassmannian} $Gr_{k,n}$ is the space of all
$k$-dimensional subspaces of an $n$-dimensional vector space $\K^n$.
In this paper we will usually let
$\K$ be an arbitrary field,
though we will often think of it as $\R$ or $\C$.
An element of
$Gr_{k,n}$ can be viewed as a full-rank $k\times n$ matrix modulo left
multiplication by nonsingular $k\times k$ matrices.  In other words, two
$k\times n$ matrices represent the same point in $Gr_{k,n}$ if and only if they
can be obtained from each other by row operations.
Let $\binom{[n]}{k}$ be the set of all $k$-element subsets of $[n]:=\{1,\dots,n\}$.
For $I\in \binom{[n]}{k}$, let $\Delta_I(A)$
be the {\it Pl\"ucker coordinate}, that is, the maximal minor of the $k\times n$ matrix $A$ located in the column set $I$.
The map $A\mapsto (\Delta_I(A))$, where $I$ ranges over $\binom{[n]}{k}$,
induces the {\it Pl\"ucker embedding\/} $Gr_{k,n}\hookrightarrow \mathbb{KP}^{\binom{n}{k}-1}$ into projective space.

We now describe several useful decompositions of the Grassmannian:
the Schubert decomposition, the positroid
stratification, and the matroid stratification.  Note that the matroid
stratification refines the positroid stratification, which refines the Schubert
decomposition.  The main subject of this paper is the
{\it Deodhar} decomposition of the Grassmannian, which refines the
positroid stratification, and is refined by the matroid stratification
(as we prove in Corollary \ref{cor:refine}).

\subsection{The Schubert decomposition of $Gr_{k,n}$}\label{sec:Schub}

Throughout this paper, we identify partitions with their Young diagrams. Recall that the partitions $\lambda$ contained in
a $k \times (n-k)$ rectangle are in bijection with $k$-element subset $I \subset [n]$.
The boundary of the Young diagram of such a partition
$\lambda$ forms a lattice path from the upper-right corner to the lower-left
corner of the rectangle.  Let us label the $n$ steps
in this path by the numbers $1,\dots,n$, and define
$I = I(\lambda)$ as the set of labels on the $k$ vertical steps in
the path. Conversely, we let $\lambda(I)$ denote the
partition corresponding to the subset $I$.

\begin{definition}
For each partition $\lambda$ contained in a $k \times (n-k)$ rectangle, we define the
\emph{Schubert cell}
$$\Omega_{\lambda} = \{A \in Gr_{k,n} \ \vert \ I(\lambda) \text{ is
the lexicographically minimal subset such that }\Delta_{I(\lambda)}(A) \neq 0 \}.$$  As $\lambda$ ranges over the partitions contained in a $k \times (n-k)$ rectangle,
this gives the \emph{Schubert decomposition}
of the Grassmannian $Gr_{k,n}$, i.e.
\[
Gr_{k,n}=\bigsqcup_{\lambda\subset (n-k)^k}\,\Omega_{\lambda}.
\]
\end{definition}

We now define the \emph{shifted linear order} $<_i$ (for $i\in [n]$) to be the total order on $[n]$
defined by $$i <_i i+1 <_i i+2 <_i \dots <_i n <_i 1 <_i \dots <_i i-1.$$
One can then define \emph{cyclically shifted Schubert cells} as follows.

\begin{definition} \label{def:Schubert}
For each partition $\lambda$ contained in a $k \times (n-k)$ rectangle, and each $i \in [n]$, we define the
\emph{cyclically shifted Schubert cell}
$$\Omega_{\lambda}^i = \{A \in Gr_{k,n} \ \vert \ I(\lambda) \text{ is
the lexicographically minimal subset with respect to }<_i \text{ such that }
\Delta_{I(\lambda)} \neq 0\}.$$
\end{definition}

\subsection{The positroid stratification of $Gr_{k,n}$}\label{sec:positroid}

The \emph{positroid stratification} of the  Grassmannian
$Gr_{k,n}$ is obtained by taking the
simultaneous refinement of the $n$ Schubert decompositions with respect to
the $n$ shifted linear orders $<_i$.  This stratification
was first considered
by Postnikov \cite{Postnikov}, who showed that
the strata are conveniently described
in terms of
\emph{Grassmann necklaces}, as well as \emph{decorated permutations}
and \emph{$\Le$-diagrams}.  Postnikov coined the terminology
\emph{positroid} because the intersection of the positroid stratification
of the real Grassmannian
with the \emph{totally non-negative part of the Grassmannian}
$(Gr_{k,n})_{\geq 0}$ gives a cell decomposition of
$(Gr_{k,n})_{\geq 0}$  (whose cells are called
\emph{positroid cells}).

\begin{definition}\cite[Definition 16.1]{Postnikov}
A \emph{Grassmann necklace} is a sequence
$\I = (I_1,\dots,I_n)$ of subsets $I_r \subset [n]$ such that,
for $i\in [n]$, if $i\in I_i$ then $I_{i+1} = (I_i \setminus \{i\}) \cup \{j\}$,
for some $j\in [n]$ ($j$ may coincide with $i$); and if $i \notin I_i$ then $I_{i+1} = I_i$.
(Here indices $i$ are taken modulo $n$.)  In particular, we have
$|I_1| = \dots = |I_n|$, which is equal to some $k \in [n]$.  We then say that
$\I$ is a Grassmann necklace of \emph{type} $(k,n)$.
\end{definition}

\begin{example}\label{ex1}
$\mathcal I = (1345, 3456, 3456, 4567, 4567, 1467, 1478, 1348)$ is
an example of a Grassmann necklace of type $(4,8)$.
\end{example}

\begin{lemma}\cite[Lemma 16.3]{Postnikov} \label{lem:necklace}
Given $A\in Gr_{k,n}$,
let $\mathcal I(A) = (I_1,\dots,I_n)$ be the sequence of subsets in
$[n]$ such that, for $i \in [n]$, $I_i$ is the lexicographically
minimal subset of $\binom{[n]}{k}$  with respect to the shifted linear order
$<_i$ such that $\Delta_{I_i}(A) \neq 0$.
Then $\I(A)$ is a Grassmann necklace of type $(k,n)$.
\end{lemma}

The \emph{positroid stratification}
of $Gr_{k,n}$ is defined as follows.

\begin{definition}\label{def:pos}
Let $\I = (I_1,\dots,I_n)$ be a Grassmann necklace of type $(k,n)$.
The \emph{positroid stratum}
$S_{\I}$ is defined to be
$$S_{\I} = \{A \in Gr_{k,n} \ \vert \ \I(A) = \I \}.$$
Equivalently,
each positroid stratum is an intersection of $n$
cyclically shifted Schubert cells, that is,
$$S_{\I} = \bigcap_{i=1}^n ~\Omega_{\lambda(I_i)}^i\ .$$
\end{definition}

Grassmann necklaces are in bijection with
tableaux called \emph{$\Le$-diagrams}.

\begin{definition}\cite[Definition 6.1]{Postnikov}\label{def:Le}
Fix $k$, $n$.  A {\it $\Le$-diagram}
$(\lambda, D)_{k,n}$ of type $(k,n)$
is a partition $\lambda$ contained in a $k \times (n-k)$ rectangle
together with a filling $D: \lambda \to \{0,+\}$ of its boxes which has the
{\it $\Le$-property}:
there is no $0$ which has a $+$ above it and a $+$ to its
left.\footnote{This forbidden pattern is in the shape of a backwards $L$,
and hence is denoted $\Le$ and pronounced ``Le."}  (Here, ``above" means above and in the same column, and
``to its left" means to the left and in the same row.)
\end{definition}
In Figure \ref{LeDiagram} we give
an example of a $\Le$-diagram.
\begin{figure}[h]
\setlength{\unitlength}{0.5mm}
\centering
\begin{picture}(50,50)
  \put(5,45){\line(1,0){40}}
  \put(5,35){\line(1,0){40}}
  \put(5,25){\line(1,0){30}}
  \put(5,15){\line(1,0){30}}
  \put(5,5){\line(1,0){20}}
  \put(5,5){\line(0,1){40}}
  \put(15,5){\line(0,1){40}}
  \put(25,5){\line(0,1){40}}
  \put(35,15){\line(0,1){30}}
  \put(45,35){\line(0,1){10}}

  \put(7,38){\scalebox{1.2}{$+$}}
  \put(18,38){0}
  \put(27,38){\scalebox{1.2}{$+$}}
  \put(37,38){\scalebox{1.2}{$+$}}
  \put(7,28){\scalebox{1.2}{$+$}}
  \put(18,28){0}
  \put(27,28){\scalebox{1.2}{$+$}}
  \put(8,18){0}
  \put(18,18){0}
  \put(27,18){\scalebox{1.2}{$+$}}
  \put(7,8){\scalebox{1.2}{$+$}}
  \put(18,8){0}

\end{picture}
\caption{A Le-diagram $L=(\lambda,D)_{k,n}$.
\label{LeDiagram}}
\end{figure}
\subsection{The matroid stratification of $Gr_{k,n}$}

\begin{definition}\label{def:matroid}
A \emph{matroid} of \emph{rank} $k$ on the set $[n]$ is a nonempty collection
$\M \subset \binom{[n]}{k}$ of $k$-element subsets in $[n]$, called \emph{bases}
of $\M$, that satisfies the \emph{exchange axiom}:\\
For any $I,J \in \M$ and $i \in I$ there exists $j\in J$ such that
$(I \setminus \{i\}) \cup \{j\} \in \M$.
\end{definition}

Given an element $A \in Gr_{k,n}$, there is an associated matroid
$\M_A$ whose bases are the $k$-subsets $I \subset [n]$ such that
$\Delta_I(A) \neq 0$.

\begin{definition}
Let $\M \subset \binom{[n]}{k}$ be a matroid.
The \emph{matroid stratum}
$S_{\M}$ is defined to be
$$S_{\M} = \{A \in Gr_{k,n} \ \vert \ \Delta_I(A) \neq 0 \text{ if and only if }
I\in \M \}.$$
This gives a stratification of $Gr_{k,n}$ called the
\emph{matroid stratification}, or \emph{Gelfand-Serganova stratification}.
\end{definition}

\begin{remark}
Clearly the matroid stratification refines the positroid stratification, which
in turn refines the Schubert decomposition.
\end{remark}


\section{The main result: network parameterizations from
Go-diagrams}\label{sec:main}

In this section we define certain tableaux called
\emph{Go-diagrams}, then explain how to parameterize the
Grassmannian using networks associated to Go-diagrams.
First we will define more general tableaux called \emph{diagrams}.

\subsection{Diagrams and networks}

\begin{definition}
Let $\lambda$ be a partition contained in a $k \times (n-k)$
rectangle.
A \emph{diagram} in $\lambda$ is an arbitrary filling of the boxes of
$\lambda$ with pluses $+$,
black stones \Black,
and white stones \White.
\end{definition}

To each diagram $D$ we associate a network $N_D$ as follows.
\begin{definition}\label{def:network}
Let $\lambda$ be a partition with $\ell$ boxes contained in a
$k \times (n-k)$
rectangle, and let $D$ be a diagram in $\lambda$.
Label the boxes of $\lambda$ from $1$ to $\ell$, starting from the
rightmost box in the bottom row,
then reading right to left across the bottom row, then right to left across
the row above that, etc.
The \emph{(weighted) network} $N_D$ associated to $D$ is a directed graph obtained
as follows:
\begin{itemize}
\item Associate an \emph{internal vertex} to each $+$ and each
\Black;
\item After labeling the southeast border of the Young diagram
with the numbers $1,2,\dots,n$ (from northeast to southwest),
associate a \emph{boundary vertex} to each number;
\item From each internal vertex, draw an edge right to the nearest
$+$-vertex or boundary vertex;
\item From each internal vertex, draw an edge down to the nearest
$+$-vertex or boundary vertex;
\item Direct all edges left and down.  After doing so,
$k$ of the boundary
vertices become \emph{sources} and the remaining $n-k$ boundary vertices
become \emph{sinks}.
\item If $e$ is a horizontal edge whose left vertex is a $+$-vertex
(respectively
a \Black-vertex)
in box $b$, assign $e$ the weight $a_b$ (respectively $c_b$).  We think of $a_b$ and $c_b$ as indeterminates, but later they will be elements of $\K^*$ and $\K$ respectively.
\item If $e$ is a vertical edge, assign $e$ the weight $1$.
\end{itemize}
\end{definition}

Note that in general such a directed graph is not planar, as two
edges may cross over each other without meeting at a vertex.
See Figure~\ref{fig:network} for an example of a diagram and its associated network.


\begin{figure}[ht]
\setlength{\unitlength}{0.9mm}
\begin{center}
\begin{picture}(50,50)
  \put(5,45){\line(1,0){40}}
  \put(5,35){\line(1,0){40}}
  \put(5,25){\line(1,0){30}}
  \put(5,15){\line(1,0){30}}
  \put(5,5){\line(1,0){20}}
  \put(5,5){\line(0,1){40}}
  \put(15,5){\line(0,1){40}}
  \put(25,5){\line(0,1){40}}
  \put(35,15){\line(0,1){30}}
  \put(45,35){\line(0,1){10}}

  \put(7,38){\scalebox{1.7}{$+$}}
  \put(17,38){\scalebox{1.7}{$+$}}
  \put(27,38){\scalebox{1.7}{$+$}}
  \put(37,38){\scalebox{1.7}{$+$}}
  \put(7,28){\scalebox{1.7}{$+$}}
  \put(18,30){\hskip0.15cm\circle*{5}}
  \put(27,28){\scalebox{1.7}{$+$}}
  \put(8,20){\hskip0.15cm\circle*{5}}
  \put(17,18){\scalebox{1.7}{$+$}}
  \put(28,20){\hskip0.15cm\circle{5}}
  \put(7,8){\scalebox{1.7}{$+$}}
  \put(18,10){\hskip0.15cm\circle{5}}

  \end{picture}
\quad
\psset{unit=.5cm,dotstyle=o,dotsize=5pt 0,linewidth=0.8pt,arrowsize=3pt 2,arrowinset=0.25}
\begin{pspicture*}(0,0.45)(10,9)
\psline{->}(5,2)(2,2) \uput[u](2.5,2){$a_2$}
\psline{->}(7,4)(4,4) \uput[u](4.5,4){$a_4$}
\psline{->}(4,4)(2,4) \uput[u](2.5,4){$c_5$}
\psline{->}(7,6)(6,6) \uput[u](6.5,6){$a_6$}
\psline{->}(6,6)(4,6) \uput[u](4.5,6){$c_7$}
\pscurve{->}(6,6)(4,7)(2,6) \uput[u](2.5,6.2){$a_8$}
\psline{->}(9,8)(8,8) \uput[u](8.5,8){$a_9$}
\psline{->}(8,8)(6,8) \uput[u](6.5,8){$a_{10}$}
\psline{->}(6,8)(4,8) \uput[u](4.5,8){$a_{11}$}
\psline{->}(4,8)(2,8) \uput[u](2.5,8){$a_{12}$}

\psline{->}(2,8)(2,6)
\pscurve{->}(2,6)(1.4,4)(2,2)
\psline{->}(2,4)(2,2)
\psline{->}(2,2)(2,1)
\pscurve{->}(4,8)(3.4,6)(4,4)
\psline{->}(4,6)(4,4)
\psline{->}(4,4)(4,1)
\psline{->}(6,8)(6,6)
\psline{->}(6,6)(6,3.02)
\psline{->}(8,8)(8,7)

\psdots[dotstyle=*,linecolor=black](2,8)
\psdots[dotstyle=*,linecolor=black](2,6)
\psdots[dotsize=9pt 0,dotstyle=*,linecolor=black](2,4)
\psdots[dotstyle=*,linecolor=black](2,2)
\psdots[dotstyle=*,linecolor=black](4,4)
\psdots[dotsize=9pt 0,dotstyle=*,linecolor=black](4,6)
\psdots[dotstyle=*,linecolor=black](4,8)
\psdots[dotstyle=*,linecolor=black](6,8)
\psdots[dotstyle=*,linecolor=black](8,8)
\psdots[dotstyle=*,linecolor=black](6,6)

\psdots[dotstyle=*,linecolor=black](9,8)
\uput[r](9.08,8.12){\large{1}}
\psdots[dotstyle=*,linecolor=black](8,7)
\uput[r](8.08,7.12){\large{2}}
\psdots[dotstyle=*,linecolor=black](7,6)
\uput[r](7.08,6.12){\large{3}}
\psdots[dotstyle=*,linecolor=black](7,4)
\uput[r](7.08,4.12){\large{4}}
\psdots[dotstyle=*,linecolor=black](6,3)
\uput[r](6.08,3.12){\large{5}}
\psdots[dotstyle=*,linecolor=black](5,2)
\uput[r](5.08,2.12){\large{6}}
\psdots[dotstyle=*,linecolor=black](4,1)
\uput[r](4.08,1.12){\large{7}}
\psdots[dotstyle=*,linecolor=black](2,1)
\uput[r](2.08,1.12){\large{8}}

\end{pspicture*}
\end{center}
\caption{An example of a diagram and its associated network.}
\label{fig:network}
\end{figure}
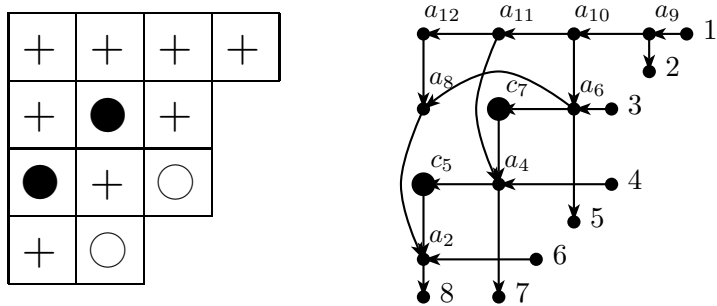


We now explain how to associate a \emph{weight matrix} to such
a network.

\begin{definition}\label{def:weightmatrix}
Let $N_D$ be a network as in Definition \ref{def:network}.
Let $I = \{i_1 < i_2 <\dots <i_k\} \subset [n]$ denote the sources.
If $P$ is a directed path in the network, let $w(P)$ denote
the product of all weights along $P$.  If $P$ is the empty path
which starts and ends at the same boundary vertex, we let $w(P)=1$.
If $s$ is a source and $t$ is any boundary vertex, define
$$W_{st} = \pm \sum_{P} w(P),$$ where the sum is over all paths
$P$ from $s$ to $t$.
The sign is chosen (uniquely) so that
$$\Delta_{I\setminus\{s\}\cup\{t\}}(W_D)=\sum_{P} w(P), \text{ where }$$
$$W_D = (W_{st})$$  is the
$k \times n$ \emph{weight matrix}.
We make the convention that the rows of $W_D$ are indexed by the sources
$i_1, \dots, i_k$ from top to bottom, and its columns are indexed
by $1,2,\dots,n$ from left to right.
An equivalent way to define the sign of $W_{st}$ is to
let $q = |\{s+1, s+2, \dots, t-1\} \cap I|$, i.e.
the number of sources which are strictly between $s$ and $t$.
Then the sign of $W_{st}$ is $(-1)^q$.
\end{definition}

\begin{example}
The weight matrix associated to the network in Figure
\ref{fig:network} is

\[\left(
    \begin{array}{cccccccc}
      1 & a_9 & 0 & 0 & a_9a_{10} & 0 & -a_9a_{10}(a_{11}+c_7) & -a_9a_{10}(a_{11}a_{12}+a_{11}c_5+a_8+c_7c_5) \\
      0 & 0 & 1 & 0 & -a_6 & 0 & a_6c_7 & a_6a_8+a_6c_7c_5 \\
      0 & 0 & 0 & 1 & 0 & 0 & -a_4 & -a_4c_5 \\
      0 & 0 & 0 & 0 & 0 & 1 & 0 & a_2 \\
    \end{array}
  \right)
\]

\end{example}

\subsection{Distinguished expressions}
We now review the notion of
distinguished subexpressions, as in
\cite{Deodhar} and \cite{MR}.
This definition will be essential for defining Go-diagrams.
We  assume the reader is familiar
with the (strong) Bruhat order $<$ on  $W=\Sym_n$, and the
basics of reduced expressions, as in  \cite{BB}.

Let $\w:= s_{i_1}\dots s_{i_m}$ be a reduced expression for $w\in W$.
A  {\it subexpression} $\v$ of $\w$
is a word obtained from the reduced expression $\w$ by replacing some of
the factors with $1$. For example, consider a reduced expression in $\Sym_4$, say $s_3
s_2 s_1 s_3 s_2 s_3$.  Then $s_3 s_2\, 1\, s_3 s_2\, 1$ is a
subexpression of $s_3 s_2 s_1 s_3 s_2 s_3$.
Given a subexpression $\v$,
we set $v_{(k)}$ to be the product of the leftmost $k$
factors of $\v$, if $k \geq 1$, and $v_{(0)}=1$.

\begin{definition}\label{d:Js}\cite{MR, Deodhar}
Given a subexpression $\v$ of a reduced expression $\w=
s_{i_1} s_{i_2} \dots s_{i_m}$, we define
\begin{align*}
J^{\circ}_\v &:=\{k\in\{1,\dotsc,m\}\ |\  v_{(k-1)}<v_{(k)}\},\\
J^{+}_\v\, &:=\{k\in\{1,\dotsc,m\}\ |\  v_{(k-1)}=v_{(k)}\},\\
J^{\bullet}_\v &:=\{k\in\{1,\dotsc,m\}\ |\  v_{(k-1)}>v_{(k)}\}.
\end{align*}
The expression  $\v$
is called {\it
non-decreasing} if $v_{(j-1)}\le v_{(j)}$ for all $j=1,\dotsc, m$,
e.g.\ if $J^{\bullet}_\v=\emptyset$.
\end{definition}

\begin{definition}[Distinguished subexpressions]\cite[Definition 2.3]{Deodhar}
A subexpression $\v$ of $\w$ is called {\it distinguished}
if we have
\begin{equation}\label{e:dist}
v_{(j)}\le v_{(j-1)}\ s_{i_j}\qquad \text{for all
$~j\in\{1,\dotsc,m\}$}.
\end{equation}
In other words, if right multiplication by $s_{i_j}$ decreases the
length of $v_{(j-1)}$, then in a distinguished subexpression we
must have
$v_{(j)}=v_{(j-1)}s_{i_j}$.

We write $\v\prec\w$ if $\v$ is a distinguished subexpression of
$\w$.
\end{definition}

\begin{definition}[Positive distinguished subexpressions]
We call a
subexpression $\v$ of $\w$ a {\it positive distinguished subexpression}
(or a PDS for short) if
 \begin{equation}\label{e:PositiveSubexpression}
v_{(j-1)}< v_{(j-1)}s_{i_j} \qquad \text{for all
$~j\in\{1,\dotsc,m\}$}.
 \end{equation}
In other words, it is distinguished and non-decreasing.
\end{definition}

\begin{lemma}\label{l:positive}\cite{MR}
Given $v\le w$
and a reduced expression $\w$ for $w$,
there is a unique PDS $\v_+$ for $v$ in $\w$.
\end{lemma}


\subsection{Go-diagrams}\label{Deodhar-combinatorics}
In this section we explain how to index
distinguished subexpressions by certain
tableaux called \emph{Go-diagrams}, which were introduced in \cite{KW2}.
Go-diagrams are certain fillings
of Young diagrams by
pluses $+$, \emph{black stones} \Black,
and \emph{white stones} \White.\footnote{In KW2, we used a slightly different convention
and used blank boxes in place of $+$'s.}

Fix $k$ and $n$.  Let
$W_k = \langle s_1,s_2,\dots,\hat{s}_{n-k},\dots,s_{n-1} \rangle$
be a parabolic subgroup of $W = \Sym_n$.
Let $W^k$ denote the set of minimal-length
coset representatives of $W/W_k$.
Recall that a \emph{descent} of a permutation $\pi$
is a position $j$ such that $\pi(j)>\pi(j+1)$.
Then $W^k$ is the subset of permutations of $\Sym_n$
which have at most one descent; and that descent must be in position $n-k$.

It is well-known that elements $w$ of $W^k$ can be identified
with partitions $\lambda_w$ contained in a $k \times (n-k)$ rectangle:
if $w=(w_1,\dots,w_n)\in W^k$ then
$\{w_{n-k+1}, w_{n-k+2}, \dots,w_n\}$ is a subset of size $k$, which
gives rise to a partition, as described at the beginning of
Section \ref{sec:Schub}.  We refer to this partition as $\lambda_w$.

Moreover, it follows from \cite{Ste} and \cite{Pro}
that the reduced expressions of $w\in W^k$ correspond to
certain \emph{reading orders} of the boxes of the partition $\lambda_w$.
Specifically, let
$Q^k$ be the poset whose elements are the boxes of a
$k \times (n-k)$ rectangle;
if $b_1$ and $b_2$
are two adjacent boxes such that $b_2$ is immediately to the left or
immediately above $b_1$, we have a cover relation $b_1 \lessdot b_2$
in $Q^k$. The partial order on $Q^k$ is the transitive closure of
$\lessdot$.
Now label the boxes of the rectangle with simple generators
$s_i$ as in Figure \ref{fig:readingorder}.  If $b$ is a box of the rectangle,
then let $s_b$ denote its label by a simple generator.
Let $w_0^k \in W^k$ denote the longest element in
$W^k$.  Then the set of
reduced expressions of $w_0^k$ can be
obtained by choosing a
linear extension of $Q^k$ and writing down the corresponding
word in the $s_i$'s.  We call such a linear extension a
\emph{reading order}; two linear extensions are shown in Figure
\ref{fig:readingorder}.
Additionally, given a partition $\lambda_w$
contained in the $k \times (n-k)$ rectangle (chosen so that
the upper-left corner of its Young diagram is aligned
with the upper-left corner of the rectangle), and a linear
extension of the sub-poset of $Q^k$ comprised of the boxes of $\lambda$, the corresponding
word in $s_i$'s is a reduced expression of
the permutation $w \in W^k$.  Moreover,
all reduced expressions of $w$ can be obtained by varying
the linear extension.

\begin{figure}
\setlength{\unitlength}{0.7mm}
\begin{center}
  \begin{picture}(60,30)

  \put(5,32){\line(1,0){45}}
  \put(5,23){\line(1,0){45}}
  \put(5,14){\line(1,0){45}}
  \put(5,5){\line(1,0){45}}
  \put(5,5){\line(0,1){27}}
  \put(14,5){\line(0,1){27}}
  \put(23,5){\line(0,1){27}}
  \put(32,5){\line(0,1){27}}
  \put(41,5){\line(0,1){27}}
  \put(50,5){\line(0,1){27}}

  \put(7,27){$s_5$}
  \put(16,27){$s_4$}
  \put(25,27){$s_3$}
  \put(34,27){$s_2$}
  \put(43,27){$s_1$}
  \put(7,18){$s_6$}
  \put(16,18){$s_5$}
  \put(25,18){$s_4$}
  \put(34,18){$s_3$}
  \put(43,18){$s_2$}
  \put(7,9){$s_7$}
  \put(16,9){$s_6$}
  \put(25,9){$s_5$}
  \put(34,9){$s_4$}
  \put(43,9){$s_3$}

  \end{picture}
  \quad
  \begin{picture}(60,40)

  \put(5,32){\line(1,0){45}}
  \put(5,23){\line(1,0){45}}
  \put(5,14){\line(1,0){45}}
  \put(5,5){\line(1,0){45}}
  \put(5,5){\line(0,1){27}}
  \put(14,5){\line(0,1){27}}
  \put(23,5){\line(0,1){27}}
  \put(32,5){\line(0,1){27}}
  \put(41,5){\line(0,1){27}}
  \put(50,5){\line(0,1){27}}

  \put(7,26){$15$}
  \put(16,26){$14$}
  \put(25,26){$13$}
  \put(34,26){$12$}
  \put(43,26){$11$}
  \put(7,17){$10$}
  \put(16,17){$~9$}
  \put(25,17){$~8$}
  \put(34,17){$~7$}
  \put(43,17){$~6$}
  \put(7,8){$~5$}
  \put(16,8){$~4$}
  \put(25,8){$~3$}
  \put(34,8){$~2$}
  \put(43,8){$~1$}

 \end{picture}
  \quad
  \begin{picture}(60,40)

  \put(5,32){\line(1,0){45}}
  \put(5,23){\line(1,0){45}}
  \put(5,14){\line(1,0){45}}
  \put(5,5){\line(1,0){45}}
  \put(5,5){\line(0,1){27}}
  \put(14,5){\line(0,1){27}}
  \put(23,5){\line(0,1){27}}
  \put(32,5){\line(0,1){27}}
  \put(41,5){\line(0,1){27}}
  \put(50,5){\line(0,1){27}}

  \put(7,26){$15$}
  \put(16,26){$12$}
  \put(25,26){$~9$}
  \put(34,26){$~6$}
  \put(43,26){$~3$}
  \put(7,17){$14$}
  \put(16,17){$11$}
  \put(25,17){$~8$}
  \put(34,17){$~5$}
  \put(43,17){$~2$}
  \put(7,8){$13$}
  \put(16,8){$10$}
  \put(25,8){$~7$}
  \put(34,8){$~4$}
  \put(43,8){$~1$}

  \end{picture}
\end{center}
\caption{The labeling of a the boxes of a partition by simple generators $s_i$,
and two reading orders.}
\label{fig:readingorder}
\end{figure}

\begin{definition}\cite[Section 4]{KW2}
Fix $k$ and $n$.  Let $w\in W^k$, let $\w$ be a reduced expression for
$w$, and let $\v$ be a distinguished subexpression of $\w$.
Then $w$ and $\w$ determine a partition $\lambda_w$ contained in
a $k\times (n-k)$ rectangle together with a reading order of
its boxes.  The \emph{Go-diagram} associated to $\v$ and $\w$
is a filling of $\lambda_w$ with pluses and black and white stones,
such that:
for each $k \in J_{\v}^{\circ}$ we place a white stone
 in the corresponding box;
for each
$k \in J_{\v}^{\bullet}$ we place a black stone
 in the corresponding box of $\lambda_w$;
and for each $k \in J_{\v}^{+}$ we place a plus in the corresponding box
of $\lambda_w$.
\end{definition}

\begin{remark}\label{rem:independence}
By \cite[Section 4]{KW2},
whether or not a filling of a partition $\lambda_w$ is a Go-diagram
does not depend on the choice of reading order of the boxes of $\lambda_w$.
\end{remark}

\begin{definition}
We define the \emph{standard reading order} of the boxes of a partition to be the reading order
which starts at the rightmost box in the bottom row, then reads right to left
across the bottom row, then right to left across the row above that,
then right to left across the row above that, etc.  This reading order
is illustrated at the right of the figure below.
\end{definition}

By default,  we will use the standard reading order in this paper.

\begin{example}\label{ex4-1}
Let $k=3$ and $n=7$, and let $\lambda=(4,3,1)$.
The standard reading order is shown at the right
of the figure below.
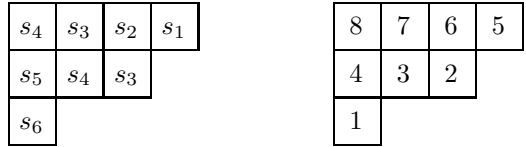
\setlength{\unitlength}{0.7mm}
\begin{center}
    \begin{picture}(50,35)
  \put(5,32){\line(1,0){36}}
  \put(5,23){\line(1,0){36}}
  \put(5,14){\line(1,0){27}}
  \put(5,5){\line(1,0){9}}
  \put(5,5){\line(0,1){27}}
  \put(14,5){\line(0,1){27}}
  \put(23,14){\line(0,1){18}}
  \put(32,14){\line(0,1){18}}
  \put(41,23){\line(0,1){9}}
  \put(7,26){$s_4$}
  \put(16,26){$s_3$}
  \put(25,26){$s_2$}
  \put(34,26){$s_1$}
  \put(7,17){$s_5$}
  \put(16,17){$s_4$}
  \put(25,17){$s_3$}
  \put(7,8){$s_6$}
  \end{picture}
  \qquad
	\begin{picture}(50,35)
  \put(5,32){\line(1,0){36}}
  \put(5,23){\line(1,0){36}}
  \put(5,14){\line(1,0){27}}
  \put(5,5){\line(1,0){9}}
  \put(5,5){\line(0,1){27}}
  \put(14,5){\line(0,1){27}}
  \put(23,14){\line(0,1){18}}
  \put(32,14){\line(0,1){18}}
  \put(41,23){\line(0,1){9}}
  \put(8,26){$8$}
  \put(17,26){$7$}
  \put(26,26){$6$}
  \put(35,26){$5$}
  \put(8,17){$4$}
  \put(17,17){$3$}
  \put(26,17){$2$}
  \put(8,8){$1$}
  \end{picture}
\end{center}

Then the following diagrams are Go-diagrams of shape $\lambda$.
\setlength{\unitlength}{0.7mm}
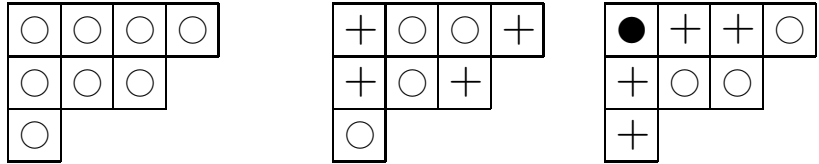
\begin{center}
    \begin{picture}(50,40)

  \put(5,35){\line(1,0){40}}
  \put(5,25){\line(1,0){40}}
  \put(5,15){\line(1,0){30}}
  \put(5,5){\line(1,0){10}}
  \put(5,5){\line(0,1){30}}
  \put(15,5){\line(0,1){30}}
  \put(25,15){\line(0,1){20}}
  \put(35,15){\line(0,1){20}}
  \put(45,25){\line(0,1){10}}

  \put(8,30){\hskip0.15cm\circle{5}}
  \put(18,30){\hskip0.15cm\circle{5}}
  \put(28,30){\hskip0.15cm\circle{5}}
  \put(38,30){\hskip0.15cm\circle{5}}
  \put(8,20){\hskip0.15cm\circle{5}}
  \put(18,20){\hskip0.15cm\circle{5}}
  \put(28,20){\hskip0.15cm\circle{5}}
  \put(8,10){\hskip0.15cm\circle{5}}

  \end{picture}
\qquad
     \begin{picture}(45,30)

  \put(5,35){\line(1,0){40}}
  \put(5,25){\line(1,0){40}}
  \put(5,15){\line(1,0){30}}
  \put(5,5){\line(1,0){10}}
  \put(5,5){\line(0,1){30}}
  \put(15,5){\line(0,1){30}}
  \put(25,15){\line(0,1){20}}
  \put(35,15){\line(0,1){20}}
  \put(45,25){\line(0,1){10}}

  \put(7,28){\scalebox{1.7}{$+$}}
  \put(18,30){\hskip0.15cm\circle{5}}
  \put(28,30){\hskip0.15cm\circle{5}}
  \put(37,28){\scalebox{1.7}{$+$}}
  \put(7,18){\scalebox{1.7}{$+$}}
  \put(18,20){\hskip0.15cm\circle{5}}
  \put(27,18){\scalebox{1.7}{$+$}}
  \put(8,10){\hskip0.15cm\circle{5}}

  \end{picture}
\quad
    \begin{picture}(45,30)

  \put(5,35){\line(1,0){40}}
  \put(5,25){\line(1,0){40}}
  \put(5,15){\line(1,0){30}}
  \put(5,5){\line(1,0){10}}
  \put(5,5){\line(0,1){30}}
  \put(15,5){\line(0,1){30}}
  \put(25,15){\line(0,1){20}}
  \put(35,15){\line(0,1){20}}
  \put(45,25){\line(0,1){10}}

  \put(8,30){\hskip0.15cm\circle*{5}}
  \put(17,28){\scalebox{1.7}{$+$}}
  \put(27,28){\scalebox{1.7}{$+$}}
  \put(38,30){\hskip0.15cm\circle{5}}
  \put(7,18){\scalebox{1.7}{$+$}}
  \put(18,20){\hskip0.15cm\circle{5}}
  \put(28,20){\hskip0.15cm\circle{5}}
  \put(7,8){\scalebox{1.7}{$+$}}

  \end{picture}
\end{center}
They correspond to the expressions
$s_6 s_3 s_4 s_5 s_1 s_2 s_3 s_4$, $s_6 1 s_4 1 1 s_2 s_3 1$, and
$1 s_3 s_4 1 s_1 1 1 s_4$.
The first and second are
positive distinguished subexpressions (PDS's), and the third one is a distinguished subexpression (but not a PDS).
\end{example}

Note that the following diagram of shape $\lambda$ is not a Go-diagram. It corresponds to the word $1 1 s_4 1 s_1 s_2 1 1$, which is not distinguished.
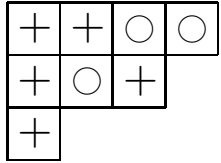
\setlength{\unitlength}{0.7mm}
\begin{center}
    \begin{picture}(50,40)

  \put(5,35){\line(1,0){40}}
  \put(5,25){\line(1,0){40}}
  \put(5,15){\line(1,0){30}}
  \put(5,5){\line(1,0){10}}
  \put(5,5){\line(0,1){30}}
  \put(15,5){\line(0,1){30}}
  \put(25,15){\line(0,1){20}}
  \put(35,15){\line(0,1){20}}
  \put(45,25){\line(0,1){10}}

  \put(7,28){\scalebox{1.7}{$+$}}
  \put(17,28){\scalebox{1.7}{$+$}}
  \put(28,30){\hskip0.15cm\circle{5}}
  \put(38,30){\hskip0.15cm\circle{5}}
  \put(7,18){\scalebox{1.7}{$+$}}
  \put(18,20){\hskip0.15cm\circle{5}}
  \put(27,18){\scalebox{1.7}{$+$}}
  \put(7,8){\scalebox{1.7}{$+$}}

  \end{picture}
\end{center}

\begin{remark}
The Go-diagrams associated to PDS's are in bijection with
$\Le$-diagrams, see \cite[Section 4]{KW2}.
Note that the Go-diagram associated to a PDS contains
only pluses and white stones.  This is precisely a
$\Le$-diagram.
\end{remark}

If we choose a reading order of $\lambda_w$, then we will also associate
to a Go-diagram of shape $\lambda_w$ a \emph{labeled Go-diagram},
as defined below.  Equivalently,
a labeled Go-diagram is associated to a pair $(\v,\w)$.

\begin{definition}\label{def:pi}\cite[Definition 4.15]{KW2}
Given a reading order of $\lambda_w$ and a Go-diagram of shape $\lambda_w$,
we obtain a \emph{labeled Go-diagram} by replacing each \raisebox{0.12cm}{\hskip0.14cm\circle{4}\hskip-0.15cm} with a $1$,
replacing each box $b$ containing a \raisebox{0.12cm}{\hskip0.14cm\circle*{4}\hskip-0.15cm} with a $-1$
\emph{and} an $m_i$, and replacing each box $b$ containing a $+$ by a $p_i$,
where the subscript $i$ corresponds to the label of $b$ inherited from the
reading order.
\end{definition}

The labeled Go-diagrams corresponding to
the examples above using the standard reading order are:

\setlength{\unitlength}{0.8mm}
\begin{center}
    \begin{picture}(50,40)

  \put(5,35){\line(1,0){40}}
  \put(5,25){\line(1,0){40}}
  \put(5,15){\line(1,0){30}}
  \put(5,5){\line(1,0){10}}
  \put(5,5){\line(0,1){30}}
  \put(15,5){\line(0,1){30}}
  \put(25,15){\line(0,1){20}}
  \put(35,15){\line(0,1){20}}
  \put(45,25){\line(0,1){10}}

  \put(8,29){1}
  \put(18,29){1}
  \put(28,29){1}
  \put(38,29){1}
  \put(8,19){1}
  \put(18,19){1}
  \put(28,19){1}
  \put(8,9){1}

  \end{picture}
\qquad
     \begin{picture}(45,30)

  \put(5,35){\line(1,0){40}}
  \put(5,25){\line(1,0){40}}
  \put(5,15){\line(1,0){30}}
  \put(5,5){\line(1,0){10}}
  \put(5,5){\line(0,1){30}}
  \put(15,5){\line(0,1){30}}
  \put(25,15){\line(0,1){20}}
  \put(35,15){\line(0,1){20}}
  \put(45,25){\line(0,1){10}}

  \put(8,30){$p_8$}
  \put(18,29){1}
  \put(28,29){1}
  \put(38,30){$p_5$}
  \put(8,20){$p_4$}
  \put(18,19){1}
  \put(28,20){$p_2$}
  \put(8,9){1}

  \end{picture}
\quad
    \begin{picture}(45,30)

  \put(5,35){\line(1,0){40}}
  \put(5,25){\line(1,0){40}}
  \put(5,15){\line(1,0){30}}
  \put(5,5){\line(1,0){10}}
  \put(5,5){\line(0,1){30}}
  \put(15,5){\line(0,1){30}}
  \put(25,15){\line(0,1){20}}
  \put(35,15){\line(0,1){20}}
  \put(45,25){\line(0,1){10}}
  \put(5,25){\line(1,1){10}}

  \put(5,31){$-1$}
  \put(8.5,27){$m_8$}
  \put(18,30){$p_7$}
  \put(28,30){$p_6$}
  \put(38,29){1}
  \put(8,20){$p_4$}
  \put(18,19){1}
  \put(28,19){1}
  \put(8,10){$p_1$}

  \end{picture}
\end{center}

\subsection{The main result}

To state the main result, we now consider Go-diagrams (not arbitrary
diagrams), the corresponding networks (\emph{Go-networks}),
and the corresponding weight matrices.

\begin{definition}
Let $D$ be a Go-diagram contained in a $k \times (n-k)$ rectangle.
We define a subset $\mathcal{R}_D$ of the  Grassmannian
$Gr_{k,n}$ by letting each variable $a_i$ of the weight matrix
(Definition \ref{def:weightmatrix})
range over all nonzero elements $\K^*$, and letting each variable
$c_i$ of the weight matrix range over all elements $\K$.
We call $\mathcal{R}_D$ the \emph{network component associated to $D$}.
\end{definition}

\begin{theorem}\label{th:main}
Let $D$ be a Go-diagram contained in a
$k \times (n-k)$ rectangle.
Suppose that $D$ has $t$ pluses and $u$ black stones.
Then $\mathcal{R}_D$ is isomorphic to
$(\K^*)^t \times \K^u$.
Furthermore, $Gr_{k,n}$ is the disjoint union of the network components
$\mathcal{R}_D$, as $D$ ranges over all Go-diagrams contained in
a $k \times (n-k)$ rectangle.  In other words, each point in the
Grassmannian $Gr_{k,n}$ can be represented uniquely by a weighted
network associated to a Go-diagram.
\end{theorem}

A more refined version of Theorem \ref{th:main} is given in
Corollary \ref{cor:2params}.

\begin{corollary}
Every $k \times \ell$ matrix $M$ can be represented by a unique weighted network associated to a Go-diagram contained in a $k \times \ell$ rectangle.
\end{corollary}
\begin{proof}
Let $n=k+\ell$.  Suppose $M=(m_{i,j})$ is a $k\times (n-k)$ matrix.
Let $A(M)=(a_{i,j})$ be the full rank $k\times n$ matrix with an identity submatrix in the first $k$ columns and the remaining columns given by $a_{i,j+k}=(-1)^{i+1}m_{n+1-i,j}$. Then $A(M)$ represents an element in the Grassmannian
$Gr_{k,n}$, so Theorem \ref{th:main} applies.  Moreoever, the minors of $M$ are in bijection with the $k\times k$ minors of $A(M)$, so if $A(M)$ is represented by the network $N$, we see that $m_{i,j}$ enumerates paths from the boundary source $i$ to the boundary vertex $j$ in $N$.
\end{proof}

We will prove Theorem \ref{th:main} by showing that each network
component $\mathcal{R}_D$ from a Go-diagram coincides with a \emph{(projected)
Deodhar component} $\mathcal{P}_{\v,\w}$
in the Grassmannian. (Therefore we may refer to
each $\mathcal{R}_D$ as a \emph{Deodhar component}.)
More specifically, such Deodhar components
have parameterizations due to Marsh and Rietsch \cite{MR}, and we will show
that after an invertible transformation of variables,
our network parameterizations coincide with theirs.


\section{The Deodhar decomposition of the Grassmannian}
\label{sec:project}

In this section we review Deodhar's decomposition of the
flag variety  $G/B$ \cite{Deodhar}, and the parameterizations of the components due to Marsh
and Rietsch \cite{MR}.
The Deodhar decomposition of the Grassmannian is obtained by
projecting the Deodhar decomposition of $G/B$ to the Grassmannian
\cite{Deodhar2}.

\subsection{The flag variety}\label{sec:flag}

Let $\K$ be a field, and
let $G$ denote the special linear group $\SL_n = \SL_n(\K)$.
Fix a maximal torus $T$, and opposite Borel subgroups $B^+$ and $B^-$, which
consist of the diagonal, upper-triangular, and lower-triangular
matrices, respectively.
Let $U^+$ and $U^-$ be the unipotent radicals of $B^+$ and $B^-$; these
are the subgroups of upper-triangular and lower-triangular matrices with $1$'s on the diagonals.
For each $1 \leq i \leq n-1$ we have a homomorphism
$\phi_i:{\rm SL}_2\to {\rm SL}_n$ such that
\[
\phi_i\begin{pmatrix} a& b\\c&d\end{pmatrix}=
\begin{pmatrix}
1 &             &       &       &           &     \\
   &\ddots  &        &      &            &        \\
   &             &   a   &   b  &          &       \\
   &             &   c    &   d  &         &       \\
   &            &          &       &  \ddots  &     \\
   &            &          &       &              & 1
   \end{pmatrix} ~\in {\rm SL}_n,
\]
that is, $\phi_i$ replaces a $2\times 2$ block of the identity matrix with $\begin{pmatrix} a&b\\c&d\end{pmatrix}$.
Here $a$ is at the $(i+1)$st diagonal entry counting from the southeast corner.
(Correspondingly, we will label the rows of such a matrix from bottom
to top, and the columns of such a matrix from right to left.)
We have $1$-parameter subgroups of $G$
defined by
\begin{equation*}
x_i(m) = \phi_i \left(
                   \begin{array}{cc}
                     1 & m \\ 0 & 1\\
                   \end{array} \right)  \text{ and }\ 
y_i(m) = \phi_i \left(
                   \begin{array}{cc}
                     1 & 0 \\ m & 1\\
                   \end{array} \right) ,\
\text{ where }m \in \K.
\end{equation*}

Let $W$ denote the Weyl group $ N_G(T) / T$,
where $N_G(T)$ is the normalizer of $T$.
The simple reflections $s_i \in W$ are given by
$s_i:= \dot{s_i} T$ where $\dot{s_i} :=
                 \phi_i \left(
                   \begin{array}{cc}
                     0 & -1 \\ 1 & 0\\
                   \end{array} \right)$
and any $w \in W$ can be expressed as a product $w = s_{i_1} s_{i_2}
\dots s_{i_\ell}$ with $\ell=\ell(w)$ factors.  We set $\dot{w} =
\dot{s}_{i_1} \dot{s}_{i_2} \dots \dot{s}_{i_\ell}$.
In our setting $W$ is isomorphic to
$\Sym_n$, the symmetric group on $n$ letters,
and $s_i$ corresponds to the transposition exchanging $i$ and $i+1$.

We can identify the flag variety $G/B$ with the variety
$\B$ of Borel subgroups, via
\begin{equation*}
gB \longleftrightarrow g \cdot B^+ := gB^+ g^{-1}.
\end{equation*}
We have two opposite Bruhat
decompositions of $\B$:
\begin{equation*}
\mathcal B=\bigsqcup_{w\in W} B^+\dot w\cdot B^+=\bigsqcup_{v\in W}
B^-\dot v\cdot B^+.
\end{equation*}
We define the intersection of opposite Bruhat cells
\begin{equation*}
\mathcal R_{v,w}:=B^+\dot w\cdot B^+\cap B^-\dot v\cdot B^+,
\end{equation*}
which is nonempty
precisely when $v\le w$.
The strata $\mathcal R_{v,w}$ are often called \emph{Richardson varieties}.

\subsection{Deodhar components in the flag variety}

We now describe the Deodhar decomposition of the  flag variety.
Marsh and Rietsch \cite{MR} gave explicit parameterizations for each Deodhar
component, identifying each one with a subset in the group.

\begin{definition}\cite[Definition 5.1]{MR}\label{d:factorization}
Let $\w=s_{i_1} \dots s_{i_m}$ be a reduced expression for $w$,
and let $\v$ be a distinguished subexpression.
Define a subset $G_{\v,\w}$ in $G$ by
\begin{equation}\label{e:Gvw}
G_{\v,\w}:=\left\{g= g_1 g_2\cdots g_m \left
|\begin{array}{ll}
 g_\ell= x_{i_\ell}(m_\ell)\dot s_{i_\ell}\inv& \text{ if $\ell\in J^{\bullet}_\v$,}\\
 g_\ell= y_{i_\ell}(p_\ell)& \text{ if $\ell\in J^{+}_\v$,}\\
 g_\ell=\dot s_{i_\ell}& \text{ if $\ell\in J^{\circ}_\v$,}
 \end{array}\quad \text{
for $p_\ell\in\K^*,\, m_\ell\in\K$. }\right. \right\}.
\end{equation}
There is an obvious map $(\K^*)^{|J^{+}_\v|}\times \K^{|J^{\bullet}_\v|}\to
G_{\v,\w}$ defined by the parameters $p_\ell$ and $m_\ell$ in
\eqref{e:Gvw}.
For $v =w=1$
we define $G_{\v,\w}=\{1\}$.
\end{definition}

\begin{example}\label{ex:g}
Let $W=\Sym_8$, $\w = s_6 s_7 s_4 s_5 s_6 s_3 s_4 s_5 s_1 s_2 s_3 s_4$ and
$\v = s_6 1 s_4 1 s_6 1 s_4 1 1 1 1 1$. This is the distinguished
expression $\v$ encoded by the diagram from Figure
\ref{fig:network} (which is a Go-diagram).
Then the corresponding element $g\in G_{\v,\w}$ (the MR-matrix)
is given by
\begin{equation} \label{eq:MR-product}
g=\dot s_6 y_7(p_2)\dot s_4 y_5(p_4) x_6(m_5) \dot s_6^{-1} y_3(p_6) x_4(m_7) \dot s_4^{-1} y_5(p_8) y_1(p_9) y_2(p_{10}) y_3(p_{11}) y_4(p_{12}),
\end{equation}
which is
\[
\left(
  \begin{array}{cccccccc}
    1 & 0 & 0 & 0 & 0 & 0 & 0 & 0 \\
    0 & 1 & 0 & 0 & 0 & 0 & 0 & 0 \\
    p_2 & -m_5 & 1 & 0 & 0 & 0 & 0 & 0 \\
    0 & 0 & p_8 & 1 & 0 & 0 & 0 & 0 \\
    0 & -p_4 & -m_7p_8 & -m_7+p_{12} & 1 & 0 & 0 & 0 \\
    0 & 0 & -p_6p_8 & -p_6+p_{11}p_{12} & p_{11} & 1 & 0 & 0 \\
    0 & 0 & 0 & p_{10}p_{11}p_{12} & p_{10}p_{11} & p_{10} & 1 & 0 \\
    0 & 0 & 0 & p_9p_{10}p_{11}p_{12} & p_9p_{10}p_{11} & p_9p_{10} & p_9 & 1 \\
  \end{array}
\right).
\]
\end{example}

The following result from \cite{MR} gives an explicit parametrization for
the Deodhar component $\mathcal R_{\v,\w}$.
We use Proposition \ref{p:parameterization} as the
\emph{definition}
of $\mathcal R_{\v,\w}$.

\begin{proposition}\label{p:parameterization}\cite[Proposition 5.2]{MR}
The map $(\K^*)^{|J^{+}_\v|}\times \K^{|J^{\bullet}_\v|}\to G_{\v,\w}$ from
Definition~\ref{d:factorization} is an isomorphism.
The map $g\mapsto g\cdot B^+$ defines an isomorphism
\begin{align}\label{e:parameterization}
G_{\v,\w}&~\overset\sim\To ~\mathcal R_{\v,\w}
\end{align}
between the subset $G_{\v,\w}$ of the group,
and the Deodhar component $\mathcal R_{\v,\w}$
 in $G/B$.
\end{proposition}

Suppose that for each $w\in W$ we choose a reduced expression
$\w$ for $w$.  Then it follows from \cite{Deodhar} and
\cite[Section 4.4]{MR} that
\begin{equation}\label{e:DeoDecomp}
\mathcal R_{v,w} = \bigsqcup_{\v \prec \w} \mathcal R_{\v,\w}\qquad  \text{ and }
\qquad  G/B=\bigsqcup_{w\in W}\left(\bigsqcup_{\v\prec \w} \mathcal
 R_{\v,\w}\right),
\end{equation}
where in the first sum $\v$ ranges over all distinguished subexpressions
for $v$ in $w$, and in the second sum $\v$ ranges over all distinguished
subexpressions of $\w$.
These two decompositions
are called the \emph{Deodhar decompositions} of $\mathcal R_{v,w}$ and  $G/B$.

\begin{remark}\label{rem:dependence}
Although the Deodhar decomposition of $\mathcal R_{v,w}$
depends on the choice $\w$ of
reduced expression for $w$,
its projection to the Grassmannian does not depend on $\w$ \cite[Proposition 4.16]{KW2}.
\end{remark}

\subsection{Projections of Deodhar components to the Grassmannian}\label{sec:projections}

Following \cite{KW2}, we now consider the projection of the Deodhar decomposition to
the Grassmannian
$Gr_{k,n}$ for $k<n$.  Given the permutation
$w = (w(1), w(2),\dots, w(n))\in W^k$,
we let $I(w)$ denote the $k$-element subset $\{w(n-k+1), w(n-k+2),\dots,w(n)\}$ of $[n]$.
The map $I$ gives a bijection between $W^k$ and $k$-element subsets of $[n]$.

Let $\pi_k: G/B \to Gr_{k,n}$ be the projection from the flag variety to the Grassmannian; this is an isomorphism on each
$\mathcal R_{v,w}$.
For each $w \in W^k$ and $v \leq w$, define
$\mathcal P_{v,w} = \pi_k(\mathcal R_{v,w})$.  Then by
\cite{Lusztig2}
we have a decomposition
\begin{equation}\label{projected-Richardson}
Gr_{k,n} = \bigsqcup_{w \in W^k} \left(\bigsqcup_{v \leq w} \mathcal P_{v,w} \right).
\end{equation}

For each reduced decomposition $\w$ for $w \in W^k$, and each $\v \prec \w$,
we define $\mathcal P_{\v,\w} = \pi_k(\mathcal R_{\v,\w})$.
Now if for each $w \in W^k$ we choose a reduced decomposition $\w$, then we have
\begin{equation}\label{e:ProjDeoDecomp}
\mathcal P_{v,w} = \bigsqcup_{\v \prec \w} \mathcal P_{\v,\w}\qquad \text{ and }\qquad
 Gr_{k,n} =\bigsqcup_{w\in W^k} \left(\bigsqcup_{\v\prec \w} \mathcal
 P_{\v,\w}\right),
\end{equation}
where in the first sum $\v$ ranges over all distinguished subexpressions
for $v$ in $w$, and in the second sum $\v$ ranges over all distinguished
subexpressions of $\w$.

Proposition \ref{p:parameterization} gives  a concrete way to
think about the projected Deodhar components $\mathcal P_{\v,\w}$.
The projection $\pi_k: G/B \to Gr_{k,n}$ maps
$g \cdot B^+ \in R_{\v,\w}$ (where $g\in G_{\v,\w}$) to the span of
the leftmost $k$ columns of $g$.
More specifically, it maps
\[
g=\begin{pmatrix}
g_{n,n} & \dots& g_{n,n-k+1} & \dots & g_{n,1}  \\
\vdots & &\vdots&  & \vdots  \\
g_{1,n} & \dots& g_{1,n-k+1} & \dots & g_{1,1} \\
\end{pmatrix}
\quad \longrightarrow\quad
M=
\begin{pmatrix}
g_{1,n-k+1}&  \dots & g_{n,n-k+1}\\
\vdots &  & \vdots\\
g_{1,n} &  \dots & g_{n,n} \\
\end{pmatrix}
\]
We call the resulting $k \times n$ matrix $M=(M_{st})$ the
\emph{MR-matrix}.  To simplify the notation  later,
we will label its rows from top to bottom by
$i_1, i_2, \dots, i_k$, where $\{i_1 < \dots < i_k \} = I(w)$.

\begin{remark}
Recall from Section \ref{Deodhar-combinatorics}
that in the Grassmannian setting
(i.e.
$W_k = \langle s_1,s_2,\dots,\hat{s}_{n-k},\dots,s_{n-1} \rangle$
is a parabolic subgroup of $W = \Sym_n$), the
distinguished subexpressions of $W^k$ are in bijection with
Go-diagrams.  Therefore each Go-diagram gives rise to an MR-matrix.
\end{remark}

\begin{example}\label{ex:mr-matrix}
We continue Example \ref{ex:g}.  Note that
$w \in W^k$, where $k=2$.  Then the map\\ $\pi_2: G_{\v,\w}\to Gr_{2,5}$
is given by
\[
g=\left(
  \begin{array}{cccc|cccc}
    1 & 0 & 0 & 0 & 0 & 0 & 0 & 0 \\
    0 & 1 & 0 & 0 & 0 & 0 & 0 & 0 \\
    p_2 & -m_5 & 1 & 0 & 0 & 0 & 0 & 0 \\
    0 & 0 & p_8 & 1 & 0 & 0 & 0 & 0 \\
    0 & -p_4 & -m_7p_8 & -m_7+p_{12} & 1 & 0 & 0 & 0 \\
    0 & 0 & -p_6p_8 & -p_6+p_{11}p_{12} & p_{11} & 1 & 0 & 0 \\
    0 & 0 & 0 & p_{10}p_{11}p_{12} & p_{10}p_{11} & p_{10} & 1 & 0 \\
    0 & 0 & 0 & p_9p_{10}p_{11}p_{12} & p_9p_{10}p_{11} & p_9p_{10} & p_9 & 1 \\
  \end{array}
\right)
\quad\longrightarrow \quad
\]
\[ M=
\left(
\begin{array}{cccccccc}
  p_9p_{10}p_{11}p_{12} & p_{10}p_{11}p_{12} & -p_6+p_{11}p_{12} & -m_7+p_{12} & 1 & 0 & 0 & 0 \\
  0 & 0 & -p_6p_8 & -m_7p_8 & p_8 & 1 & 0 & 0 \\
  0 & 0 & 0 & -p_4 & 0 & -m_5 & 1 & 0 \\
  0 & 0 & 0 & 0 & 0 & p_2 & 0 & 1 \\
\end{array}\right)
\]
We label the rows of $M$ from top to bottom by
the index set $\{1, 3, 4, 6\}$, and the columns from
left to right by the index set $\{1,2, \dots, 8\}$, so e.g.
$M_{34} = -m_7 p_8$.
\end{example}

The following lemma is a consequence of \cite[Section 5.1]{KW2} and
in particular \cite[Corollary 5.8]{KW2}.

\begin{lemma}\label{lem:Mproperties}
Let $M = M_D$ be the MR-matrix associated to the diagram $D$.
The leftmost nonzero entry in row $i_{\ell}$ of $M$ is in column $i_{\ell}$.
Furthermore, that entry is equal to $(-1)^b \prod p_i$,
where $b$ is the number of black stones in the row $i_\ell$ of $D$,
and the product is over all boxes in the row $i_\ell$
of the labeled Go-diagram of $D$
containing a $p_i$.
\end{lemma}


\section{Formulas for entries of the MR-matrices}
\label{sec:MR-matrix}

In this section we consider arbitrary diagrams (not necessarily
Go-diagrams) contained in a $k \times (n-k)$ rectangle
and the corresponding MR-matrices,
obtained by multiplying factors
$\dot s_{i}$, $y_i(p_j)$, $x_i(m_j) \dot s_i^{-1}$
as specified by the filling of the diagram, and then projecting
the resulting $n \times n$ matrix to a $k \times n$ matrix.
We will give formulas for the entries of the MR-matrices
in terms of \emph{pseudopaths} in the corresponding network.
For the purpose of giving this formula,
we will replace weights $a_i$ and $c_j$ on the edges of the network
by weights $p_i$ and $m_j$.

Recall that if $D$ is a diagram, its network $N_D$ has three types of vertices:
$+$-vertices, $\smallblack$-vertices, and
boundary vertices.
A \emph{step} on a network is an edge
between two vertices.
Let $W$ denote a single step west,
$S$ denote a single step south, and
$E$ denote either a single step east, or
an \emph{east-west combination step} consisting
of a step east, followed by a step west ending at a \smallblack.
 Let $A^*$ indicate
$0$ or more instances of a step of type $A$.

\begin{definition}\label{def:pseudopath}
A \emph{pseudopath} $\widetilde{P}$ on a network is a path on the
(undirected version of the) network such that:
\begin{itemize}
\item it starts and ends at two different boundary vertices,
or else is the empty path from a boundary vertex to itself;
\item it does not cross the same edge twice;
\item its sequence of steps (for a nonempty path) has the form
$$WW^*S(EE^*S)^*E^*.$$
\end{itemize}
\end{definition}
In particular, a pseudopath may not take two consecutive
steps south.

\begin{figure}
\label{fig:double-east}
\end{figure}

\begin{definition}\label{def:pseudoweight}
The \emph{weight} $w(\widetilde{P})$ of a pseudopath $\widetilde{P}$
in a network is a Laurent monomial
in $p_i$'s and $m_j$'s, which is obtained by multiplying the following terms:
\begin{itemize}
\item $\frac{1}{p_i}$ for every step west along an edge weighted $p_i$;
\item $p_i$ for every step east along an edge weighted $p_i$ which is
preceded by a step east;
\item $m_j$ for every step west along an edge weighted $m_j$;
\item $(-1)^{b+w}$, where $b$ (respectively $w$) is  the number of black (resp. white)
stones that the pseudopath skips over in the horizontal (resp. vertical) direction, when we superimpose the Go-diagram onto the network.
\end{itemize}
\end{definition}

\begin{example}\label{ex:pseudo}
In Figure \ref{fig:pseudopath}, there are
two pseudopaths from $1$ to $4$, with weights
$\frac{1}{p_9 p_{10} p_{11}}$ and $\frac{-m_7}{p_9 p_{10} p_{11} p_{12}}$, and there
is one pseudopath from $1$ to $5$, with weight
$\frac{1}{p_9 p_{10} p_{11} p_{12}}$.

\begin{figure}[ht]
\begin{center}
\psset{unit=.5cm,dotstyle=o,dotsize=5pt 0,linewidth=0.8pt,arrowsize=3pt 2,arrowinset=0.25}
\begin{pspicture*}(0,0)(10,9)
\psline{-}(5,2)(2,2) \uput[u](2.5,2){$p_2$}
\psline[linewidth=2.5pt]{-}(7,4)(4,4) \uput[u](4.5,4){$p_4$}
\psline{-}(4,4)(2,4) \uput[u](2.5,4){$m_5$}
\psline{-}(7,6)(6,6) \uput[u](6.5,6){$p_6$}
\psline{-}(6,6)(4,6) \uput[u](4.5,6){$m_7$}
\pscurve{-}(6,6)(4,7)(2,6) \uput[u](2.5,6.2){$p_8$}
\psline[linewidth=2.5pt]{-}(9,8)(8,8) \uput[u](8.5,8){$p_9$}
\psline[linewidth=2.5pt]{-}(8,8)(6,8) \uput[u](6.5,8){$p_{10}$}
\psline[linewidth=2.5pt]{-}(6,8)(4,8) \uput[u](4.5,8){$p_{11}$}
\psline{-}(4,8)(2,8) \uput[u](2.5,8){$p_{12}$}

\psline{-}(2,8)(2,6)
\pscurve{-}(2,6)(1.4,4)(2,2)
\psline{-}(2,4)(2,2)
\psline{-}(2,2)(2,1)
\pscurve[linewidth=2.5pt]{-}(4,8)(3.4,6)(4,4)
\psline{-}(4,6)(4,4)
\psline{-}(4,4)(4,1)
\psline{-}(6,8)(6,6)
\psline{-}(6,6)(6,3.02)
\psline{-}(8,8)(8,7)

\psdots[dotstyle=*,linecolor=black](2,8)
\psdots[dotstyle=*,linecolor=black](2,6)
\psdots[dotsize=9pt 0,dotstyle=*,linecolor=black](2,4)
\psdots[dotstyle=*,linecolor=black](2,2)
\psdots[dotstyle=*,linecolor=black](4,4)
\psdots[dotsize=9pt 0,dotstyle=*,linecolor=black](4,6)
\psdots[dotstyle=*,linecolor=black](4,8)
\psdots[dotstyle=*,linecolor=black](6,8)
\psdots[dotstyle=*,linecolor=black](8,8)
\psdots[dotstyle=*,linecolor=black](6,6)
\psdots[dotstyle=*,linecolor=black](9,8)

\uput[r](9.08,8.12){\large{1}}
\psdots[dotstyle=*,linecolor=black](8,7)
\uput[r](8.08,7.12){\large{2}}
\psdots[dotstyle=*,linecolor=black](7,6)
\uput[r](7.08,6.12){\large{3}}
\psdots[dotstyle=*,linecolor=black](7,4)
\uput[r](7.08,4.12){\large{4}}
\psdots[dotstyle=*,linecolor=black](6,3)
\uput[r](6.08,3.12){\large{5}}
\psdots[dotstyle=*,linecolor=black](5,2)
\uput[r](5.08,2.12){\large{6}}
\psdots[dotstyle=*,linecolor=black](4,1)
\uput[r](4.08,1.12){\large{7}}
\psdots[dotstyle=*,linecolor=black](2,1)
\uput[r](2.08,1.12){\large{8}}

\end{pspicture*}\quad\begin{pspicture*}(0,0)(10,9)
\psline{-}(5,2)(2,2) \uput[u](2.5,2){$p_2$}
\psline[linewidth=2.5pt]{-}(7,4)(4,4) \uput[u](4.5,4){$p_4$}
\psline{-}(4,4)(2,4) \uput[u](2.5,4){$m_5$}
\psline{-}(7,6)(6,6) \uput[u](6.5,6){$p_6$}
\psline[linewidth=2.5pt]{-}(6,6)(4,6) \uput[u](4.5,6){$m_7$}
\pscurve[linewidth=2.5pt]{-}(6,6)(4,7)(2,6) \uput[u](2.5,6.2){$p_8$}
\psline[linewidth=2.5pt]{-}(9,8)(8,8) \uput[u](8.5,8){$p_9$}
\psline[linewidth=2.5pt]{-}(8,8)(6,8) \uput[u](6.5,8){$p_{10}$}
\psline[linewidth=2.5pt]{-}(6,8)(4,8) \uput[u](4.5,8){$p_{11}$}
\psline[linewidth=2.5pt]{-}(4,8)(2,8) \uput[u](2.5,8){$p_{12}$}

\psline[linewidth=2.5pt]{-}(2,8)(2,6)
\pscurve{-}(2,6)(1.4,4)(2,2)
\psline{-}(2,4)(2,2)
\psline{-}(2,2)(2,1)
\pscurve{-}(4,8)(3.4,6)(4,4)
\psline[linewidth=2.5pt]{-}(4,6)(4,4)
\psline{-}(4,4)(4,1)
\psline{-}(6,8)(6,6)
\psline{-}(6,6)(6,3.02)
\psline{-}(8,8)(8,7)

\psdots[dotstyle=*,linecolor=black](2,8)
\psdots[dotstyle=*,linecolor=black](2,6)
\psdots[dotsize=9pt 0,dotstyle=*,linecolor=black](2,4)
\psdots[dotstyle=*,linecolor=black](2,2)
\psdots[dotstyle=*,linecolor=black](4,4)
\psdots[dotsize=9pt 0,dotstyle=*,linecolor=black](4,6)
\psdots[dotstyle=*,linecolor=black](4,8)
\psdots[dotstyle=*,linecolor=black](6,8)
\psdots[dotstyle=*,linecolor=black](8,8)
\psdots[dotstyle=*,linecolor=black](6,6)
\psdots[dotstyle=*,linecolor=black](9,8)

\uput[r](9.08,8.12){\large{1}}
\psdots[dotstyle=*,linecolor=black](8,7)
\uput[r](8.08,7.12){\large{2}}
\psdots[dotstyle=*,linecolor=black](7,6)
\uput[r](7.08,6.12){\large{3}}
\psdots[dotstyle=*,linecolor=black](7,4)
\uput[r](7.08,4.12){\large{4}}
\psdots[dotstyle=*,linecolor=black](6,3)
\uput[r](6.08,3.12){\large{5}}
\psdots[dotstyle=*,linecolor=black](5,2)
\uput[r](5.08,2.12){\large{6}}
\psdots[dotstyle=*,linecolor=black](4,1)
\uput[r](4.08,1.12){\large{7}}
\psdots[dotstyle=*,linecolor=black](2,1)
\uput[r](2.08,1.12){\large{8}}

\end{pspicture*}\quad\begin{pspicture*}(0,0)(10,9)
\psline{-}(5,2)(2,2) \uput[u](2.5,2){$p_2$}
\psline{-}(7,4)(4,4) \uput[u](4.5,4){$p_4$}
\psline{-}(4,4)(2,4) \uput[u](2.5,4){$m_5$}
\psline{-}(7,6)(6,6) \uput[u](6.5,6){$p_6$}
\psline{-}(6,6)(4,6) \uput[u](4.5,6){$m_7$}
\pscurve[linewidth=2.5pt]{-}(6,6)(4,7)(2,6) \uput[u](2.5,6.2){$p_8$}
\psline[linewidth=2.5pt]{-}(9,8)(8,8) \uput[u](8.5,8){$p_9$}
\psline[linewidth=2.5pt]{-}(8,8)(6,8) \uput[u](6.5,8){$p_{10}$}
\psline[linewidth=2.5pt]{-}(6,8)(4,8) \uput[u](4.5,8){$p_{11}$}
\psline[linewidth=2.5pt]{-}(4,8)(2,8) \uput[u](2.5,8){$p_{12}$}

\psline[linewidth=2.5pt]{-}(2,8)(2,6)
\pscurve{-}(2,6)(1.4,4)(2,2)
\psline{-}(2,4)(2,2)
\psline{-}(2,2)(2,1)
\pscurve{-}(4,8)(3.4,6)(4,4)
\psline{-}(4,6)(4,4)
\psline{-}(4,4)(4,1)
\psline{-}(6,8)(6,6)
\psline[linewidth=2.5pt]{-}(6,6)(6,3)
\psline{-}(8,8)(8,7)

\psdots[dotstyle=*,linecolor=black](2,8)
\psdots[dotstyle=*,linecolor=black](2,6)
\psdots[dotsize=9pt 0,dotstyle=*,linecolor=black](2,4)
\psdots[dotstyle=*,linecolor=black](2,2)
\psdots[dotstyle=*,linecolor=black](4,4)
\psdots[dotsize=9pt 0,dotstyle=*,linecolor=black](4,6)
\psdots[dotstyle=*,linecolor=black](4,8)
\psdots[dotstyle=*,linecolor=black](6,8)
\psdots[dotstyle=*,linecolor=black](8,8)
\psdots[dotstyle=*,linecolor=black](6,6)
\psdots[dotstyle=*,linecolor=black](9,8)

\uput[r](9.08,8.12){\large{1}}
\psdots[dotstyle=*,linecolor=black](8,7)
\uput[r](8.08,7.12){\large{2}}
\psdots[dotstyle=*,linecolor=black](7,6)
\uput[r](7.08,6.12){\large{3}}
\psdots[dotstyle=*,linecolor=black](7,4)
\uput[r](7.08,4.12){\large{4}}
\psdots[dotstyle=*,linecolor=black](6,3.02)
\uput[r](6.08,3.12){\large{5}}
\psdots[dotstyle=*,linecolor=black](5,2)
\uput[r](5.08,2.12){\large{6}}
\psdots[dotstyle=*,linecolor=black](4,1)
\uput[r](4.08,1.12){\large{7}}
\psdots[dotstyle=*,linecolor=black](2,1)
\uput[r](2.08,1.12){\large{8}}

\end{pspicture*}
\end{center}
\caption{The two pseudopaths from 1 to 4 and the unique pseudopath from 1 to 5, indicated in bold. Note that the pseudopath in the middle figure
contains an east-west combination step.}
\label{fig:pseudopath}
\end{figure}
\end{example}

\begin{definition}
If $M$ is an MR-matrix, we will let
$\widetilde{M}$ denote the matrix obtained from $M$
by rescaling rows so that the leftmost nonzero entry in each row is $1$.
\end{definition}

\begin{definition}
If $D$ is a diagram contained in a $k \times (n-k)$ rectangle,
then we let
$i_1 <  i_2 < \dots < i_k$ denote the labels of the sources
in the corresponding network.
If $M$ and $\widetilde{M}$ are the corresponding $k \times n$
MR and rescaled MR-matrices
associated to $D$, then we will
index their rows by $i_1, \dots, i_k$ from top to bottom, and their
columns by $1, 2, \dots, n$ from left to right.
\end{definition}

\begin{theorem}\label{th:MR-matrix}
Let $D$ be a diagram contained in a $k \times (n-k)$ rectangle,
and let $\widetilde{M} = (\widetilde{M}_{st})$
be the corresponding $k \times n$ rescaled MR-matrix.
Then $$\widetilde{M}_{st} = \sum_{\widetilde{P}} w(\widetilde{P}),$$
where the sum is over all pseudopaths from
the source $s$ to the boundary vertex $t$ in the network .
\end{theorem}

Theorem~\ref{th:MR-matrix} will follow from Theorem~\ref{th2:MR-matrix} and Lemma~\ref{lem:equiv}.

\begin{example}
The MR-matrix $M$ from Example~\ref{ex:mr-matrix} corresponds to the
network from Figure~\ref{fig:pseudopath}.  The rows of $M$ are indexed by
$1, 3, 4, 6$ from top to bottom.  Note that after we rescale the rows of
$M$, obtaining $\widetilde{M}$, we have
$\widetilde{M}_{14} =
\frac{1}{p_9 p_{10} p_{11}}-\frac{m_7}{p_9 p_{10} p_{11} p_{12}}$, and $\widetilde{M}_{15} =
\frac{1}{p_9 p_{10} p_{11} p_{12}}$.
This agrees with our pseudopath computation
from Example \ref{ex:pseudo}.
\end{example}

Next we will give a formula for entries of MR-matrices, in terms of
pseudopaths in \emph{modified networks}.

\begin{definition}
Given a network $N_D$ with $k$ sources labeled $i_1, \dots, i_k$
and $n$ boundary vertices,
we obtain from it a corresponding \emph{modified network} $N'_D$,
by:
\begin{itemize}
\item adding $k$ new boundary vertices to the left of $N_D$, labeled $i'_1,\dots, i'_k$
from top to bottom;
\item adding a new horizontal edge which connects $i'_j$ to the
nearest vertex of the network to its right.
\end{itemize}
\end{definition}
See Figure \ref{fig:modified} for the modified network associated to the
network from Figure \ref{fig:pseudopath}.

\begin{definition}\label{def:pseudopathmodified}
A \emph{pseudopath} $P$ on a modified network is a path on the
modified network which:
\begin{itemize}
\item starts at one of the boundary vertices
labeled $i'_1,\dots, i'_k$, and ends
at one of the boundary vertices labeled $1,2,\dots, n$;
\item takes a sequence of steps which has the form
$$(EE^*S)^*E^*.$$
\end{itemize}
\end{definition}
The arrows in Figure \ref{fig:modified} indicate the allowed directions
in which a path may travel.

The \emph{weight} of a pseudopath in a modified network is defined
the same way as the weight of a pseudopath in a network (see Definition
\ref{def:pseudoweight}).  Note that since a pseudopath in a modified network
does not contain steps west along edges weighted $p_i$, its weight is a monomial (not a Laurent monomial).

\begin{figure}[ht]
\begin{center}
\psset{xunit=0.6cm,yunit=0.6cm,dotstyle=o,dotsize=5pt 0,linewidth=0.8pt,arrowsize=3pt 2,arrowinset=0.25}
\begin{pspicture*}(-1,0)(10,9)

\psline{<-}(5,2)(2,2) \uput[u](2.5,2){$p_2$}
\psline{<-}(7,4)(4,4) \uput[u](4.5,4){$p_4$}
\psline{->}(4,4)(2,4) \uput[u](2.5,4){$m_5$}
\psline{<-}(7,6)(6,6) \uput[u](6.5,6){$p_6$}
\psline{->}(6,6)(4,6) \uput[u](4.5,6){$m_7$}
\pscurve{<-}(6,6)(4,7)(2,6) \uput[u](2.5,6.2){$p_8$}
\psline{<-}(9,8)(8,8) \uput[u](8.5,8){$p_9$}
\psline{<-}(8,8)(6,8) \uput[u](6.5,8){$p_{10}$}
\psline{<-}(6,8)(4,8) \uput[u](4.5,8){$p_{11}$}
\psline{<-}(4,8)(2,8) \uput[u](2.5,8){$p_{12}$}

\psline{->}(2,8)(2,6)
\pscurve{->}(2,6)(1.4,4)(2,2)
\psline{->}(2,4)(2,2)
\psline{->}(2,2)(2,1)
\pscurve{->}(4,8)(3.4,6)(4,4)
\psline{->}(4,6)(4,4)
\psline{->}(4,4)(4,1)
\psline{->}(6,8)(6,6)
\psline{->}(6,6)(6,3.02)
\psline{->}(8,8)(8,7)

\psdots[dotstyle=*,linecolor=black](2,8)
\psdots[dotstyle=*,linecolor=black](2,6)
\psdots[dotsize=9pt 0,dotstyle=*,linecolor=black](2,4)
\psdots[dotstyle=*,linecolor=black](2,2)
\psdots[dotstyle=*,linecolor=black](4,4)
\psdots[dotsize=9pt 0,dotstyle=*,linecolor=black](4,6)
\psdots[dotstyle=*,linecolor=black](4,8)
\psdots[dotstyle=*,linecolor=black](6,8)
\psdots[dotstyle=*,linecolor=black](8,8)
\psdots[dotstyle=*,linecolor=black](6,6)
\psdots[dotstyle=*,linecolor=black](9,8)


\uput[r](9.08,8.12){\large{1}}
\psdots[dotstyle=*,linecolor=black](8,7)
\uput[r](8.08,7.12){\large{2}}
\psdots[dotstyle=*,linecolor=black](7,6)
\uput[r](7.08,6.12){\large{3}}
\psdots[dotstyle=*,linecolor=black](7,4)
\uput[r](7.08,4.12){\large{4}}
\psdots[dotstyle=*,linecolor=black](6,3.02)
\uput[r](6.08,3.14){\large{5}}
\psdots[dotstyle=*,linecolor=black](5,2)
\uput[r](5.08,2.12){\large{6}}
\psdots[dotstyle=*,linecolor=black](4,1)
\uput[r](4.08,1.12){\large{7}}
\psdots[dotstyle=*,linecolor=black](2,1)
\uput[r](2.08,1.12){\large{8}}

\psdots[dotstyle=*,linecolor=black](0,8)
\uput[l](0.08,8.12){\large{$1'$}}
\psdots[dotstyle=*,linecolor=black](0,6)
\uput[l](0.08,6.12){\large{$3'$}}
\psdots[dotstyle=*,linecolor=black](0,4)
\uput[l](0.08,4.12){\large{$4'$}}
\psdots[dotstyle=*,linecolor=black](0,2)
\uput[l](0.08,2.12){\large{$6'$}}

\psline{->}(0,8)(2,8)
\psline{->}(0,6)(2,6)
\psline{->}(0,4)(2,4)
\psline{->}(0,2)(2,2)

\end{pspicture*}
\end{center}
\caption{Example of a modified network.}
\label{fig:modified}
\end{figure}
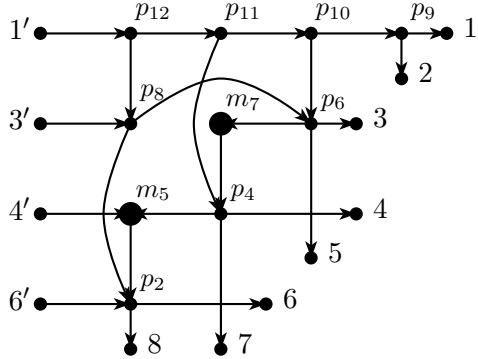


\begin{theorem}\label{th2:MR-matrix}
Let $D$ be a diagram contained in a $k \times (n-k)$ rectangle,
and let $M=(M_{st})$ be
the corresponding $k \times n$ MR-matrix.
Then $${M}_{st} = \sum_P w(P),$$
where the sum is over all pseudopaths in the modified network from
the boundary vertex  $s'$ to the boundary vertex $t$.
\end{theorem}

\begin{lemma}\label{lem:equiv}
Theorems \ref{th:MR-matrix} and \ref{th2:MR-matrix} are equivalent.
\end{lemma}
\begin{proof}
There is an obvious bijection between pseudopaths in a network
starting at boundary vertex $s$, and
pseudopaths in the corresponding modified network starting at boundary
vertex $s'$.  The weights of the corresponding pseudopaths are the same
except for a factor of $(-1)^b \prod p_i$, where
$b$ is the number of
\raisebox{0.12cm}{\hskip0.15cm\circle*{4}}-vertices in row $s$ of the network,
and the $p_i$'s range over
all edge weights in row $s$.

On the other hand,  Lemma \ref{lem:Mproperties}
implies that the leftmost nonzero entry of row $s$ of the MR-matrix $M$
is precisely the quantity $(-1)^b \prod {p_i}$ above.
Therefore
Theorems \ref{th:MR-matrix} and \ref{th2:MR-matrix} are equivalent.
\end{proof}

By Lemma \ref{lem:equiv}, in order to prove Theorem
\ref{th:MR-matrix}, it suffices to prove
Theorem \ref{th2:MR-matrix}.
Our strategy for proving Theorem \ref{th2:MR-matrix} will be to interpret entries of the MR-matrix
in terms of paths in a chip network, and then construct a weight-preserving
bijection between these paths and between pseudopaths in the modified network.

\begin{definition}
A \emph{chip} is one of the three configurations shown in Figure
\ref{fig:chip}.  We call the three configurations
\emph{$y_i(p)$}- or \emph{$y_i$-chips},
\emph{$s_i$-chips}, and
\emph{$x_i(m)$}- or \emph{$x_i$-chips}, respectively.
\end{definition}

\begin{figure}[ht]
\begin{center}
\psset{unit=.6cm,dotstyle=o,dotsize=5pt 0,linewidth=0.8pt,arrowsize=3pt 2,arrowinset=0.25}
\begin{pspicture*}(-0.5,-2)(6,8)
\psline{-}(1,7)(5,7) \uput[l](1,7){$1$}
\psline{-}(1,6)(5,6) \uput[l](1,6){$2$}
\uput[l](3.25,5){$\vdots$}
\psline{-}(1,4)(5,4) \uput[l](1,4){$i$}
\psline{-}(1,3)(5,3) \uput[l](1,3){$i+1$}
\psline{-}(2.5,4)(3.5,3)\uput[l](3,3.5){$p$}
\uput[l](3.25,2){$\vdots$}
\psline{-}(1,1)(5,1) \uput[l](1,1){$n$}
\uput[l](3.75,0){$y_i(p)$}
\uput[l](4.25,-1){``$y_i$-chip''}

\end{pspicture*}\quad\begin{pspicture*}(-0.5,-2)(6,8)
\psline{-}(1,7)(5,7) \uput[l](1,7){$1$}
\psline{-}(1,6)(5,6) \uput[l](1,6){$2$}
\uput[l](3.25,5){$\vdots$}
\uput[l](1,4){$i$}
\uput[l](1,3){$i+1$}
\uput[u](2,2.8){$-1$}
\pscurve{-}(1,4)(2.25,4)(2.5,3.95)(3,3.5)(3.5,3.05)(3.75,3)(5,3)
\pscurve{-}(1,3)(2.25,3)(2.5,3.05)(3,3.5)(3.5,3.95)(3.75,4)(5,4)
\uput[l](3.25,2){$\vdots$}
\psline{-}(1,1)(5,1) \uput[l](1,1){$n$}
\uput[l](3.5,0){$\dot s_i$}
\uput[l](4.25,-1){``$s_i$-chip''}

\end{pspicture*}\quad\begin{pspicture*}(-0.5,-2)(5,8)
\psline{-}(1,7)(5,7) \uput[l](1,7){$1$}
\psline{-}(1,6)(5,6) \uput[l](1,6){$2$}
\uput[l](3.25,5){$\vdots$}
\uput[l](1,4){$i$}
\uput[l](1,3){$i+1$}
\uput[u](4.5,2.8){$-1$}
\pscurve{-}(1,4)(2.75,4)(3,3.95)(3.5,3.5)(4,3.05)(4.25,3)(5,3)
\pscurve{-}(1,3)(2.75,3)(3,3.05)(3.5,3.5)(4,3.95)(4.25,4)(5,4)
\psline{-}(1.5,3)(2.5,4) \uput[l](2,3.5){$m$}
\uput[l](3.25,2){$\vdots$}
\psline{-}(1,1)(5,1) \uput[l](1,1){$n$}
\uput[l](4.3,0){$x_i(m)\dot s_i^{-1}$}
\uput[l](4.25,-1){``$x_i$-chip''}

\end{pspicture*}
\end{center}
\caption{The three types of chips. }
\label{fig:chip}
\end{figure}
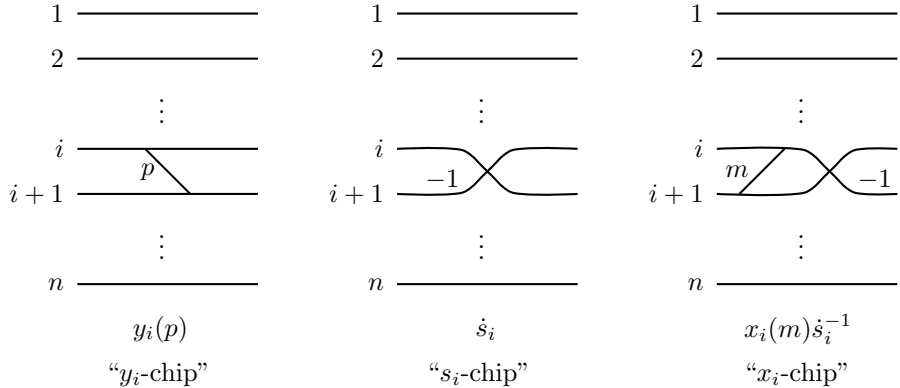

\begin{definition}
A \emph{chip network} is a concatenation of chips. Note that
it has $n$ boundary vertices at the left and $n$ boundary vertices
at the right.
Let $g$ be any
product of factors of the form
$y_i(p)$, $\dot s_i$, and $x_i(m) \dot s_i^{-1}$.
We associate a chip network $C_g$ to $g$
by concatenating the chips corresponding to the factors of $g$ in the order given by the factorization.
\end{definition}

\begin{example}
The chip network $C_g$ associated to the product $g$ from
\eqref{eq:MR-product} is shown in Figure \ref{fig:chipnetwork}.
\end{example}
\begin{figure}[ht]
\begin{center}
\psset{unit=.6cm,dotstyle=o,dotsize=5pt 0,linewidth=0.8pt,arrowsize=3pt 2,arrowinset=0.25}
\begin{pspicture*}(-1.5,-0.5)(21.5,8.5)

\psline{-}(-0.5,8)(20.5,8) \uput[l](-0.5,8){$1$}\uput[r](20.5,8){$1$}
\psline{-}(-0.5,7)(20.5,7) \uput[l](-0.5,7){$2$}\uput[r](20.5,7){$2$}
\psline{-}(-0.5,6)(20.5,6) \uput[l](-0.5,6){$3$}\uput[r](20.5,6){$3$}
\psline{-}(-0.5,5)(3,5)\pscurve{-}(3,5)(3.25,4.95)(3.75, 4.5)(4.25,4.05)(4.5,4)\psline{-}(4.5,4)(11,4)\pscurve{-}(11,4)(11.25,4.05)(11.75, 4.5)(12.25,4.95)(12.5,5)\psline{-}(12.5,5)(20.5,5) \uput[l](-0.5,5){$4$}\uput[r](20.5,5){$4$}
\psline{-}(-0.5,4)(3,4)\pscurve{-}(3,4)(3.25,4.05)(3.75, 4.5)(4.25,4.95)(4.5,5)\psline{-}(4.5,5)(11,5)\pscurve{-}(11,5)(11.25,4.95)(11.75, 4.5)(12.25,4.05)(12.5,4)\psline{-}(12.5,4)(20.5,4) \uput[l](-0.5,4){$5$}\uput[r](20.5,4){$5$}
\psline{-}(-0.5,3)(0,3)\pscurve{-}(0,3)(0.25,2.95)(0.75, 2.5)(1.25,2.05)(1.5,2)\psline{-}(1.5,2)(7.25,2)\pscurve{-}(7.25,2)(7.5,2.05)(8, 2.5)(8.5,2.95)(8.75,3)\psline{-}(8.75,3)(20.5,3) \uput[l](-0.5,3){$6$}\uput[r](20.5,3){$6$}
\psline{-}(-0.5,2)(0,2)\pscurve{-}(0,2)(0.25,2.05)(0.75, 2.5)(1.25,2.95)(1.5,3)\psline{-}(1.5,3)(7.25,3)\pscurve{-}(7.25,3)(7.5,2.95)(8, 2.5)(8.5,2.05)(8.75,2)\psline{-}(8.75,2)(20.5,2) \uput[l](-0.5,2){$7$}\uput[r](20.5,2){$7$}
\psline{-}(-0.5,1)(20.5,1) \uput[l](-0.5,1){$8$}\uput[r](20.5,1){$8$}

\uput[u](3,3.8){\footnotesize$-1$}
\uput[u](12.5,3.8){\footnotesize$-1$}
\uput[u](0,1.8){\footnotesize$-1$}
\uput[u](8.75,1.8){\footnotesize$-1$}

\psline{-}(2,2)(3,1)\uput[l](2.5,1.5){$p_2$}
\psline{-}(5,4)(6,3)\uput[l](5.5,3.5){$p_4$}
\psline{-}(8.75,6)(9.75,5)\uput[l](9.25,5.5){$p_6$}
\psline{-}(13,4)(14,3)\uput[l](13.5,3.5){$p_8$}
\psline{-}(14,8)(15,7)\uput[l](14.5,7.5){$p_9$}
\psline{-}(15.5,7)(16.5,6)\uput[l](16,6.5){$p_{10}$}
\psline{-}(17,6)(18,5)\uput[l](17.5,5.5){$p_{11}$}
\psline{-}(18.5,5)(19.5,4)\uput[l](19,4.5){$p_{12}$}

\psline{-}(6,2)(7,3) \uput[l](6.5,2.5){$m_5$}
\psline{-}(9.75,4)(10.75,5) \uput[l](10.25,4.5){$m_7$}

\end{pspicture*}
\end{center}
\caption{The chip network corresponding to the product
\[g=\dot s_6 y_7(p_2)\dot s_4 y_5(p_4) x_6(m_5) \dot s_6^{-1} y_3(p_6) x_4(m_7) \dot s_4^{-1} y_5(p_8) y_1(p_9) y_2(p_{10}) y_3(p_{11}) y_4(p_{12}).\]}
\label{fig:chipnetwork}
\end{figure}
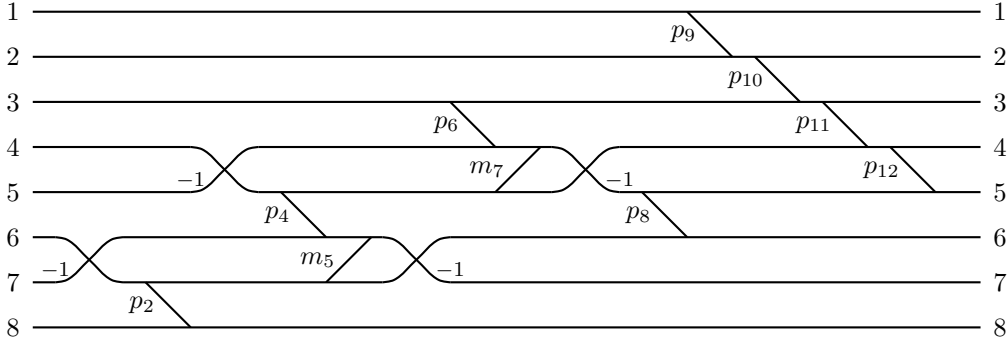

\begin{definition}\label{def:route}
A \emph{route} $Q$ in a chip network is a path in the network
whose steps all travel east (or southeast or northeast
for slanted edges).  The \emph{weight} $w(Q)$ of such a route
is the product of all weights on its edges.
To each chip network $C$ we associate
a \emph{weight matrix} $x(C) = x_{ij}$ where
$x_{ij} = \sum_{Q} w(Q)$ and the sum is over all routes
from the boundary vertex $i$ at the west to the boundary vertex $j$
at the east.
\end{definition}

It is simple to verify the following result.  Recall our convention
from Section \ref{sec:flag} that the rows of $g$ are labeled from
bottom to top, and the columns of $g$ are labeled from right to left.

\begin{lemma}\label{lem:chipnetwork}
Let $g$ be a
product of factors of the form
$y_i(p)$, $\dot s_i$, and $x_i(m) \dot s_i^{-1}$.
Then the weight matrix $x(C_g)$ of the chip network $C_g$ associated
to $g$ coincides with the matrix $g$.
\end{lemma}

We now prove Theorem \ref{th2:MR-matrix}.

\begin{proof}
Let $D$ be a diagram contained in a $k \times (n-k)$ rectangle, and
$N'_D$ the corresponding modified network.  Let
$i'_1 < \dots < i'_k$ be the labels of the sources of $N'_D$.
Let $g$ be the product of factors of the form
$y_i(p)$, $\dot s_i$, and $x_i(m) \dot s_i^{-1}$ which is
encoded by $D$, and let
$M$ be the corresponding MR-matrix, whose
rows are indexed from top to bottom by  $i_1, \dots, i_k$.
Recall that the projection from $g$ to $M$ switches rows and columns,
and the columns labeled $n-k+1, n-k+2, \dots, n$ in $g$
become rows labeled $i_1, i_2, \dots, i_k$ in $M$.
Therefore to prove Theorem \ref{th2:MR-matrix}, it suffices
to prove that for all $1 \leq t \leq n$ and $1 \leq s \leq k$
we have
\begin{equation}\label{suffice:1}
g_{t, s+(n-k)} = \sum_P w(P),
\end{equation}
where the sum
is over all pseudopaths $P$ from
$i'_s$ to $t$ in the modified network.

By Lemma \ref{lem:chipnetwork}, the matrix $g$ coincides with
the weight matrix $x(C_g) = (x_{st})$ of the chip network associated
to $g$. Therefore  by \eqref{suffice:1} it suffices to prove
that for all $1 \leq t \leq n$ and $1 \leq s \leq k$
we have
\begin{equation}\label{suffice:2}
x_{t, s+(n-k)} = \sum_P w(P),
\end{equation}
where the sum
is over all pseudopaths $P$ from
$i'_s$ to $t$ in the modified network.

Recall from Definition \ref{def:route} that $x_{ij} = \sum_{Q} w(Q)$, where
the sum is over all routes $Q$ in the chip network
from the boundary vertex
$i$ at the west to the boundary vertex $j$ at the east.
To prove \eqref{suffice:2}, we will give a weight-preserving
bijection between pseudopaths $P$ in the modified network from
$i'_s$ to $t$, and routes $Q$ in the chip network
from the boundary vertex $t$ at the west to the boundary
vertex $s+(n-k)$ at the east.
More specifically,
given a pseudopath $P$, we will examine its sequence of steps from
source to sink,
and explain how to build the corresponding
route $Q$ in the chip network.  As illustrated in Figure~\ref{fig:bijection},
each step in a pseudopath corresponds to a portion of a route
in a chip network.
(Note that our bijection will build the route in the chip network
from east to west, rather than west to east.)

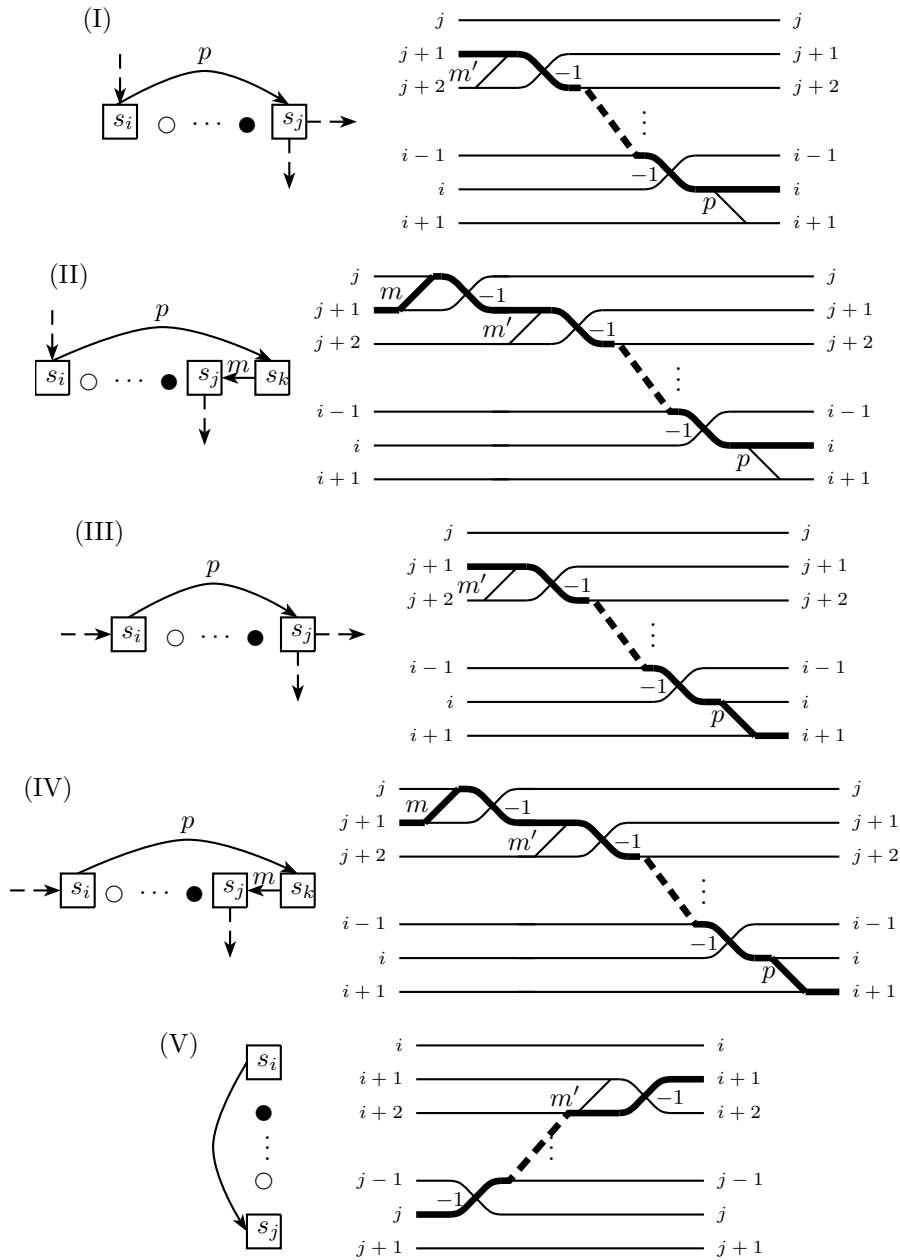
\begin{figure}[ht]
\begin{center}
\psset{unit=.45cm,dotstyle=o,dotsize=5pt 0,linewidth=0.8pt,arrowsize=3pt 2,arrowinset=0.25}
\begin{pspicture*}(-8,0)(15,7.5)
\uput[r](-8,7){(I)}
\psline{-}(-1,3.5)(-2,3.5)(-2,4.5)(-1,4.5)(-1,3.5) \uput[r](-2.1,4){$s_j$}
\psline{-}(-6,3.5)(-7,3.5)(-7,4.5)(-6,4.5)(-6,3.5) \uput[r](-7.1,4){$s_i$}
\pscurve{->}(-6.6,4.5)(-4,5.5)(-1.5,4.5)\uput[u](-4,5.3){$p$}

\uput[u](-3.7,3.5){\circle{0.5} $\hdots\hspace{6pt}$ \circle*{0.5}}

\psline[linestyle=dashed]{->}(-6.5,6)(-6.5,4.5)
\psline[linestyle=dashed]{->}(-1,4)(0.5,4)
\psline[linestyle=dashed]{->}(-1.5,3.5)(-1.5,2)

\psline{-}(3.5,7)(13,7)
\psline[linewidth=2.5pt]{-}(3.5,6)(5.25,6)\pscurve[linewidth=2.5pt]{-}(5.25,6)(5.5,5.95)(6, 5.5)(6.5,5.05)(6.75,5)\psline{-}(6.75,5)(13,5)
\psline[linewidth=2.5pt,linestyle=dashed]{-}(6.75,5)(7.25,5)(8.75,3)(9,3)
\psline{-}(4,5)(5,6) \uput[l](4.5,5.5){$m'$}
\psline{-}(3.5,5)(5.25,5)\pscurve{-}(5.25,5)(5.5,5.05)(6, 5.5)(6.5,5.95)(6.75,6)\psline{-}(6.75,6)(13,6)
\uput[u](6.75,4.8){\footnotesize{$-1$}}
\uput[l](9.5,4.2){$\vdots$}
\psline{-}(3.5,3)(9,3)\pscurve[linewidth=2.5pt]{-}(9,3)(9.25,2.95)(9.75, 2.5)(10.25,2.05)(10.5,2)\psline[linewidth=2.5pt]{-}(10.5,2)(13,2)
\psline{-}(3.5,2)(9,2)\pscurve{-}(9,2)(9.25,2.05)(9.75, 2.5)(10.25,2.95)(10.5,3)\psline{-}(10.5,3)(13,3)
\uput[u](9,1.8){\footnotesize{$-1$}}
\psline{-}(3.5,1)(13,1)
\psline{-}(11,2)(12,1) \uput[l](11.5,1.5){$p$}

\uput[l](3.5,7){\scriptsize$j$}
\uput[l](3.5,6){\scriptsize$j+1$}
\uput[l](3.5,5){\scriptsize$j+2$}
\uput[l](3.5,3){\scriptsize$i-1$}
\uput[l](3.5,2){\scriptsize$i$}
\uput[l](3.5,1){\scriptsize$i+1$}

\uput[r](13,7){\scriptsize$j$}
\uput[r](13,6){\scriptsize$j+1$}
\uput[r](13,5){\scriptsize$j+2$}
\uput[r](13,3){\scriptsize$i-1$}
\uput[r](13,2){\scriptsize$i$}
\uput[r](13,1){\scriptsize$i+1$}

\end{pspicture*}

\begin{pspicture*}(-10,0)(15,7.5)
\uput[r](-10,7){(II)}
\psline{-}(-2.5,3.5)(-3.5,3.5)(-3.5,4.5)(-2.5,4.5)(-2.5,3.5) \uput[r](-3.6,4){$s_k$}
\psline{-}(-9,3.5)(-10,3.5)(-10,4.5)(-9,4.5)(-9,3.5) \uput[r](-10.1,4){$s_i$}
\psline{-}(-4.5,3.5)(-5.5,3.5)(-5.5,4.5)(-4.5,4.5)(-4.5,3.5) \uput[r](-5.6,4){$s_j$}

\pscurve{->}(-9.5,4.5)(-6.25,5.5)(-3,4.5) \uput[u](-6.25,5.3){$p$}
\psline{->}(-3.5,4)(-4.5,4) \uput[u](-4,3.8){$m$}
\uput[u](-7,3.5){\circle{0.5} $\hdots\hspace{6pt}$ \circle*{0.5}}

\psline[linestyle=dashed]{->}(-9.5,6)(-9.5,4.5)
\psline[linestyle=dashed]{->}(-5,3.5)(-5,2)

\psline{-}(0,7)(1.75,7)\psline[linewidth=2.5pt]{-}(1.75,7)(2,7)\pscurve[linewidth=2.5pt]{-}(2,7)(2.25,6.95)(2.75, 6.5)(3.25,6.05)(3.5,6)\psline[linewidth=2.5pt]{-}(3.5,6)(4,6)
\psline[linewidth=2.5pt]{-}(0,6)(0.75,6)\psline{-}(0.75,6)(2,6)\pscurve{-}(2,6)(2.25,6.05)(2.75, 6.5)(3.25,6.95)(3.5,7)\psline{-}(3.5,7)(4,7)
\psline[linewidth=2.5pt]{-}(0.75,6)(1.75,7) \uput[l](1.25,6.5){$m$}
\psline{-}(0,5)(4,5)
\uput[u](3.5,5.8){\footnotesize{$-1$}}
\psline{-}(0,3)(4,3)
\psline{-}(0,2)(4,2)
\psline{-}(0,1)(4,1)

\psline{-}(3.5,7)(13,7)
\psline[linewidth=2.5pt]{-}(3.5,6)(5.25,6)\pscurve[linewidth=2.5pt]{-}(5.25,6)(5.5,5.95)(6, 5.5)(6.5,5.05)(6.75,5)\psline{-}(6.75,5)(13,5)
\psline[linewidth=2.5pt,linestyle=dashed]{-}(6.75,5)(7.25,5)(8.75,3)(9,3)
\psline{-}(4,5)(5,6) \uput[l](4.5,5.5){$m'$}
\psline{-}(3.5,5)(5.25,5)\pscurve{-}(5.25,5)(5.5,5.05)(6, 5.5)(6.5,5.95)(6.75,6)\psline{-}(6.75,6)(13,6)
\uput[u](6.75,4.8){\footnotesize{$-1$}}
\uput[l](9.5,4.2){$\vdots$}
\psline{-}(3.5,3)(9,3)\pscurve[linewidth=2.5pt]{-}(9,3)(9.25,2.95)(9.75, 2.5)(10.25,2.05)(10.5,2)\psline[linewidth=2.5pt]{-}(10.5,2)(13,2)
\psline{-}(3.5,2)(9,2)\pscurve{-}(9,2)(9.25,2.05)(9.75, 2.5)(10.25,2.95)(10.5,3)\psline{-}(10.5,3)(13,3)
\uput[u](9,1.8){\footnotesize{$-1$}}
\psline{-}(3.5,1)(13,1)
\psline{-}(11,2)(12,1) \uput[l](11.5,1.5){$p$}

\uput[l](0,7){\scriptsize$j$}
\uput[l](0,6){\scriptsize$j+1$}
\uput[l](0,5){\scriptsize$j+2$}
\uput[l](0,3){\scriptsize$i-1$}
\uput[l](0,2){\scriptsize$i$}
\uput[l](0,1){\scriptsize$i+1$}

\uput[r](13,7){\scriptsize$j$}
\uput[r](13,6){\scriptsize$j+1$}
\uput[r](13,5){\scriptsize$j+2$}
\uput[r](13,3){\scriptsize$i-1$}
\uput[r](13,2){\scriptsize$i$}
\uput[r](13,1){\scriptsize$i+1$}

\end{pspicture*}

\begin{pspicture*}(-8.5,0)(15,7.5)
\uput[r](-8.5,7){(III)}
\psline{-}(-1,3.5)(-2,3.5)(-2,4.5)(-1,4.5)(-1,3.5)\uput[r](-2.1,4){$s_j$}
\psline{-}(-6,3.5)(-7,3.5)(-7,4.5)(-6,4.5)(-6,3.5)\uput[r](-7.1,4){$s_i$}
\pscurve{->}(-6.5,4.5)(-4,5.5)(-1.5,4.5)\uput[u](-4,5.3){$p$}

\uput[u](-3.7,3.5){\circle{0.5} $\hdots\hspace{6pt}$ \circle*{0.5}}

\psline[linestyle=dashed]{->}(-8.5,4)(-7,4)
\psline[linestyle=dashed]{->}(-1,4)(0.5,4)
\psline[linestyle=dashed]{->}(-1.5,3.5)(-1.5,2)

\psline{-}(3.5,7)(13,7)
\psline[linewidth=2.5pt]{-}(3.5,6)(5.25,6)\pscurve[linewidth=2.5pt]{-}(5.25,6)(5.5,5.95)(6, 5.5)(6.5,5.05)(6.75,5)\psline{-}(6.75,5)(13,5)
\psline[linewidth=2.5pt,linestyle=dashed]{-}(6.75,5)(7.25,5)(8.75,3)(9,3)
\psline{-}(4,5)(5,6) \uput[l](4.5,5.5){$m'$}
\psline{-}(3.5,5)(5.25,5)\pscurve{-}(5.25,5)(5.5,5.05)(6, 5.5)(6.5,5.95)(6.75,6)\psline{-}(6.75,6)(13,6)
\uput[u](6.75,4.8){\footnotesize{$-1$}}
\uput[l](9.5,4.2){$\vdots$}
\psline{-}(3.5,3)(9,3)\pscurve[linewidth=2.5pt]{-}(9,3)(9.25,2.95)(9.75, 2.5)(10.25,2.05)(10.5,2)\psline[linewidth=2.5pt]{-}(10.5,2)(11,2)\psline{-}(11,2)(13,2)
\psline{-}(3.5,2)(9,2)\pscurve{-}(9,2)(9.25,2.05)(9.75, 2.5)(10.25,2.95)(10.5,3)\psline{-}(10.5,3)(13,3)
\uput[u](9,1.8){\footnotesize{$-1$}}
\psline{-}(3.5,1)(12,1)\psline[linewidth=2.5pt]{-}(12,1)(13,1)
\psline[linewidth=2.5pt]{-}(11,2)(12,1) \uput[l](11.5,1.5){$p$}

\uput[l](3.5,7){\scriptsize$j$}
\uput[l](3.5,6){\scriptsize$j+1$}
\uput[l](3.5,5){\scriptsize$j+2$}
\uput[l](3.5,3){\scriptsize$i-1$}
\uput[l](3.5,2){\scriptsize$i$}
\uput[l](3.5,1){\scriptsize$i+1$}

\uput[r](13,7){\scriptsize$j$}
\uput[r](13,6){\scriptsize$j+1$}
\uput[r](13,5){\scriptsize$j+2$}
\uput[r](13,3){\scriptsize$i-1$}
\uput[r](13,2){\scriptsize$i$}
\uput[r](13,1){\scriptsize$i+1$}

\end{pspicture*}

\begin{pspicture*}(-11.5,0)(15,7.5)
\uput[r](-11.5,7){(IV)}
\psline{-}(-2.5,3.5)(-3.5,3.5)(-3.5,4.5)(-2.5,4.5)(-2.5,3.5) \uput[r](-3.6,4){$s_k$}
\psline{-}(-9,3.5)(-10,3.5)(-10,4.5)(-9,4.5)(-9,3.5) \uput[r](-10.1,4){$s_i$}
\psline{-}(-4.5,3.5)(-5.5,3.5)(-5.5,4.5)(-4.5,4.5)(-4.5,3.5) \uput[r](-5.6,4){$s_j$}

\pscurve{->}(-9.5,4.5)(-6.25,5.5)(-3,4.5) \uput[u](-6.25,5.3){$p$}
\psline{->}(-3.5,4)(-4.5,4) \uput[u](-4,3.8){$m$}
\uput[u](-7,3.5){\circle{0.5} $\hdots\hspace{6pt}$ \circle*{0.5}}

\psline[linestyle=dashed]{->}(-11.5,4)(-10,4)
\psline[linestyle=dashed]{->}(-5,3.5)(-5,2)

\psline{-}(0,7)(1.75,7)\psline[linewidth=2.5pt]{-}(1.75,7)(2,7)\pscurve[linewidth=2.5pt]{-}(2,7)(2.25,6.95)(2.75, 6.5)(3.25,6.05)(3.5,6)\psline[linewidth=2.5pt]{-}(3.5,6)(4,6)
\psline[linewidth=2.5pt]{-}(0,6)(0.75,6)\psline{-}(0.75,6)(2,6)\pscurve{-}(2,6)(2.25,6.05)(2.75, 6.5)(3.25,6.95)(3.5,7)\psline{-}(3.5,7)(4,7)
\psline[linewidth=2.5pt]{-}(0.75,6)(1.75,7) \uput[l](1.25,6.5){$m$}
\psline{-}(0,5)(4,5)
\uput[u](3.5,5.8){\footnotesize{$-1$}}
\psline{-}(0,3)(4,3)
\psline{-}(0,2)(4,2)
\psline{-}(0,1)(4,1)

\psline{-}(3.5,7)(13,7)
\psline[linewidth=2.5pt]{-}(3.5,6)(5.25,6)\pscurve[linewidth=2.5pt]{-}(5.25,6)(5.5,5.95)(6, 5.5)(6.5,5.05)(6.75,5)\psline{-}(6.75,5)(13,5)
\psline[linewidth=2.5pt,linestyle=dashed]{-}(6.75,5)(7.25,5)(8.75,3)(9,3)
\psline{-}(4,5)(5,6) \uput[l](4.5,5.5){$m'$}
\psline{-}(3.5,5)(5.25,5)\pscurve{-}(5.25,5)(5.5,5.05)(6, 5.5)(6.5,5.95)(6.75,6)\psline{-}(6.75,6)(13,6)
\uput[u](6.75,4.8){\footnotesize{$-1$}}
\uput[l](9.5,4.2){$\vdots$}
\psline{-}(3.5,3)(9,3)\pscurve[linewidth=2.5pt]{-}(9,3)(9.25,2.95)(9.75, 2.5)(10.25,2.05)(10.5,2)\psline[linewidth=2.5pt]{-}(10.5,2)(11,2)\psline{-}(11,2)(13,2)
\psline{-}(3.5,2)(9,2)\pscurve{-}(9,2)(9.25,2.05)(9.75, 2.5)(10.25,2.95)(10.5,3)\psline{-}(10.5,3)(13,3)
\uput[u](9,1.8){\footnotesize{$-1$}}
\psline{-}(3.5,1)(12,1)\psline[linewidth=2.5pt]{-}(12,1)(13,1)
\psline[linewidth=2.5pt]{-}(11,2)(12,1) \uput[l](11.5,1.5){$p$}

\uput[l](0,7){\scriptsize$j$}
\uput[l](0,6){\scriptsize$j+1$}
\uput[l](0,5){\scriptsize$j+2$}
\uput[l](0,3){\scriptsize$i-1$}
\uput[l](0,2){\scriptsize$i$}
\uput[l](0,1){\scriptsize$i+1$}

\uput[r](13,7){\scriptsize$j$}
\uput[r](13,6){\scriptsize$j+1$}
\uput[r](13,5){\scriptsize$j+2$}
\uput[r](13,3){\scriptsize$i-1$}
\uput[r](13,2){\scriptsize$i$}
\uput[r](13,1){\scriptsize$i+1$}

\end{pspicture*}

\begin{pspicture*}(-8,0)(10.5,7.5)
\uput[r](-8,7){(V)}
\psline{-}(-4,7)(-4,6)(-5,6)(-5,7)(-4,7)\uput[r](-5.1,6.5){$s_i$}
\psline{-}(-4,2)(-4,1)(-5,1)(-5,2)(-4,2)\uput[r](-5.1,1.5){$s_j$}
\pscurve{->}(-5,6.5)(-6,4)(-5,1.5)

\uput[r](-4.9,5){\circle*{0.5}}
\uput[r](-4.9,4.2){$\vdots$}
\uput[r](-4.9,3){\circle{0.5}}

\psline{-}(0,7)(8.5,7)
\psline{-}(0,6)(6,6)\pscurve{-}(6,6)(6.25,5.95)(6.75, 5.5)(7.25,5.05)(7.5,5)\psline{-}(7.5,5)(8.5,5)
\psline[linewidth=2.5pt,linestyle=dashed]{-}(2.5,3)(2.75,3)(4.5,5)(4.75,5)
\psline{-}(4.75,5)(5.75,6) \uput[l](5.25,5.5){$m'$}
\psline{-}(0,5)(4.75,5)\psline[linewidth=2.5pt]{-}(4.75,5)(6,5)\pscurve[linewidth=2.5pt]{-}(6,5)(6.25,5.05)(6.75, 5.5)(7.25,5.95)(7.5,6)\psline[linewidth=2.5pt]{-}(7.5,6)(8.5,6)
\uput[u](7.5,4.8){\footnotesize{$-1$}}
\uput[l](4.5,4.2){$\vdots$}
\psline{-}(0,3)(1,3)\pscurve{-}(1,3)(1.25,2.95)(1.75,2.5)(2.25,2.05)(2.5,2)\psline{-}(2.5,2)(8.5,2)
\psline[linewidth=2.5pt]{-}(0,2)(1,2)\pscurve[linewidth=2.5pt]{-}(1,2)(1.25,2.05)(1.75, 2.5)(2.25,2.95)(2.5,3)\psline{-}(2.5,3)(8.5,3)
\uput[u](1,1.8){\footnotesize{$-1$}}
\psline{-}(0,1)(8.5,1)

\uput[l](0,7){\scriptsize$i$}
\uput[l](0,6){\scriptsize$i+1$}
\uput[l](0,5){\scriptsize$i+2$}
\uput[l](0,3){\scriptsize$j-1$}
\uput[l](0,2){\scriptsize$j$}
\uput[l](0,1){\scriptsize$j+1$}

\uput[r](8.5,7){\scriptsize$i$}
\uput[r](8.5,6){\scriptsize$i+1$}
\uput[r](8.5,5){\scriptsize$i+2$}
\uput[r](8.5,3){\scriptsize$j-1$}
\uput[r](8.5,2){\scriptsize$j$}
\uput[r](8.5,1){\scriptsize$j+1$}

\end{pspicture*}
\end{center}
\caption{Steps in pseudopaths and their corresponding fragments of the chip network.}
\label{fig:bijection}
\end{figure}

Recall from Figure \ref{fig:readingorder} that each modified network comes from a
diagram, and that every box of a diagram is naturally associated
with a simple generator $s_i$.  Therefore every internal vertex in a modified
network is naturally associated with a simple generator $s_i$ for some $i$.
We will call this the \emph{position} of the vertex.

Let us consider the various kinds of steps in a pseudopath.
Such steps naturally
fall into one of the following categories
(illustrated in Figure \ref{fig:bijection}):
\begin{enumerate}
\item[0.] A single step east, which starts at a source, and ends
at position $s_i$.
\item[I.] A single step east, which is preceded by a south step,
and followed by an east or south step.  Such a step starts and ends
at positions $s_i$ and $s_j$ (for $i>j$), and is labeled by some weight $p$.
It may skip over some (positions corresponding to)
white and black stones in the Go-diagram.
\item[II.] An east-west combination step, which is preceded by a south
step, and travels from position $s_i$ to $s_k$ to $s_j$
(where $i>j>k$).  The two components of such a step are labeled by
some weights $p$ and $m$, and may skip over some white and black stones.
\item[III.] A single step east, which is preceded by an east step,
and followed by an east or south step.  Such a step starts and ends
at positions $s_i$ and $s_j$ (for $i>j$), and is labeled by some weight $p$.
It may skip over some white and black stones.
\item[IV.]
An east-west combination step, which is preceded by an east
step, and travels from position $s_i$ to $s_k$ to $s_j$
(where $i>j>k$).  The two components of such a step are labeled by
some weights $p$ and $m$, and may skip over some white and black stones.
\item[V.] A south step, which starts and ends at positions
$s_i$ and $s_j$ (for $i<j$).  Such a step may skip over some
white and black stones.
\end{enumerate}

The above steps in a pseudopath correspond to the following portion of a
route in a chip network:
\begin{enumerate}
\item[0.] Start at the boundary vertex $i+1$ at the east of the chip network.
\item[I.] Start at level $i$,
then travel west straight across the $y_i(p)$ chip.
For each white
or black stone (say in position $s_{\ell}$)
which lies in between positions $s_i$ and $s_j$,
travel northwest through the corresponding $s_{\ell}$  or $x_{\ell}$-chip, ending at level $j+1$.
\item[II.] Start at level $i$,
then travel west straight across the $y_i(p)$ chip.
For each white
or black stone (say in position $s_{\ell}$)
which lies in between positions $s_i$ and $s_j$,
travel northwest through the corresponding $s_{\ell}$  or $x_{\ell}$-chip.
Finally, travel along the $-1$-edge and then the $m$-edge of the
$x_j(m)$ chip, ending at level $j+1$.
\item[III.] Start at level $i+1$,
then travel northwest along the $p$-edge in the  $y_i(p)$ chip.
For each white
or black stone (say in position $s_{\ell}$)
which lies in between positions $s_i$ and $s_j$,
travel northwest through the corresponding $s_{\ell}$  or $x_{\ell}$-chip, ending at level $j+1$.
\item[IV.] Start at level $i+1$,
then travel northwest along the $p$-edge in  the $y_i(p)$ chip.
For each white
or black stone (say in position $s_{\ell}$)
which lies in between positions $s_i$ and $s_j$,
travel northwest through the corresponding $s_{\ell}$  or $x_{\ell}$-chip.
Finally, travel along the $-1$-edge and the $m$-edge of the
$x_j(m)$ chip, ending at level $j+1$.
\item[V.] Start at level $i+1$.
For each white
or black stone (say in position $s_{\ell}$)
which lies in between positions $s_i$ and $s_j$,
travel southwest through the corresponding $s_{\ell}$  or $x_{\ell}$-chip, ending at level $j$.
\end{enumerate}

It is a straightforward exercise to verify that this map
is a bijection between pseudopaths $P$ from $i'_s$ to $t$
in the modified
network, and routes $Q$ between the $t$ vertex at the west
and the $s+(n-k)$ vertex at the east in the chip network.
Moreover, the
weights of $P$ and $Q$ are equal.
See Figure~\ref{fig:pseudo-chip-eg} for example of entire pseudopaths and routes.
\end{proof}

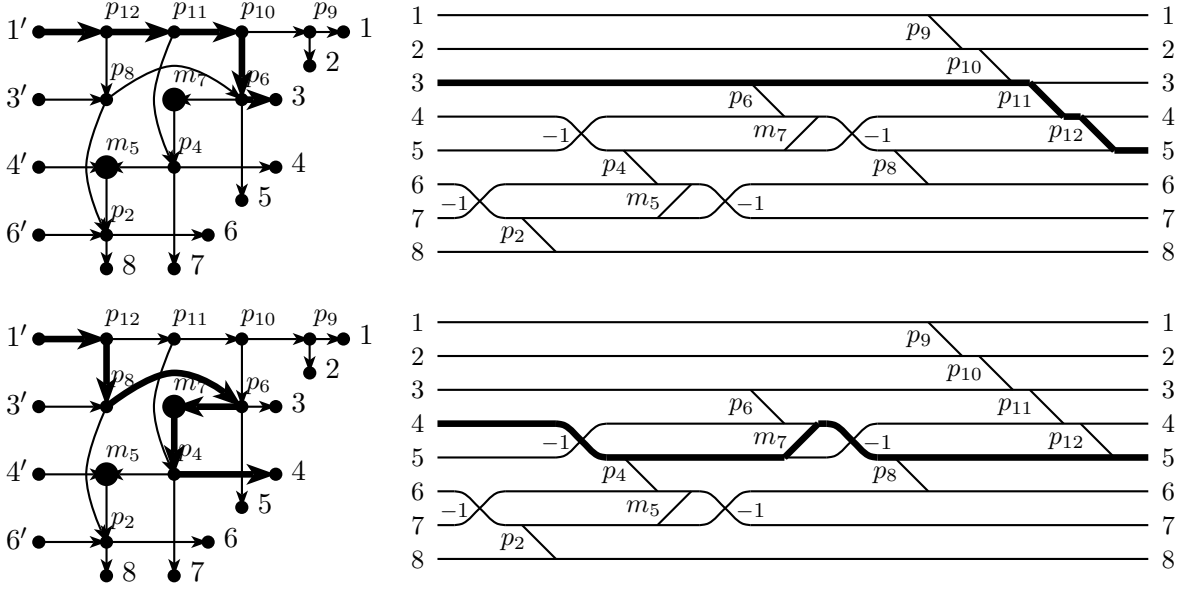
\begin{figure}[ht]
\begin{center}
\psset{unit=.45cm,dotstyle=o,dotsize=5pt 0,linewidth=0.8pt,arrowsize=3pt 2,arrowinset=0.25}

\begin{pspicture*}(-1,0)(10,9)

\psline{<-}(5,2)(2,2) \uput[u](2.5,2){$p_2$}
\psline{<-}(7,4)(4,4) \uput[u](4.5,4){$p_4$}
\psline{->}(4,4)(2,4) \uput[u](2.5,4){$m_5$}
\psline[linewidth=2.5pt]{<-}(7,6)(6,6) \uput[u](6.5,6){$p_6$}
\psline{->}(6,6)(4,6) \uput[u](4.5,6){$m_7$}
\pscurve{<-}(6,6)(4,7)(2,6) \uput[u](2.5,6.2){$p_8$}
\psline{<-}(9,8)(8,8) \uput[u](8.5,8){$p_9$}
\psline{<-}(8,8)(6,8) \uput[u](6.5,8){$p_{10}$}
\psline[linewidth=2.5pt]{<-}(6,8)(4,8) \uput[u](4.5,8){$p_{11}$}
\psline[linewidth=2.5pt]{<-}(4,8)(2,8) \uput[u](2.5,8){$p_{12}$}

\psline{->}(2,8)(2,6)
\pscurve{->}(2,6)(1.4,4)(2,2)
\psline{->}(2,4)(2,2)
\psline{->}(2,2)(2,1)
\pscurve{->}(4,8)(3.4,6)(4,4)
\psline{->}(4,6)(4,4)
\psline{->}(4,4)(4,1)
\psline[linewidth=2.5pt]{->}(6,8)(6,6)
\psline{->}(6,6)(6,3.02)
\psline{->}(8,8)(8,7)

\psdots[dotstyle=*,linecolor=black](2,8)
\psdots[dotstyle=*,linecolor=black](2,6)
\psdots[dotsize=9pt 0,dotstyle=*,linecolor=black](2,4)
\psdots[dotstyle=*,linecolor=black](2,2)
\psdots[dotstyle=*,linecolor=black](4,4)
\psdots[dotsize=9pt 0,dotstyle=*,linecolor=black](4,6)
\psdots[dotstyle=*,linecolor=black](4,8)
\psdots[dotstyle=*,linecolor=black](6,8)
\psdots[dotstyle=*,linecolor=black](8,8)
\psdots[dotstyle=*,linecolor=black](6,6)
\psdots[dotstyle=*,linecolor=black](9,8)


\uput[r](9.08,8.12){\large{1}}
\psdots[dotstyle=*,linecolor=black](8,7)
\uput[r](8.08,7.12){\large{2}}
\psdots[dotstyle=*,linecolor=black](7,6)
\uput[r](7.08,6.12){\large{3}}
\psdots[dotstyle=*,linecolor=black](7,4)
\uput[r](7.08,4.12){\large{4}}
\psdots[dotstyle=*,linecolor=black](6,3.02)
\uput[r](6.08,3.14){\large{5}}
\psdots[dotstyle=*,linecolor=black](5,2)
\uput[r](5.08,2.12){\large{6}}
\psdots[dotstyle=*,linecolor=black](4,1)
\uput[r](4.08,1.12){\large{7}}
\psdots[dotstyle=*,linecolor=black](2,1)
\uput[r](2.08,1.12){\large{8}}

\psdots[dotstyle=*,linecolor=black](0,8)
\uput[l](0.08,8.12){\large{$1'$}}
\psdots[dotstyle=*,linecolor=black](0,6)
\uput[l](0.08,6.12){\large{$3'$}}
\psdots[dotstyle=*,linecolor=black](0,4)
\uput[l](0.08,4.12){\large{$4'$}}
\psdots[dotstyle=*,linecolor=black](0,2)
\uput[l](0.08,2.12){\large{$6'$}}

\psline[linewidth=2.5pt]{->}(0,8)(2,8)
\psline{->}(0,6)(2,6)
\psline{->}(0,4)(2,4)
\psline{->}(0,2)(2,2)

\end{pspicture*}\quad\begin{pspicture*}(-1.5,-0.5)(21.5,8.5)

\psline{-}(-0.5,8)(20.5,8) \uput[l](-0.5,8){$1$}\uput[r](20.5,8){$1$}
\psline{-}(-0.5,7)(20.5,7) \uput[l](-0.5,7){$2$}\uput[r](20.5,7){$2$}
\psline[linewidth=2.5pt]{-}(-0.5,6)(17,6)\psline{-}(17,6)(20.5,6) \uput[l](-0.5,6){$3$}\uput[r](20.5,6){$3$}
\psline{-}(-0.5,5)(3,5)\pscurve{-}(3,5)(3.25,4.95)(3.75, 4.5)(4.25,4.05)(4.5,4)\psline{-}(4.5,4)(11,4)\pscurve{-}(11,4)(11.25,4.05)(11.75, 4.5)(12.25,4.95)(12.5,5)\psline{-}(12.5,5)(18,5)\psline[linewidth=2.5pt]{-}(18,5)(18.5,5)\psline{-}(18.5,5)(20.5,5) \uput[l](-0.5,5){$4$}\uput[r](20.5,5){$4$}
\psline{-}(-0.5,4)(3,4)\pscurve{-}(3,4)(3.25,4.05)(3.75, 4.5)(4.25,4.95)(4.5,5)\psline{-}(4.5,5)(11,5)\pscurve{-}(11,5)(11.25,4.95)(11.75, 4.5)(12.25,4.05)(12.5,4)\psline{-}(12.5,4)(19.5,4)\psline[linewidth=2.5pt]{-}(19.5,4)(20.5,4) \uput[l](-0.5,4){$5$}\uput[r](20.5,4){$5$}
\psline{-}(-0.5,3)(0,3)\pscurve{-}(0,3)(0.25,2.95)(0.75, 2.5)(1.25,2.05)(1.5,2)\psline{-}(1.5,2)(7.25,2)\pscurve{-}(7.25,2)(7.5,2.05)(8, 2.5)(8.5,2.95)(8.75,3)\psline{-}(8.75,3)(20.5,3) \uput[l](-0.5,3){$6$}\uput[r](20.5,3){$6$}
\psline{-}(-0.5,2)(0,2)\pscurve{-}(0,2)(0.25,2.05)(0.75, 2.5)(1.25,2.95)(1.5,3)\psline{-}(1.5,3)(7.25,3)\pscurve{-}(7.25,3)(7.5,2.95)(8, 2.5)(8.5,2.05)(8.75,2)\psline{-}(8.75,2)(20.5,2) \uput[l](-0.5,2){$7$}\uput[r](20.5,2){$7$}
\psline{-}(-0.5,1)(20.5,1) \uput[l](-0.5,1){$8$}\uput[r](20.5,1){$8$}

\uput[u](3,3.8){\footnotesize$-1$}
\uput[u](12.5,3.8){\footnotesize$-1$}
\uput[u](0,1.8){\footnotesize$-1$}
\uput[u](8.75,1.8){\footnotesize$-1$}

\psline{-}(2,2)(3,1)\uput[l](2.5,1.5){$p_2$}
\psline{-}(5,4)(6,3)\uput[l](5.5,3.5){$p_4$}
\psline{-}(8.75,6)(9.75,5)\uput[l](9.25,5.5){$p_6$}
\psline{-}(13,4)(14,3)\uput[l](13.5,3.5){$p_8$}
\psline{-}(14,8)(15,7)\uput[l](14.5,7.5){$p_9$}
\psline{-}(15.5,7)(16.5,6)\uput[l](16,6.5){$p_{10}$}
\psline[linewidth=2.5pt]{-}(17,6)(18,5)\uput[l](17.5,5.5){$p_{11}$}
\psline[linewidth=2.5pt]{-}(18.5,5)(19.5,4)\uput[l](19,4.5){$p_{12}$}

\psline{-}(6,2)(7,3) \uput[l](6.5,2.5){$m_5$}
\psline{-}(9.75,4)(10.75,5) \uput[l](10.25,4.5){$m_7$}

\end{pspicture*}

\begin{pspicture*}(-1,0)(10,9)

\psline{<-}(5,2)(2,2) \uput[u](2.5,2){$p_2$}
\psline[linewidth=2.5pt]{<-}(7,4)(4,4) \uput[u](4.5,4){$p_4$}
\psline{->}(4,4)(2,4) \uput[u](2.5,4){$m_5$}
\psline{<-}(7,6)(6,6) \uput[u](6.5,6){$p_6$}
\psline[linewidth=2.5pt]{->}(6,6)(4,6) \uput[u](4.5,6){$m_7$}
\pscurve[linewidth=2.5pt]{<-}(6,6)(4,7)(2,6) \uput[u](2.5,6.2){$p_8$}
\psline{<-}(9,8)(8,8) \uput[u](8.5,8){$p_9$}
\psline{<-}(8,8)(6,8) \uput[u](6.5,8){$p_{10}$}
\psline{<-}(6,8)(4,8) \uput[u](4.5,8){$p_{11}$}
\psline{<-}(4,8)(2,8) \uput[u](2.5,8){$p_{12}$}

\psline[linewidth=2.5pt]{->}(2,8)(2,6)
\pscurve{->}(2,6)(1.4,4)(2,2)
\psline{->}(2,4)(2,2)
\psline{->}(2,2)(2,1)
\pscurve{->}(4,8)(3.4,6)(4,4)
\psline[linewidth=2.5pt]{->}(4,6)(4,4)
\psline{->}(4,4)(4,1)
\psline{->}(6,8)(6,6)
\psline{->}(6,6)(6,3.02)
\psline{->}(8,8)(8,7)

\psdots[dotstyle=*,linecolor=black](2,8)
\psdots[dotstyle=*,linecolor=black](2,6)
\psdots[dotsize=9pt 0,dotstyle=*,linecolor=black](2,4)
\psdots[dotstyle=*,linecolor=black](2,2)
\psdots[dotstyle=*,linecolor=black](4,4)
\psdots[dotsize=9pt 0,dotstyle=*,linecolor=black](4,6)
\psdots[dotstyle=*,linecolor=black](4,8)
\psdots[dotstyle=*,linecolor=black](6,8)
\psdots[dotstyle=*,linecolor=black](8,8)
\psdots[dotstyle=*,linecolor=black](6,6)
\psdots[dotstyle=*,linecolor=black](9,8)


\uput[r](9.08,8.12){\large{1}}
\psdots[dotstyle=*,linecolor=black](8,7)
\uput[r](8.08,7.12){\large{2}}
\psdots[dotstyle=*,linecolor=black](7,6)
\uput[r](7.08,6.12){\large{3}}
\psdots[dotstyle=*,linecolor=black](7,4)
\uput[r](7.08,4.12){\large{4}}
\psdots[dotstyle=*,linecolor=black](6,3.02)
\uput[r](6.08,3.14){\large{5}}
\psdots[dotstyle=*,linecolor=black](5,2)
\uput[r](5.08,2.12){\large{6}}
\psdots[dotstyle=*,linecolor=black](4,1)
\uput[r](4.08,1.12){\large{7}}
\psdots[dotstyle=*,linecolor=black](2,1)
\uput[r](2.08,1.12){\large{8}}

\psdots[dotstyle=*,linecolor=black](0,8)
\uput[l](0.08,8.12){\large{$1'$}}
\psdots[dotstyle=*,linecolor=black](0,6)
\uput[l](0.08,6.12){\large{$3'$}}
\psdots[dotstyle=*,linecolor=black](0,4)
\uput[l](0.08,4.12){\large{$4'$}}
\psdots[dotstyle=*,linecolor=black](0,2)
\uput[l](0.08,2.12){\large{$6'$}}

\psline[linewidth=2.5pt]{->}(0,8)(2,8)
\psline{->}(0,6)(2,6)
\psline{->}(0,4)(2,4)
\psline{->}(0,2)(2,2)

\end{pspicture*}\quad\begin{pspicture*}(-1.5,-0.5)(21.5,8.5)

\psline{-}(-0.5,8)(20.5,8) \uput[l](-0.5,8){$1$}\uput[r](20.5,8){$1$}
\psline{-}(-0.5,7)(20.5,7) \uput[l](-0.5,7){$2$}\uput[r](20.5,7){$2$}
\psline{-}(-0.5,6)(20.5,6) \uput[l](-0.5,6){$3$}\uput[r](20.5,6){$3$}
\psline[linewidth=2.5pt]{-}(-0.5,5)(3,5)\pscurve[linewidth=2.5pt]{-}(3,5)(3.25,4.95)(3.75, 4.5)(4.25,4.05)(4.5,4)\psline[linewidth=2.5pt]{-}(4.5,4)(9.75,4)\psline{-}(9.75,4)(11,4)\pscurve{-}(11,4)(11.25,4.05)(11.75, 4.5)(12.25,4.95)(12.5,5)\psline{-}(12.5,5)(20.5,5) \uput[l](-0.5,5){$4$}\uput[r](20.5,5){$4$}
\psline{-}(-0.5,4)(3,4)\pscurve{-}(3,4)(3.25,4.05)(3.75, 4.5)(4.25,4.95)(4.5,5)\psline{-}(4.5,5)(10.75,5)\psline[linewidth=2.5pt]{-}(10.75,5)(11,5)\pscurve[linewidth=2.5pt]{-}(11,5)(11.25,4.95)(11.75, 4.5)(12.25,4.05)(12.5,4)\psline[linewidth=2.5pt]{-}(12.5,4)(20.5,4) \uput[l](-0.5,4){$5$}\uput[r](20.5,4){$5$}
\psline{-}(-0.5,3)(0,3)\pscurve{-}(0,3)(0.25,2.95)(0.75, 2.5)(1.25,2.05)(1.5,2)\psline{-}(1.5,2)(7.25,2)\pscurve{-}(7.25,2)(7.5,2.05)(8, 2.5)(8.5,2.95)(8.75,3)\psline{-}(8.75,3)(20.5,3) \uput[l](-0.5,3){$6$}\uput[r](20.5,3){$6$}
\psline{-}(-0.5,2)(0,2)\pscurve{-}(0,2)(0.25,2.05)(0.75, 2.5)(1.25,2.95)(1.5,3)\psline{-}(1.5,3)(7.25,3)\pscurve{-}(7.25,3)(7.5,2.95)(8, 2.5)(8.5,2.05)(8.75,2)\psline{-}(8.75,2)(20.5,2) \uput[l](-0.5,2){$7$}\uput[r](20.5,2){$7$}
\psline{-}(-0.5,1)(20.5,1) \uput[l](-0.5,1){$8$}\uput[r](20.5,1){$8$}

\uput[u](3,3.8){\footnotesize$-1$}
\uput[u](12.5,3.8){\footnotesize$-1$}
\uput[u](0,1.8){\footnotesize$-1$}
\uput[u](8.75,1.8){\footnotesize$-1$}

\psline{-}(2,2)(3,1)\uput[l](2.5,1.5){$p_2$}
\psline{-}(5,4)(6,3)\uput[l](5.5,3.5){$p_4$}
\psline{-}(8.75,6)(9.75,5)\uput[l](9.25,5.5){$p_6$}
\psline{-}(13,4)(14,3)\uput[l](13.5,3.5){$p_8$}
\psline{-}(14,8)(15,7)\uput[l](14.5,7.5){$p_9$}
\psline{-}(15.5,7)(16.5,6)\uput[l](16,6.5){$p_{10}$}
\psline{-}(17,6)(18,5)\uput[l](17.5,5.5){$p_{11}$}
\psline{-}(18.5,5)(19.5,4)\uput[l](19,4.5){$p_{12}$}

\psline{-}(6,2)(7,3) \uput[l](6.5,2.5){$m_5$}
\psline[linewidth=2.5pt]{-}(9.75,4)(10.75,5) \uput[l](10.25,4.5){$m_7$}

\end{pspicture*}
\end{center}
\caption{Pseudopaths in the modified network and their corresponding routes in the chip network.}
\label{fig:pseudo-chip-eg}
\end{figure}


\section{Proof of the main result}\label{sec:proof}

Let $D$ be a diagram which contains $t$ pluses and $u$ black stones.
In Section \ref{sec:rational} we will define an isomorphism
$$\Psi = \Psi_D: (\K^*)^t \times \K^u \longrightarrow
(\K^*)^t \times \K^u$$
which maps each
parameter from the weight matrix of the network $N_D$ to
a Laurent monomial in the
parameters used in the MR-matrix $M=M_D$.
Then in Section \ref{sec:row}, we will show that after applying
$\Psi$,
our network parameterization of the network
component $\mathcal{R}_D$ coincides with the corresponding
MR-parameterization of the projected Deodhar component $\mathcal{P}_{\v,\w}$.
Combining this fact with
Proposition \ref{p:parameterization} yields
Theorem \ref{th:main}.

\subsection{A rational transformation of parameters}\label{sec:rational}

\begin{definition}\label{def:rationalmap}
Let $D$ be a diagram, and
let $b_0$ be a box of $D$ containing a $+$ or $\smallblack$.
Let $b_1$ be the nearest box to the right of $b_0$ which
contains a $+$ (if it exists). Let $R_t$ be the set
of boxes in the same row as $b_0$ which are to the right
of $b_0$ and left of $b_1$.  Let $R_{\ell}$ be the set of
boxes in the same column as $b_0$ and below $b_0$.
If $b_1$ exists, let $R_{r}$ be the set of boxes
in the same column as $b_1$ and below $b_1$ (otherwise
$R_r = \emptyset$).
(See Figure \ref{fig:rationalmap}.)
Let $R_r^+$ (resp. $R_{\ell}^+$) be the set of boxes
in $R_r$ (resp. $R_{\ell}$) containing a $+$.
Let $R=R_r \cup R_{\ell} \cup R_t$ and
let $R^{\bullet}$ be the set of boxes in $R$ containing a $\smallblack$.

If $b_0$ contains a $+$, then define
$$\Psi(a_{b_0}) =
\frac{(-1)^{|R^{\bullet}|} \prod_{b\in R_r^+} p_b}
{p_{b_0} \prod_{b\in R_{\ell}^+} p_b}.$$
And if $b_0$ contains a $\smallblack$, then define
$$\Psi(c_{b_0}) =
\frac{m_{b_0}  (-1)^{|R^{\bullet}|} \prod_{b\in R_r^+} p_b}
{\prod_{b\in R_{\ell}^+} p_b}.$$
We also extend the definition of $\Psi$ to all polynomials
in the $a_b$'s and $c_b$'s by requiring it to be a  ring homomorphism.
\end{definition}

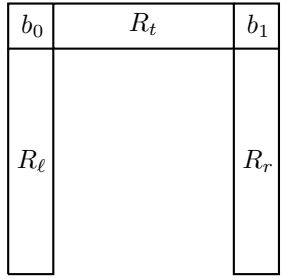
\begin{figure}[ht]
\begin{center}
\psset{unit=.3cm,dotstyle=o,dotsize=5pt 0,linewidth=0.8pt,arrowsize=3pt 2,arrowinset=0.25}
\begin{pspicture*}(0,0)(14,14)
\psline{-}(1,1)(3,1)(3,11)(1,11)(1,1)\uput[r](0.8,6){$R_\ell$}
\psline{-}(1,11)(1,13)(3,13) \uput[r](1,12){$b_0$}
\psline{-}(3,13)(11,13)(11,11)(3,11)(3,13) \uput[u](7,11){$R_t$}
\psline{-}(11,13)(13,13)(13,11) \uput[r](11,12){$b_1$}
\psline{-}(11,11)(13,11)(13,1)(11,1)(11,11) \uput[r](10.8,6){$R_r$}
\end{pspicture*}
\end{center}
\caption{The definition of $b_0$, $b_1$, $R_t$, $R_{\ell}$, and $R_r$.}
\label{fig:rationalmap}
\end{figure}

\begin{remark}
Clearly $\Psi = \Psi_D$ is an isomorphism from
$(\K^*)^t \times \K^u$ to itself.
\end{remark}

\begin{example}
Consider the network from Figure \ref{fig:network}
(shown again in Figure \ref{fig:paths-network}).
Then we have
$\Psi(a_2) = \frac{1}{p_2}$,
$\Psi(a_4) = \frac{1}{p_4}$,
$\Psi(c_5) = \frac{m_5}{p_2}$,
$\Psi(a_6) = \frac{1}{p_6}$,
$\Psi(c_7) = \frac{m_7}{p_4}$,
$\Psi(a_8) = \frac{1}{p_2 p_8}$,
$\Psi(a_9) = \frac{1}{p_9}$,
$\Psi(a_{10}) = \frac{1}{p_6 p_{10}}$,
$\Psi(a_{11}) = \frac{-p_6}{p_4 p_{11}}$,
and $\Psi(a_{12}) = \frac{p_4}{p_2 p_8 p_{12}}$.
\end{example}

From Definition \ref{def:rationalmap}, one may easily deduce
a formula for $\Psi(w(P))$, where $P$ is a path in $N_D$.
We will state this formula in terms of the Go-diagram.
Note that one may identify a path $P$ in $N_D$ with a
connected sequence $\P$
of boxes in the Go-diagram, where any two adjacent
boxes must share a side.
We call a box $b$ in $\P$ a
\emph{corner box} if the path
$P$ turns from west to south, or from south to west,
at the vertex associated to $b$.
(Such a box $b$ in $D$ must contain a $+$ or $\smallblack$.)

\begin{proposition}\label{prop:Psi}
Let $P$ be a path in the network $N_D$, which we identify with
a sequence $\P$ of boxes in the Go-diagram $D$.
Among the boxes in $\P$, let $B_1$ denote the subset containing a $+$;
let $B_2$ denote the subset containing a $\smallblack$  which are
corner boxes of $\P$; and let $B_3$ denote the subset containing
a $\smallblack$ which are not corner boxes of $\P$.
Then
\begin{equation}
\label{eq:rational}
\Psi(w(P))  = \frac{(-1)^{|B_3|} \prod_{b\in B_2} m_b}
                     {\prod_{b\in B_1} p_b}.
\end{equation}
\end{proposition}
\begin{proof}
We leave this as an exercise for the reader.  It is a simple
consequence of Definition \ref{def:rationalmap}.
\end{proof}

\begin{figure}
\psset{unit=.4cm,dotstyle=o,dotsize=5pt 0,linewidth=0.8pt,arrowsize=2pt 2,arrowinset=0.25}
\begin{pspicture*}(0,0.45)(10,9)
\psline{->}(5,2)(2,2) 
\psline{->}(7,4)(4,4) 
\psline[linewidth=2.5pt]{->}(4,4)(2,4) 
\psline{->}(7,6)(6,6) 
\psline{->}(6,6)(4,6) 
\pscurve{->}(6,6)(4,7)(2,6) 
\psline[linewidth=2.5pt]{->}(9,8)(8,8) 
\psline[linewidth=2.5pt]{->}(8,8)(6,8) 
\psline[linewidth=2.5pt]{->}(6,8)(4,8) 
\psline{->}(4,8)(2,8) 

\psline{->}(2,8)(2,6)
\pscurve{->}(2,6)(1.4,4)(2,2)
\psline[linewidth=2.5pt]{->}(2,4)(2,2)
\psline[linewidth=2.5pt]{->}(2,2)(2,1)
\pscurve[linewidth=2.5pt]{->}(4,8)(3.4,6)(4,4)
\psline{->}(4,6)(4,4)
\psline{->}(4,4)(4,1)
\psline{->}(6,8)(6,6)
\psline{->}(6,6)(6,3.02)
\psline{->}(8,8)(8,7)

\psdots[dotstyle=*,linecolor=black](2,8)
\psdots[dotstyle=*,linecolor=black](2,6)
\psdots[dotsize=9pt 0,dotstyle=*,linecolor=black](2,4)
\psdots[dotstyle=*,linecolor=black](2,2)
\psdots[dotstyle=*,linecolor=black](4,4)
\psdots[dotsize=9pt 0,dotstyle=*,linecolor=black](4,6)
\psdots[dotstyle=*,linecolor=black](4,8)
\psdots[dotstyle=*,linecolor=black](6,8)
\psdots[dotstyle=*,linecolor=black](8,8)
\psdots[dotstyle=*,linecolor=black](6,6)

\psdots[dotstyle=*,linecolor=black](9,8)
\uput[r](9.08,8.12){\large{1}}
\psdots[dotstyle=*,linecolor=black](8,7)
\uput[r](8.08,7.12){\large{2}}
\psdots[dotstyle=*,linecolor=black](7,6)
\uput[r](7.08,6.12){\large{3}}
\psdots[dotstyle=*,linecolor=black](7,4)
\uput[r](7.08,4.12){\large{4}}
\psdots[dotstyle=*,linecolor=black](6,3)
\uput[r](6.08,3.12){\large{5}}
\psdots[dotstyle=*,linecolor=black](5,2)
\uput[r](5.08,2.12){\large{6}}
\psdots[dotstyle=*,linecolor=black](4,1)
\uput[r](4.08,1.12){\large{7}}
\psdots[dotstyle=*,linecolor=black](2,1)
\uput[r](2.08,1.12){\large{8}}
\end{pspicture*}
\quad\setlength{\unitlength}{1.9pt}\begin{picture}(46,46)(4,0)
  \put(5,45){\line(1,0){40}}
  \put(5,35){\line(1,0){40}}
  \put(5,25){\line(1,0){30}}
  \put(5,15){\line(1,0){30}}
  \put(5,5){\line(1,0){20}}
  \put(5,5){\line(0,1){40}}
  \put(15,5){\line(0,1){40}}
  \put(25,5){\line(0,1){40}}
  \put(35,15){\line(0,1){30}}
  \put(45,35){\line(0,1){10}}

  \put(7,38){\scalebox{1.7}{$+$}}
  \put(17,38){\scalebox{1.7}{$+$}}
  \put(27,38){\scalebox{1.7}{$+$}}
  \put(37,38){\scalebox{1.7}{$+$}}
  \put(7,28){\scalebox{1.7}{$+$}}
  \put(18,30){\hskip0.15cm\circle*{5}}
  \put(27,28){\scalebox{1.7}{$+$}}
  \put(8,20){\hskip0.15cm\circle*{5}}
  \put(17,18){\scalebox{1.7}{$+$}}
  \put(28,20){\hskip0.15cm\circle{5}}
  \put(7,8){\scalebox{1.7}{$+$}}
  \put(18,10){\hskip0.15cm\circle{5}}

\linethickness{0.7mm}
\put(15,45){\line(1,0){30}}
\put(15,15){\line(0,1){30}}
\put(5,25){\line(1,0){20}}
\put(5,5){\line(0,1){20}}
\put(15,35){\line(1,0){30}}
\put(25,15){\line(0,1){30}}
\put(5,15){\line(1,0){20}}
\put(15,5){\line(0,1){20}}
\put(5,5){\line(1,0){10}}
\put(35,35){\line(0,1){10}}
\put(45,35){\line(0,1){10}}
\end{picture}\quad
\begin{pspicture*}(0,0.45)(10,9)
\psline{->}(5,2)(2,2) 
\psline{->}(7,4)(4,4) 
\psline[linewidth=2.5pt]{->}(4,4)(2,4) 
\psline{->}(7,6)(6,6) 
\psline[linewidth=2.5pt]{->}(6,6)(4,6) 
\pscurve{->}(6,6)(4,7)(2,6) 
\psline[linewidth=2.5pt]{->}(9,8)(8,8) 
\psline[linewidth=2.5pt]{->}(8,8)(6,8) 
\psline{->}(6,8)(4,8) 
\psline{->}(4,8)(2,8) 

\psline{->}(2,8)(2,6)
\pscurve{->}(2,6)(1.4,4)(2,2)
\psline[linewidth=2.5pt]{->}(2,4)(2,2)
\psline[linewidth=2.5pt]{->}(2,2)(2,1)
\pscurve{->}(4,8)(3.4,6)(4,4)
\psline[linewidth=2.5pt]{->}(4,6)(4,4)
\psline{->}(4,4)(4,1)
\psline[linewidth=2.5pt]{->}(6,8)(6,6)
\psline{->}(6,6)(6,3.02)
\psline{->}(8,8)(8,7)

\psdots[dotstyle=*,linecolor=black](2,8)
\psdots[dotstyle=*,linecolor=black](2,6)
\psdots[dotsize=9pt 0,dotstyle=*,linecolor=black](2,4)
\psdots[dotstyle=*,linecolor=black](2,2)
\psdots[dotstyle=*,linecolor=black](4,4)
\psdots[dotsize=9pt 0,dotstyle=*,linecolor=black](4,6)
\psdots[dotstyle=*,linecolor=black](4,8)
\psdots[dotstyle=*,linecolor=black](6,8)
\psdots[dotstyle=*,linecolor=black](8,8)
\psdots[dotstyle=*,linecolor=black](6,6)

\psdots[dotstyle=*,linecolor=black](9,8)
\uput[r](9.08,8.12){\large{1}}
\psdots[dotstyle=*,linecolor=black](8,7)
\uput[r](8.08,7.12){\large{2}}
\psdots[dotstyle=*,linecolor=black](7,6)
\uput[r](7.08,6.12){\large{3}}
\psdots[dotstyle=*,linecolor=black](7,4)
\uput[r](7.08,4.12){\large{4}}
\psdots[dotstyle=*,linecolor=black](6,3)
\uput[r](6.08,3.12){\large{5}}
\psdots[dotstyle=*,linecolor=black](5,2)
\uput[r](5.08,2.12){\large{6}}
\psdots[dotstyle=*,linecolor=black](4,1)
\uput[r](4.08,1.12){\large{7}}
\psdots[dotstyle=*,linecolor=black](2,1)
\uput[r](2.08,1.12){\large{8}}
\end{pspicture*}\quad\begin{picture}(46,46)(4,0)
  \put(5,45){\line(1,0){40}}
  \put(5,35){\line(1,0){40}}
  \put(5,25){\line(1,0){30}}
  \put(5,15){\line(1,0){30}}
  \put(5,5){\line(1,0){20}}
  \put(5,5){\line(0,1){40}}
  \put(15,5){\line(0,1){40}}
  \put(25,5){\line(0,1){40}}
  \put(35,15){\line(0,1){30}}
  \put(45,35){\line(0,1){10}}

  \put(7,38){\scalebox{1.7}{$+$}}
  \put(17,38){\scalebox{1.7}{$+$}}
  \put(27,38){\scalebox{1.7}{$+$}}
  \put(37,38){\scalebox{1.7}{$+$}}
  \put(7,28){\scalebox{1.7}{$+$}}
  \put(18,30){\hskip0.15cm\circle*{5}}
  \put(27,28){\scalebox{1.7}{$+$}}
  \put(8,20){\hskip0.15cm\circle*{5}}
  \put(17,18){\scalebox{1.7}{$+$}}
  \put(28,20){\hskip0.15cm\circle{5}}
  \put(7,8){\scalebox{1.7}{$+$}}
  \put(18,10){\hskip0.15cm\circle{5}}

\linethickness{0.7mm}
\put(25,45){\line(1,0){20}}
\put(15,35){\line(1,0){30}}
\put(5,25){\line(1,0){30}}
\put(5,15){\line(1,0){20}}
\put(5,5){\line(1,0){10}}

\put(5,5){\line(0,1){20}}
\put(15,5){\line(0,1){30}}
\put(25,15){\line(0,1){30}}
\put(35,25){\line(0,1){20}}
\put(45,35){\line(0,1){10}}
\end{picture}
\caption{Two paths in $N_D$ and their corresponding sequences of boxes in $D$.}
\label{fig:paths-network}
\end{figure}
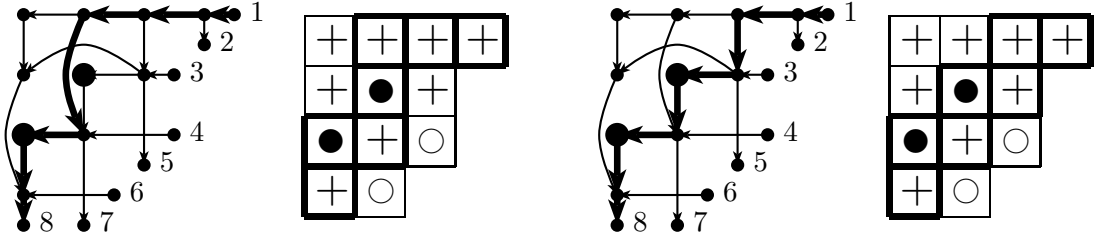

\begin{example}
Let $P_1$ and $P_2$ be the paths shown in Figure \ref{fig:paths-network}.
Then $\Psi(w(P_1)) = \Psi(a_9 a_{10} a_{11} c_5) =
\frac{-m_5}{p_2 p_4 p_9 p_{10} p_{11}}$ and
$\Psi(w(P_2)) = \Psi(a_9 a_{10} c_5 c_7) = \frac{m_5 m_7}{p_2 p_4 p_6 p_9 p_{10}}.$
\end{example}

\subsection{Applying row operations to the rescaled MR-matrix}\label{sec:row}

\begin{theorem}\label{th:row}
Let $D$ be a diagram, let $M = M_D$ be the corresponding MR-matrix,
and let $L$ be the matrix we obtain by putting $M$ into
reduced row-echelon form.
Let $W_D = (W_{ij})$ be the weight matrix associated to $D$,
and let $\Psi(W_D)$ be the matrix obtained from $W_D$ by applying
the rational map $\Psi$ to each entry.
Then $\Psi(W_D) = L$.
\end{theorem}

\begin{proof}
To prove Theorem \ref{th:row},
we start by considering the rescaled MR-matrix $\widetilde{M}$.
Its rows are indexed by
the set $i_1, \dots, i_k$, the set of sources of the network $N_D$,
and the leftmost nonzero entry in every row
is a $1$.  Moreover, by Lemma \ref{lem:Mproperties},
the $1$ in row $i_{\ell}$ is located in column $i_{\ell}$.
The entries of the
reduced row echelon matrix $L$ obtained from $\widetilde{M}$
are given by the formula
\begin{equation}\label{eq:L-entry}
L_{st} = \widetilde{M}_{st} + \sum_{s<j_1 < \dots < j_r < t}
(-1)^r \widetilde{M}_{s j_1} \widetilde{M}_{j_1 j_2} \dots
\widetilde{M}_{j_r t},
\end{equation}
where the sum ranges over all nonempty subsets
$\{j_1,\dots,j_r\} \subset \{i_1,\dots,i_k\}$ of sources
of the network between $s$ and $t$.

By Theorem \ref{th:MR-matrix},
the entry $\widetilde{M}_{st}$ equals
$\sum_{\widetilde{P}} w(\widetilde{P})$, where the sum is over all
pseudopaths in the network $N_D$ from the source $s$ to the boundary
vertex $t$.
Therefore the right-hand side of \eqref{eq:L-entry} can be interpreted
as a generating function for all
concatenations of pseudopaths, where the first pseudopath starts
at $s$, and the last pseudopath ends at $t$.

Let us identify a pseudopath with its collection of directed edges.
Given a set of pseudopaths on $N_D$, we define
its \emph{signed union} to be the union of directed edges
that one obtains by taking the multiset of all directed edges
in the pseudopaths, and then cancelling pairs
which traverse the same edge but in opposite directions.
We define the \emph{weight} of a set of pseudopaths
to be the product of the weights of each pseudopath in the set.

Our goal is to show that after cancellation, the only terms which
survive on the right-hand  side of \eqref{eq:L-entry} correspond
to concatenations of pseudopaths whose signed union forms a
directed path (and not merely a pseudopath) from $s$ to $t$ in $N_D$.
This will allow us to relate \eqref{eq:L-entry} to $W_{st}$,
which is defined as a sum over all paths from $s$ to $t$ in $N_D$.

\begin{definition}
A \emph{U-turn} pseudopath in a network is a pseudopath
whose sequence of steps has the form $(WW^*S)E^*$.
\end{definition}

First note that any path $P$ in a network $N_D$ has a
unique decomposition as a signed union of
U-turn pseudopaths.  Moreover, the products of the
weights of the pseudopaths is precisely the quantity
on the right-hand side of \eqref{eq:rational}.
See Figure~\ref{fig:decomposition1} for an example
of the decomposition into U-turn pseudopaths.

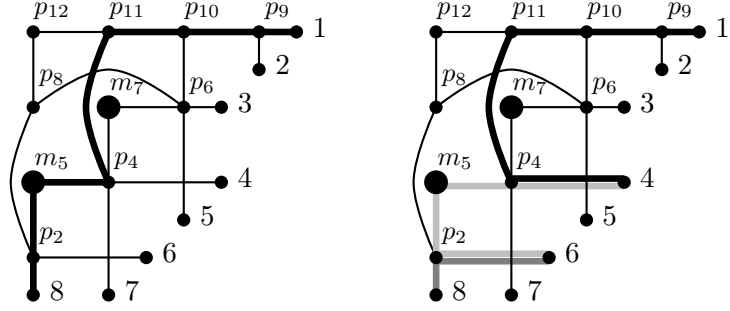
\begin{figure}[ht]
\begin{center}
\psset{unit=.5cm,dotstyle=o,dotsize=5pt 0,linewidth=0.8pt,arrowsize=3pt 2,arrowinset=0.25}
\begin{pspicture*}(0,0)(10,9)
\psline{-}(5,2)(2,2) \uput[u](2.5,2){$p_2$}
\psline{-}(7,4)(4,4) \uput[u](4.5,4){$p_4$}
\psline[linewidth=2.5pt]{-}(4,4)(2,4) \uput[u](2.5,4){$m_5$}
\psline{-}(7,6)(6,6) \uput[u](6.5,6){$p_6$}
\psline{-}(6,6)(4,6) \uput[u](4.5,6){$m_7$}
\pscurve{-}(6,6)(4,7)(2,6) \uput[u](2.5,6.2){$p_8$}
\psline[linewidth=2.5pt]{-}(9,8)(8,8) \uput[u](8.5,8){$p_9$}
\psline[linewidth=2.5pt]{-}(8,8)(6,8) \uput[u](6.5,8){$p_{10}$}
\psline[linewidth=2.5pt]{-}(6,8)(4,8) \uput[u](4.5,8){$p_{11}$}
\psline{-}(4,8)(2,8) \uput[u](2.5,8){$p_{12}$}

\psline{-}(2,8)(2,6)
\pscurve{-}(2,6)(1.4,4)(2,2)
\psline[linewidth=2.5pt]{-}(2,4)(2,2)
\psline[linewidth=2.5pt]{-}(2,2)(2,1)
\pscurve[linewidth=2.5pt]{-}(4,8)(3.4,6)(4,4)
\psline{-}(4,6)(4,4)
\psline{-}(4,4)(4,1)
\psline{-}(6,8)(6,6)
\psline{-}(6,6)(6,3.02)
\psline{-}(8,8)(8,7)

\psdots[dotstyle=*,linecolor=black](2,8)
\psdots[dotstyle=*,linecolor=black](2,6)
\psdots[dotsize=9pt 0,dotstyle=*,linecolor=black](2,4)
\psdots[dotstyle=*,linecolor=black](2,2)
\psdots[dotstyle=*,linecolor=black](4,4)
\psdots[dotsize=9pt 0,dotstyle=*,linecolor=black](4,6)
\psdots[dotstyle=*,linecolor=black](4,8)
\psdots[dotstyle=*,linecolor=black](6,8)
\psdots[dotstyle=*,linecolor=black](8,8)
\psdots[dotstyle=*,linecolor=black](6,6)
\psdots[dotstyle=*,linecolor=black](9,8)

\uput[r](9.08,8.12){\large{1}}
\psdots[dotstyle=*,linecolor=black](8,7)
\uput[r](8.08,7.12){\large{2}}
\psdots[dotstyle=*,linecolor=black](7,6)
\uput[r](7.08,6.12){\large{3}}
\psdots[dotstyle=*,linecolor=black](7,4)
\uput[r](7.08,4.12){\large{4}}
\psdots[dotstyle=*,linecolor=black](6,3)
\uput[r](6.08,3.12){\large{5}}
\psdots[dotstyle=*,linecolor=black](5,2)
\uput[r](5.08,2.12){\large{6}}
\psdots[dotstyle=*,linecolor=black](4,1)
\uput[r](4.08,1.12){\large{7}}
\psdots[dotstyle=*,linecolor=black](2,1)
\uput[r](2.08,1.12){\large{8}}

\end{pspicture*}\quad\begin{pspicture*}(0,0)(10,9)
\psline[linewidth=2.5pt, linecolor=lightgray]{-}(5,2.1)(2,2.1) \uput[u](2.5,2){$p_2$}
\psline[linewidth=2.5pt, linecolor=gray]{-}(5,1.9)(2,1.9)
\psline[linewidth=2.5pt]{-}(7,4.1)(4,4.1) \uput[u](4.5,4){$p_4$}
\psline[linewidth=2.5pt, linecolor=lightgray]{-}(7,3.9)(4,3.9)
\psline[linewidth=2.5pt, linecolor=lightgray]{-}(4,3.9)(2,3.9) \uput[u](2.5,4){$m_5$}
\psline{-}(7,6)(6,6) \uput[u](6.5,6){$p_6$}
\psline{-}(6,6)(4,6) \uput[u](4.5,6){$m_7$}
\pscurve{-}(6,6)(4,7)(2,6) \uput[u](2.5,6.2){$p_8$}
\psline[linewidth=2.5pt]{-}(9,8)(8,8) \uput[u](8.5,8){$p_9$}
\psline[linewidth=2.5pt]{-}(8,8)(6,8) \uput[u](6.5,8){$p_{10}$}
\psline[linewidth=2.5pt]{-}(6,8)(4,8) \uput[u](4.5,8){$p_{11}$}
\psline{-}(4,8)(2,8) \uput[u](2.5,8){$p_{12}$}

\psline{-}(2,8)(2,6)
\pscurve{-}(2,6)(1.4,4)(2,2)
\psline[linewidth=2.5pt, linecolor=lightgray]{-}(2,4)(2,2)
\psline[linewidth=2.5pt, linecolor=gray]{-}(2,2)(2,1)
\pscurve[linewidth=2.5pt]{-}(4,8)(3.4,6)(4,4)
\psline{-}(4,6)(4,4)
\psline{-}(4,4)(4,1)
\psline{-}(6,8)(6,6)
\psline{-}(6,6)(6,3)
\psline{-}(8,8)(8,7)

\psdots[dotstyle=*,linecolor=black](2,8)
\psdots[dotstyle=*,linecolor=black](2,6)
\psdots[dotsize=9pt 0,dotstyle=*,linecolor=black](2,4)
\psdots[dotstyle=*,linecolor=black](2,2)
\psdots[dotstyle=*,linecolor=black](4,4)
\psdots[dotsize=9pt 0,dotstyle=*,linecolor=black](4,6)
\psdots[dotstyle=*,linecolor=black](4,8)
\psdots[dotstyle=*,linecolor=black](6,8)
\psdots[dotstyle=*,linecolor=black](8,8)
\psdots[dotstyle=*,linecolor=black](6,6)
\psdots[dotstyle=*,linecolor=black](9,8)

\uput[r](9.08,8.12){\large{1}}
\psdots[dotstyle=*,linecolor=black](8,7)
\uput[r](8.08,7.12){\large{2}}
\psdots[dotstyle=*,linecolor=black](7,6)
\uput[r](7.08,6.12){\large{3}}
\psdots[dotstyle=*,linecolor=black](7,4)
\uput[r](7.08,4.12){\large{4}}
\psdots[dotstyle=*,linecolor=black](6,3)
\uput[r](6.08,3.12){\large{5}}
\psdots[dotstyle=*,linecolor=black](5,2)
\uput[r](5.08,2.12){\large{6}}
\psdots[dotstyle=*,linecolor=black](4,1)
\uput[r](4.08,1.12){\large{7}}
\psdots[dotstyle=*,linecolor=black](2,1)
\uput[r](2.08,1.12){\large{8}}

\end{pspicture*}
\end{center}
\caption{A path in $N_D$ and its decomposition into U-turn pseudopaths.}
\label{fig:decomposition1}
\end{figure}

This observation on the decomposition of paths may be generalized
to pseudopaths.  Consider
a pseudopath $\widetilde{P}$ which is not a path, and turns from south
to east precisely $q$ times (for $q \geq 1$).
Then for each $0 \leq r \leq q$,
there are ${q \choose r}$ decompositions of
$\widetilde{P}$
as a signed union of $r$ pseudopaths.  Moreover, each set of pseudopaths
forming a decomposition of $\widetilde{P}$ has the same weight.
See Figure \ref{fig:decomposition2} for an example of
all decompositions of a pseudopath as a signed union of pseudopaths.
It is easy to check that each decomposition has the same weight.
\begin{figure}[ht]
\begin{center}
\psset{unit=.4cm,dotstyle=o,dotsize=5pt 0,linewidth=2pt,arrowsize=3pt 2,arrowinset=0.25}
\begin{pspicture*}(-0.2,-0.2)(6.2,6.2)
\psline{->}(6,6)(0,6)
\psline{->}(0,6)(0,4)
\psline{->}(0,4)(2,4)
\psline{->}(2,4)(2,2)
\psline{->}(2,2)(4,2)
\psline{->}(4,2)(4,0)
\uput[r](0,1){$r=0$}
\end{pspicture*}\quad\begin{pspicture*}(-0.2,-0.2)(6.2,6.2)
\psline{->}(6,6)(0,6)
\psline{->}(0,6)(0,4)
\psline{->}(0,4.1)(6,4.1)

\psline[linecolor=lightgray]{->}(6,3.9)(2,3.9)
\psline[linecolor=lightgray]{->}(2,4)(2,2)
\psline[linecolor=lightgray]{->}(2,2)(4,2)
\psline[linecolor=lightgray]{->}(4,2)(4,0)
\uput[r](0,1){$r=1$}
\end{pspicture*}\quad\begin{pspicture*}(-0.2,-0.2)(6.2,6.2)
\psline{->}(6,6)(0,6)
\psline{->}(0,6)(0,4)
\psline{->}(0,4)(2,4)
\psline{->}(2,4)(2,2)
\psline{->}(2,2.1)(6,2.1)

\psline[linecolor=lightgray]{->}(6,1.9)(4,1.9)
\psline[linecolor=lightgray]{->}(4,2)(4,0)
\uput[r](0,1){$r=1$}
\end{pspicture*}\quad\begin{pspicture*}(-0.2,-0.2)(6.2,6.2)
\psline{->}(6,6)(0,6)
\psline{->}(0,6)(0,4)
\psline{->}(0,4.1)(6,4.1)

\psline[linecolor=lightgray]{->}(6,3.9)(2,3.9)
\psline[linecolor=lightgray]{->}(2,4)(2,2)
\psline[linecolor=lightgray]{->}(2,2.1)(6,2.1)

\psline[linecolor=darkgray]{->}(6,1.9)(4,1.9)
\psline[linecolor=darkgray]{->}(4,2)(4,0)
\uput[r](0,1){$r=2$}
\end{pspicture*}
\end{center}
\caption{An example of the $2^q$ decompositions of a path in $N_D$ into pseudopaths.}
\label{fig:decomposition2}
\end{figure}
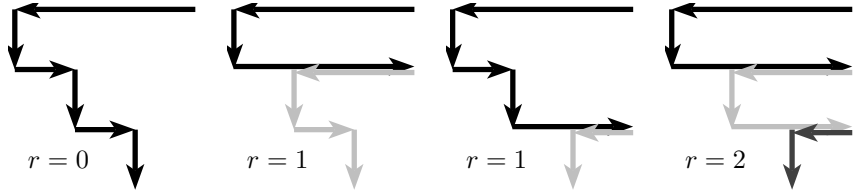

Note that since all decompositions of a pseudopath have the
same weight, and because
$${q\choose 0} - {q \choose 1} + {q \choose 2} - \dots \pm {q \choose q}=0
\text{ for }q \geq 1,$$
the net contribution of the corresponding concatenations of pseudopaths
in \eqref{eq:L-entry} is $0$.

More generally, a term on the right-hand side of \eqref{eq:L-entry}
corresponds to a concatenation of pseudopaths (whose signed union
may not be a pseudopath).  However, just as before, one may
decompose the signed union $\widetilde{P}$ in $2^q$ ways,
where $q$ is the number of times that $\widetilde{P}$ turns from
south to east.  And again, for $q\geq 1$, the
net contribution of the corresponding concatenations of pseudopaths
in \eqref{eq:L-entry} is $0$.
Therefore when one interprets the right-hand side of \eqref{eq:L-entry}
as a sum over concatenations of pseudopaths, the only terms that
are not cancelled are the terms corresponding to concatenations
of pseudopaths whose signed union is a directed path.

\begin{example}
In Figure \ref{fig:cancelling-stacks}, the
left and middle diagrams show a pseudopath from $1$ to $5$,
and a concatenation
of two pseudopaths from $1$ to $5$, whose weights cancel each other.  The right
diagram shows a path from $1$ to $5$, written as a signed union of two pseudopaths,
which will not be cancelled by any other term.
\end{example}

It is a simple exercise to verify that the absolute value of
the weight of such a
concatenation is precisely the absolute value of the right-hand side of
\eqref{eq:rational}.  We also need to verify that the signs
agree.   Once we do this, then
since $W_{st}$ is the sum of all weights of paths in $N_D$
from $s$ to $t$,
$\Psi(W_{st})$ equals the corresponding term in
the expression for $L_{st}$ in
\eqref{eq:L-entry}, so the proof is done.

We now check that the signs agree.  More specifically,
consider a concatenation of $(r+1)$ pseudopaths whose
signed union is a directed path $P$.
Using Definition \ref{def:pseudoweight},
the total sign associated to the concatenation of
pseudopaths from  \eqref{eq:L-entry}
is $(-1)^{r+b+w}$, where
$b$ (respectively $w$) is the number of black (resp. white)
stones that $P$ skips over in the horizontal (resp. vertical)
direction.  (Note that for the purpose of computing $b$ and $w$,
we can count stones skipped by $P$ here, as opposed to the set
of pseudopaths whose signed union is $P$.)
Meanwhile, using
Definition \ref{def:weightmatrix} and Proposition \ref{prop:Psi},
the total sign associated to the directed path $P$
in the expression $\Psi(W_{st})$ is
$(-1)^{q+|B_3|}$, where $q$ is the number of sources
in the network which are strictly between $s$ and $t$, and
$|B_3|$ is the number of non-corner black stones in the boxes
of the Go-diagram which $P$ traces out.
We need to show that
$(-1)^{r+b+w} = (-1)^{q+|B_3|}$.

Note that $|B_3|$ is the number of
black stones skipped either vertically
or horizontally by $P$.
Let $b_v$ (resp. $w_v$) denote the number of black stones
(resp. white stones) skipped vertically by $P$.
Similarly, let $b_h$ (resp. $w_h$) denote the number of black stones
(resp. white stones) skipped horizontally by $P$.
With this notation, we need to show that
$(-1)^{b_h+w_v+r} = (-1)^{b_h+b_v+q}$, i.e. that
$(-1)^{w_v+r} = (-1)^{b_v+q}$.

To prove this, we will show that $q-r = w_v+b_v$.
Note that $q-r$ is the number of sources strictly between
$s$ and $t$ which are \emph{not} sources of any pseudopath
in the pseudopath decomposition of $P$, i.e. which are
not in the set $\{j_1,\dots,j_r\}$ from
\eqref{eq:L-entry}.
Recall that the
pseudopath decomposition of $P$ is a U-turn pseudopath decomposition;
therefore, a source between $s$ and $t$ lies in
$\{j_1,\dots,j_r\}$ if and only if $P$ has a vertical edge
ending at a $+$ in this row.  Otherwise $P$ skips
a black or white stone in this row.  This proves
that $q-r = w_v+b_v$, and hence completes the proof of
Theorem \ref{th:row}.
\end{proof}

We have now shown that after an invertible transformation of
the parameters,
our network parameterization of
$\mathcal{R}_D$ coincides with the corresponding MR-parameterization
of the projected Deodhar component $\mathcal{P}_{\v,\w}$.
Combining this result with
Proposition \ref{p:parameterization} yields
Theorem \ref{th:main}.

Our proof yields the following statement.
\begin{corollary}\label{cor:2params}
Let $D$ be the Go-diagram associated to the distinguished subexpression
$\v$ of $\w$.  Then $\mathcal{R}_D = \mathcal{P}_{\v,\w}$ as subsets
of $Gr_{k,n}$.  Furthermore,
$$Gr_{k,n} = \bigsqcup_D \mathcal{R}_D,$$ where the union is over all
Go-diagrams $D$ contained in a $k \times (n-k)$ rectangle.
\end{corollary}

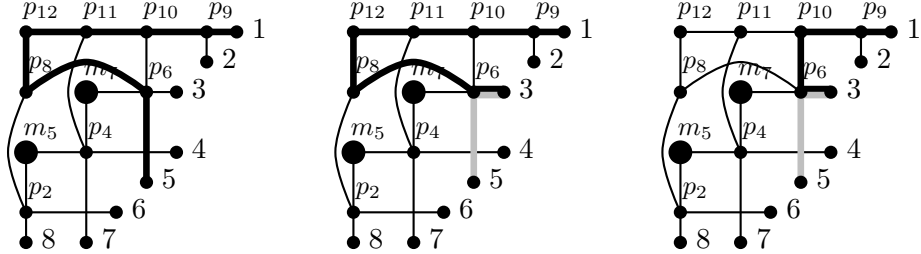
\begin{figure}[ht]
\begin{center}
\psset{unit=.4cm,dotstyle=o,dotsize=5pt 0,linewidth=0.8pt,arrowsize=3pt 2,arrowinset=0.25}
\begin{pspicture*}(0,0)(10,9)
\psline{-}(5,2)(2,2) \uput[u](2.5,2){$p_2$}
\psline{-}(7,4)(4,4) \uput[u](4.5,4){$p_4$}
\psline{-}(4,4)(2,4) \uput[u](2.5,4){$m_5$}
\psline{-}(7,6)(6,6) \uput[u](6.5,6){$p_6$}
\psline{-}(6,6)(4,6) \uput[u](4.5,6){$m_7$}
\pscurve[linewidth=2.5pt]{-}(6,6)(4,7)(2,6) \uput[u](2.5,6.2){$p_8$}
\psline[linewidth=2.5pt]{-}(9,8)(8,8) \uput[u](8.5,8){$p_9$}
\psline[linewidth=2.5pt]{-}(8,8)(6,8) \uput[u](6.5,8){$p_{10}$}
\psline[linewidth=2.5pt]{-}(6,8)(4,8) \uput[u](4.5,8){$p_{11}$}
\psline[linewidth=2.5pt]{-}(4,8)(2,8) \uput[u](2.5,8){$p_{12}$}

\psline[linewidth=2.5pt]{-}(2,8)(2,6)
\pscurve{-}(2,6)(1.4,4)(2,2)
\psline{-}(2,4)(2,2)
\psline{-}(2,2)(2,1)
\pscurve{-}(4,8)(3.4,6)(4,4)
\psline{-}(4,6)(4,4)
\psline{-}(4,4)(4,1)
\psline{-}(6,8)(6,6)
\psline[linewidth=2.5pt]{-}(6,6)(6,3)
\psline{-}(8,8)(8,7)

\psdots[dotstyle=*,linecolor=black](2,8)
\psdots[dotstyle=*,linecolor=black](2,6)
\psdots[dotsize=9pt 0,dotstyle=*,linecolor=black](2,4)
\psdots[dotstyle=*,linecolor=black](2,2)
\psdots[dotstyle=*,linecolor=black](4,4)
\psdots[dotsize=9pt 0,dotstyle=*,linecolor=black](4,6)
\psdots[dotstyle=*,linecolor=black](4,8)
\psdots[dotstyle=*,linecolor=black](6,8)
\psdots[dotstyle=*,linecolor=black](8,8)
\psdots[dotstyle=*,linecolor=black](6,6)
\psdots[dotstyle=*,linecolor=black](9,8)

\uput[r](9.08,8.12){\large{1}}
\psdots[dotstyle=*,linecolor=black](8,7)
\uput[r](8.08,7.12){\large{2}}
\psdots[dotstyle=*,linecolor=black](7,6)
\uput[r](7.08,6.12){\large{3}}
\psdots[dotstyle=*,linecolor=black](7,4)
\uput[r](7.08,4.12){\large{4}}
\psdots[dotstyle=*,linecolor=black](6,3)
\uput[r](6.08,3.12){\large{5}}
\psdots[dotstyle=*,linecolor=black](5,2)
\uput[r](5.08,2.12){\large{6}}
\psdots[dotstyle=*,linecolor=black](4,1)
\uput[r](4.08,1.12){\large{7}}
\psdots[dotstyle=*,linecolor=black](2,1)
\uput[r](2.08,1.12){\large{8}}

\end{pspicture*}\quad\begin{pspicture*}(0,0)(10,9)
\psline{-}(5,2)(2,2) \uput[u](2.5,2){$p_2$}
\psline{-}(7,4)(4,4) \uput[u](4.5,4){$p_4$}
\psline{-}(4,4)(2,4) \uput[u](2.5,4){$m_5$}
\psline[linewidth=2.5pt]{-}(7,6.1)(6,6.1) \uput[u](6.5,6){$p_6$}
\psline[linewidth=2.5pt, linecolor=lightgray]{-}(7,5.9)(6,5.9)
\psline{-}(6,6)(4,6) \uput[u](4.5,6){$m_7$}
\pscurve[linewidth=2.5pt]{-}(6,6)(4,7)(2,6) \uput[u](2.5,6.2){$p_8$}
\psline[linewidth=2.5pt]{-}(9,8)(8,8) \uput[u](8.5,8){$p_9$}
\psline[linewidth=2.5pt]{-}(8,8)(6,8) \uput[u](6.5,8){$p_{10}$}
\psline[linewidth=2.5pt]{-}(6,8)(4,8) \uput[u](4.5,8){$p_{11}$}
\psline[linewidth=2.5pt]{-}(4,8)(2,8) \uput[u](2.5,8){$p_{12}$}

\psline[linewidth=2.5pt]{-}(2,8)(2,6)
\pscurve{-}(2,6)(1.4,4)(2,2)
\psline{-}(2,4)(2,2)
\psline{-}(2,2)(2,1)
\pscurve{-}(4,8)(3.4,6)(4,4)
\psline{-}(4,6)(4,4)
\psline{-}(4,4)(4,1)
\psline{-}(6,8)(6,6)
\psline[linewidth=2.5pt, linecolor=lightgray]{-}(6,6)(6,3)
\psline{-}(8,8)(8,7)

\psdots[dotstyle=*,linecolor=black](2,8)
\psdots[dotstyle=*,linecolor=black](2,6)
\psdots[dotsize=9pt 0,dotstyle=*,linecolor=black](2,4)
\psdots[dotstyle=*,linecolor=black](2,2)
\psdots[dotstyle=*,linecolor=black](4,4)
\psdots[dotsize=9pt 0,dotstyle=*,linecolor=black](4,6)
\psdots[dotstyle=*,linecolor=black](4,8)
\psdots[dotstyle=*,linecolor=black](6,8)
\psdots[dotstyle=*,linecolor=black](8,8)
\psdots[dotstyle=*,linecolor=black](6,6)
\psdots[dotstyle=*,linecolor=black](9,8)

\uput[r](9.08,8.12){\large{1}}
\psdots[dotstyle=*,linecolor=black](8,7)
\uput[r](8.08,7.12){\large{2}}
\psdots[dotstyle=*,linecolor=black](7,6)
\uput[r](7.08,6.12){\large{3}}
\psdots[dotstyle=*,linecolor=black](7,4)
\uput[r](7.08,4.12){\large{4}}
\psdots[dotstyle=*,linecolor=black](6,3)
\uput[r](6.08,3.12){\large{5}}
\psdots[dotstyle=*,linecolor=black](5,2)
\uput[r](5.08,2.12){\large{6}}
\psdots[dotstyle=*,linecolor=black](4,1)
\uput[r](4.08,1.12){\large{7}}
\psdots[dotstyle=*,linecolor=black](2,1)
\uput[r](2.08,1.12){\large{8}}

\end{pspicture*}\quad\begin{pspicture*}(0,0)(10,9)
\psline{-}(5,2)(2,2) \uput[u](2.5,2){$p_2$}
\psline{-}(7,4)(4,4) \uput[u](4.5,4){$p_4$}
\psline{-}(4,4)(2,4) \uput[u](2.5,4){$m_5$}
\psline[linewidth=2.5pt]{-}(7,6.1)(6,6.1) \uput[u](6.5,6){$p_6$}
\psline[linewidth=2.5pt, linecolor=lightgray]{-}(7,5.9)(6,5.9)
\psline{-}(6,6)(4,6) \uput[u](4.5,6){$m_7$}
\pscurve{-}(6,6)(4,7)(2,6) \uput[u](2.5,6.2){$p_8$}
\psline[linewidth=2.5pt]{-}(9,8)(8,8) \uput[u](8.5,8){$p_9$}
\psline[linewidth=2.5pt]{-}(8,8)(6,8) \uput[u](6.5,8){$p_{10}$}
\psline{-}(6,8)(4,8) \uput[u](4.5,8){$p_{11}$}
\psline{-}(4,8)(2,8) \uput[u](2.5,8){$p_{12}$}

\psline{-}(2,8)(2,6)
\pscurve{-}(2,6)(1.4,4)(2,2)
\psline{-}(2,4)(2,2)
\psline{-}(2,2)(2,1)
\pscurve{-}(4,8)(3.4,6)(4,4)
\psline{-}(4,6)(4,4)
\psline{-}(4,4)(4,1)
\psline[linewidth=2.5pt]{-}(6,8)(6,6)
\psline[linewidth=2.5pt, linecolor=lightgray]{-}(6,6)(6,3)
\psline{-}(8,8)(8,7)

\psdots[dotstyle=*,linecolor=black](2,8)
\psdots[dotstyle=*,linecolor=black](2,6)
\psdots[dotsize=9pt 0,dotstyle=*,linecolor=black](2,4)
\psdots[dotstyle=*,linecolor=black](2,2)
\psdots[dotstyle=*,linecolor=black](4,4)
\psdots[dotsize=9pt 0,dotstyle=*,linecolor=black](4,6)
\psdots[dotstyle=*,linecolor=black](4,8)
\psdots[dotstyle=*,linecolor=black](6,8)
\psdots[dotstyle=*,linecolor=black](8,8)
\psdots[dotstyle=*,linecolor=black](6,6)
\psdots[dotstyle=*,linecolor=black](9,8)

\uput[r](9.08,8.12){\large{1}}
\psdots[dotstyle=*,linecolor=black](8,7)
\uput[r](8.08,7.12){\large{2}}
\psdots[dotstyle=*,linecolor=black](7,6)
\uput[r](7.08,6.12){\large{3}}
\psdots[dotstyle=*,linecolor=black](7,4)
\uput[r](7.08,4.12){\large{4}}
\psdots[dotstyle=*,linecolor=black](6,3)
\uput[r](6.08,3.12){\large{5}}
\psdots[dotstyle=*,linecolor=black](5,2)
\uput[r](5.08,2.12){\large{6}}
\psdots[dotstyle=*,linecolor=black](4,1)
\uput[r](4.08,1.12){\large{7}}
\psdots[dotstyle=*,linecolor=black](2,1)
\uput[r](2.08,1.12){\large{8}}

\end{pspicture*}
\end{center}
\caption{The three possible concatenations of pseudopaths  from $1$ to $5$.}
\label{fig:cancelling-stacks}
\end{figure}

Finally we explain how our proof also yields 
Corollary \ref{cor:coincide}.

\begin{proof}
We first note that the Marsh-Rietch parameterizations of Deodhar components
restrict to parameterizations of cells in the totally non-negative part
of the complete flag variety (using Lusztig's definition of total positivity), 
if $\v$ is a positive distinguished
subexpression of $\w$, and the parameters $p_i$ range over 
$\R_{>0}$ \cite{MR}.

Our proof of Theorem \ref{th:main} shows that
if one takes a (particular) Marsh-Rietsch parameterization of a cell 
in the non-negative part of the complete flag variety, 
then projects it to $(Gr_{k,n})_{\geq}$ (using Lusztig's definition
of total positivity),
and then uses an invertible transformation of variables, one
gets a network parameterizations of a cell of $(Gr_{k,n})_{\geq 0}$
(using Postnikov's definition of total positivity).
It follows that Lusztig's definition of $(Gr_{k,n})_{\geq 0}$ coincides
with Postnikov's definition of $(Gr_{k,n})_{\geq 0}$, and moreover that
the cell decompositions coincide.
\end{proof}

\section{A characterization of Deodhar components in terms of
Pl\"ucker coordinates}
\label{sec:characterization}

In this section we
characterize Deodhar components in the Grassmannian by
a list of vanishing and nonvanishing Pl\"ucker coordinates. Our main result in this section is Theorem~\ref{th:characterize}.
The proof uses results of Kodama and the second author
from \cite{KW2}, which gave formulas for Pl\"ucker coordinates
of Deodhar components.

\subsection{Pl\"ucker coordinates of Deodhar components
in terms of the MR parameters} \label{Plucker}

Consider the Deodhar component
$\mathcal P_{\v,\w} \subset Gr_{k,n}$, where
$\w$ is a reduced expression for $w \in W^k$ and $\v \prec \w$.
In this section we will review some formulas from \cite{KW2}
for the Pl\"ucker coordinates
of the elements of $\mathcal P_{\v,\w}$, in terms of the parameters
which Marsh and Rietsch \cite{MR} used
to define $G_{\v,\w}$.

\subsubsection{Formulas for Pl\"ucker coordinates}

\begin{theorem}\label{p:maxmin}\cite[Lemma 5.1 and Theorem 5.2]{KW2}
Let $\w=s_{i_1} \dots s_{i_m}$ be a reduced expression for $w \in W^k$ and $\v \prec \w$ be a distinguished subexpression for $v$.
Let $A = \pi_k(g)\in \mathcal P_{\v,\w}$ for any $g \in G_{\v,\w}$.
Then the lexicographically minimal and maximal
nonzero Pl\"ucker coordinates of $A$ are
$\Delta_I$ and $\Delta_{I'}$, where
$I = w\{n,n-1,\dots,n-k+1\}$ and
$I' = v\{n,n-1,\dots,n-k+1\}$.
If we write $g = g_1 \dots g_m$ as in Definition \ref{d:factorization}, then
\begin{equation}
\Delta_I(A) = (-1)^{|J_{\v}^{\bullet}|} \prod_{i\in J_{\v}^{+}} p_i\qquad  \text{ and }\qquad
\Delta_{I'}(A) = 1.
\end{equation}
\end{theorem}

\begin{remark}
If we write $I = \{i_1,\dots,i_k\}$, then
$I' = vw^{-1}\{i_1,\dots,i_k\}$.
\end{remark}

\begin{definition}\label{inout}\cite[Definition 5.4]{KW2}
Let $W = \Sym_n$,
let $\w=s_{i_1} \dots s_{i_m}$ be a reduced expression for $w \in W^k$ and
choose $\v \prec \w$.   This determines a Go-diagram $D$
of shape $\lambda=\lambda_w$.
Let $b$ be any box of $D$.  Note that the set of all boxes of $D$ which are weakly southeast
of $b$ forms a Young diagram $\lambda_b^{\In}$; also the complement of $\lambda_b^{\In}$ in $\lambda$ is a Young diagram
which we call $\lambda_b^{\Out}$ (see Example \ref{ex:out} below).  By looking at the restriction of $\w$ to the positions corresponding
to boxes of $\lambda_b^{\In}$, we obtain a reduced expression $\w_b^{\In}$ for some permutation
$w_b^{\In}$, together with a distinguished subexpression $\v_b^{\In}$ for some permutation $v_b^{\In}$.
Similarly, by using the positions corresponding to boxes of $\lambda_b^{\Out}$, we obtain
$\w_b^{\Out}$, $w_b^{\Out}$, $\v_b^{\Out}$, and $v_b^{\Out}$.  When the box $b$ is understood,
we will often omit the subscript $b$.

For any box $b$, note that it is always possible to choose a reading order of $\lambda=\lambda_w$ which orders
all the boxes of $\lambda^{\Out}$ after those of $\lambda^{\In}$.  We can then adjust $\w$ accordingly;
this does not affect whether the corresponding
expression $\v$ is distinguished.
Having chosen such a reading order, we can then write
$\w = \w^{\In} \w^{\Out}$ and $\v = \v^{\In} \v^{\Out}$.  We then use
$g^{\In}$ and $g^{\Out}$ to denote the corresponding factors of
$g \in G_{\v,\w}$.
We define $J^{+}_{\v^{\Out}}$ to be the subset of $J^{+}_{\v}$
coming from the factors of $\v$ contained in $\v^{\Out}$.  Similarly,
for $J^{\circ}_{\v^{\Out}}$
and $J^{\bullet}_{\v^{\Out}}$.
\end{definition}

\begin{example}\label{ex:out}
Let $W =\Sym_7$ and $\w = s_4 s_5 s_2 s_3 s_4 s_6 s_5 s_1 s_2 s_3 s_4$
be a reduced expression for $w \in W^3$.
Let $\v = s_4 s_5 1 1 s_4 1 s_5 s_1 1 1 s_4$ be a distinguished subexpression.
So $w = (3,5,6,7,1,2,4)$ and $v=(2,1,3,4,6,5,7)$.
We can represent this data by the poset $\lambda_w$ and the corresponding
 Go-diagram:\\
\setlength{\unitlength}{0.7mm}
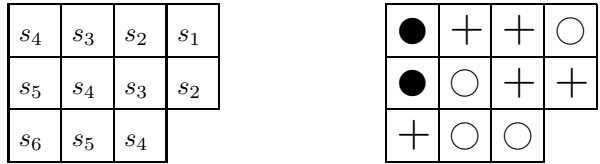
\begin{center}
  \begin{picture}(60,30)

  \put(5,35){\line(1,0){40}}
  \put(5,25){\line(1,0){40}}
  \put(5,15){\line(1,0){40}}
  \put(5,5){\line(1,0){30}}
  \put(5,5){\line(0,1){30}}
  \put(15,5){\line(0,1){30}}
  \put(25,5){\line(0,1){30}}
  \put(35,5){\line(0,1){30}}
  \put(45,15){\line(0,1){20}}

  \put(7,28){$s_4$}
  \put(17,28){$s_3$}
  \put(27,28){$s_2$}
  \put(37,28){$s_1$}
  \put(7,18){$s_5$}
  \put(17,18){$s_4$}
  \put(27,18){$s_3$}
  \put(37,18){$s_2$}
  \put(7,8){$s_6$}
  \put(17,8){$s_5$}
  \put(27,8){$s_4$}

  \end{picture}
  \qquad
  \begin{picture}(50,30)

  \put(5,35){\line(1,0){40}}
  \put(5,25){\line(1,0){40}}
  \put(5,15){\line(1,0){40}}
  \put(5,5){\line(1,0){30}}
  \put(5,5){\line(0,1){30}}
  \put(15,5){\line(0,1){30}}
  \put(25,5){\line(0,1){30}}
  \put(35,5){\line(0,1){30}}
  \put(45,15){\line(0,1){20}}

  \put(10,30){\circle*{5}}
  \put(17,28){\scalebox{1.7}{$+$}}
  \put(27,28){\scalebox{1.7}{$+$}}
  \put(40,30){\circle{5}}
   \put(10,20){\circle*{5}}
   \put(20,20){\circle{5}}
   \put(27,18){\scalebox{1.7}{$+$}}
   \put(37,18){\scalebox{1.7}{$+$}}

  \put(30,10){\circle{5}}
  \put(20,10){\circle{5}}
  \put(7,8){\scalebox{1.7}{$+$}}

  \end{picture}
\end{center}

Let $b$ be the box of the Young diagram which is in the second row
and the second column (counting from left to right).
Then the diagram below shows: the
boxes of
$\lambda^{\In}$ and
$\lambda^{\Out}$;
a reading order which
puts the boxes of $\lambda^{\Out}$ after those of $\lambda^{\In}$; and
the corresponding labeled Go-diagram.
Using this reading order,
$\w^{\In} = s_4 s_5 s_2 s_3 s_4$, $\w^{\Out} = s_6 s_5 s_1 s_2 s_3 s_4$,
$\v^{\In} = s_4 s_5 1 1 s_4$, and $\v^{\Out} = 1 s_5 s_1 1 1 s_4$.
\setlength{\unitlength}{0.8mm}
\begin{center}
  \begin{picture}(60,40)

  \put(5,35){\line(1,0){40}}
  \put(5,25){\line(1,0){40}}
  \put(5,15){\line(1,0){40}}
  \put(5,5){\line(1,0){30}}
  \put(5,5){\line(0,1){30}}
  \put(15,5){\line(0,1){30}}
  \put(25,5){\line(0,1){30}}
  \put(35,5){\line(0,1){30}}
  \put(45,15){\line(0,1){20}}

  \put(6,28){out}
  \put(16,28){out}
  \put(26,28){out}
  \put(36,28){out}
  \put(6,18){out}
  \put(18,18){in}
  \put(28,18){in}
  \put(38,18){in}
  \put(6,8){out}
  \put(18,8){in}
  \put(28,8){in}
  \end{picture}
  \begin{picture}(60,40)
  \put(5,35){\line(1,0){40}}
  \put(5,25){\line(1,0){40}}
  \put(5,15){\line(1,0){40}}
  \put(5,5){\line(1,0){30}}
  \put(5,5){\line(0,1){30}}
  \put(15,5){\line(0,1){30}}
  \put(25,5){\line(0,1){30}}
  \put(35,5){\line(0,1){30}}
  \put(45,15){\line(0,1){20}}

  \put(7,28){$11$}
  \put(17,28){$10$}
  \put(27,28){$~9$}
  \put(37,28){$~8$}
  \put(7,18){$~7$}
  \put(17,18){$~5$}
  \put(27,18){$~4$}
  \put(37,18){$~3$}
  \put(7,8){$~6$}
  \put(17,8){$~2$}
  \put(27,8){$~1$}

  \end{picture}
   \begin{picture}(45,40)
  \put(5,35){\line(1,0){40}}
  \put(5,25){\line(1,0){40}}
  \put(5,15){\line(1,0){40}}
  \put(5,5){\line(1,0){30}}
  \put(5,5){\line(0,1){30}}
  \put(15,5){\line(0,1){30}}
  \put(25,5){\line(0,1){30}}
  \put(35,5){\line(0,1){30}}
  \put(45,15){\line(0,1){20}}
  \put(5,25){\line(1,1){10}}
  \put(5,15){\line(1,1){10}}

  \put(5,31){$-1$}
  \put(8.5,27){\small{$m_{11}$}}
  \put(17,28){$p_{10}$}
  \put(27,28){$~p_{9}$}
  \put(37,28){$~1$}
  \put(5,21){$-1$}
  \put(8.5,17){$m_7$}
  \put(17,18){$~1$}
  \put(27,18){$~p_4$}
  \put(37,18){$~p_3$}
  \put(7,8){$~p_6$}
  \put(17,8){$~1$}
  \put(27,8){$~1$}
  \end{picture}
\end{center}
Note that $J^{\bullet}_{\v^{\Out}}=\{7, 11\}$ and
$J^{+}_{\v^{\Out}}=\{6,9,10\}$.
Then  $g \in G_{\v,\w}$ has the form
$$g = g^{\In} g^{\Out} = (\dot s_4 \dot s_5 y_2(p_3) y_3(p_4) \dot s_4 )\ (y_6(p_6) x_5(m_7) \dot s_5^{-1} \dot s_1
y_2(p_9) y_3(p_{10}) x_4(m_{11}) s_4^{-1}).$$
When we project the resulting $7 \times 7$ matrix to its first three columns, we get
the matrix
\[
A=
\begin{pmatrix}
-p_9 p_{10} & -p_3 p_{10}& -p_{10} & -m_{11} & 0 & -1 & 0\\
0 & -p_3 p_4 & -p_4 & -m_7 & 1 & 0 & 0\\
0 & 0 & 0 & p_6 & 0 & 0 & 1\\
\end{pmatrix}
\]
\end{example}

\begin{theorem}\label{p:Plucker}\cite[Theorem 5.6]{KW2}
Let $\w=s_{i_1} \dots s_{i_m}$ be a reduced expression for $w \in W^k$ and $\v \prec \w$,
and let $D$ be the corresponding Go-diagram.  Choose any box $b$ of $D$,
and let $v^{\In} = v_b^{\In}$ and $w^{\In} = w_b^{\In}$, and
 $v^{\Out} = v_b^{\Out}$ and $w^{\Out} = w_b^{\Out}$.
Let $A = \pi_k(g)$ for any $g \in G_{\v,\w}$, and let
$I = w\{n,n-1,\dots,n-k+1\}$.
If $b$ contains a $+$, define $I_b = v^{\In} (w^{\In})^{-1} I \in {[n] \choose k}$.
If $b$ contains a white or black stone, define $I_b = v^{\In} s_{b} (w^{\In})^{-1} I \in {[n] \choose k}$.
If we write $g = g_1 \dots g_m$ as in Definition \ref{d:factorization}, then
\begin{enumerate}
\item If $b$ contains a $+$, then $\Delta_{I_b}(A) = (-1)^{|J_{\v^{\Out}}^{\bullet}|} \prod_{i\in J_{\v^{\Out}}^+} p_i.$
\item If $b$ contains a white stone, then $\Delta_{I_b}(A) = 0.$
\item If $b$ contains a black stone, then $\Delta_{I_b}(A) = (-1)^{|J_{\v^{\Out}}^{\bullet}|+1} m_b \prod_{i\in J_{\v^{\Out}}^+} p_i
+ \Delta_{I_b}(A_b),$ where $m_b$ is the parameter corresponding to $b$,
and $A_b$ is the matrix $A$ with $m_b = 0$.
\end{enumerate}

\end{theorem}

\begin{example}\label{ex:out2}
We continue Example \ref{ex:out}.
By Theorem \ref{p:maxmin}, $I = w \{5,6,7\} = \{1,2,4\}$ and
$I' = v \{5,6,7\} = \{5,6,7\}$, and
the lexicographically minimal and maximal
nonzero Pl\"ucker coordinates for $A$ are
$\Delta_I(A) = p_3 p_4 p_6 p_9 p_{10}$ and
$\Delta_{I'}(A) = 1$; this can be verified
for the matrix $A$ above.

We now verify Theorem \ref{p:Plucker}
for the box $b$ labeled $5$ in the reading order.  Then
$I_b = v^{\In} (w^{\In})^{-1}I = \{1,4,6\}$.
Theorem \ref{p:Plucker} says that
$\Delta_{I_b}(A) = 0$, since this box contains a white stone.
The analogous computations for the boxes labeled
$7$, $6$, $4$, $3$, $2$, $1$, respectively,
yield $\Delta_{1,5,7} = -p_9 p_{10}$,
$\Delta_{1,2,7} = p_3 p_4  p_9 p_{10}$,
$\Delta_{1,4,5} = p_6 p_9 p_{10}$,
$\Delta_{1,3,4} = p_4 p_6 p_9 p_{10}$,
$\Delta_{1,2,4} =p_3 p_4 p_6 p_9 p_{10}$, and
$\Delta_{1,2,4} =p_3 p_4 p_6 p_9 p_{10}$.
These can be checked for the matrix $A$ above.
\end{example}

\begin{proposition}\cite[Corollary 5.11]{KW2} For any box $b$, the rescaled Pl\"ucker coordinate
$$\frac{\Delta_{I_b}(A)}{\prod_{i\in J_{\v}^+} p_i}$$ depends only on
the parameters $p_{b'}$ and $m_{b'}$ which correspond to boxes
$b'$ weakly southeast of $b$ in the Go-diagram.
\end{proposition}

\subsection{The characterization of Deodhar components
in terms of minors}
Given a Go-diagram $D$  of shape $\lambda$, contained in
a $k \times (n-k)$ rectangle,
let $I = I(\lambda).$
It is not hard to check that if $D$ corresponds to the
distinguished subexpression $\v$ of the reduced expression $\w$,
then $I = w\{n,n-1,\dots,n-k+1\}$.

\begin{theorem}\label{th:characterize}
Let $D$ be a Go-diagram of shape $\lambda$ contained in a $k \times (n-k)$ rectangle.
Let $A \in Gr_{k,n}$.  Then $A$ lies in the Deodhar component
$\mathcal{R}_D$ if and only if the following conditions are satisfied:
\begin{enumerate}
\item $\Delta_{I_b}(A) = 0$ for all boxes in $D$ containing a white stone.
\item $\Delta_{I_b}(A) \neq 0$ for all boxes in $D$ containing a $+$.
\item $\Delta_{I(\lambda)}(A) \neq 0$.
\item $\Delta_{J}(A) = 0$ for all $k$-subsets $J$ which are lexicographically
  smaller than $I(\lambda)$.
\end{enumerate}
\end{theorem}

\begin{proof}
Suppose that $A$ lies in the Deodhar component $\mathcal{R}_D$.
Then by Theorem \ref{p:maxmin}, conditions (3) and (4) hold.
And by Theorem \ref{p:Plucker}, conditions (1) and (2) hold.

Now suppose that $A \in Gr_{k,n}$, and conditions (1)-(4) hold.
We want to show that $A \in \mathcal{R}_D$.  Since the Deodhar components
partition $Gr_{k,n}$, it suffices to show that $A$ cannot
lie in $\mathcal{R}_{D'}$ for any other Go-diagram $D'$.
For the sake of contradiction, assume that $A \in \mathcal{R}_{D'}$.
Then by conditions (3) and (4) above, and Theorem
\ref{p:maxmin}, it follows that $D$ and $D'$ must be Go-diagrams
of the same shape.  Therefore $D$ and $D'$ correspond to
distinguished subexpressions $\v$ and $\v'$ of the same
reduced expression $\w$.

Choose a reading order
for the boxes of the Go-diagrams $D$ and $D'$, and let $b$ be the first
box in that order where $D$ and $D'$ differ.
Then without loss of generality, in $D$ the box $b$ must
contain a $+$, and in $D'$ the box $b$ must contain a stone.
(Because $\v$
and $\v'$ are distinguished subexpressions of the
same reduced word $\w$, which agree in the first $\ell$ factors,
and differ in the $(\ell+1)$st factor -- therefore one of $\v$
and $\v'$ must use the $(\ell+1)$st simple generator $s_b$,
and one must omit it.)
In fact it follows from the definition of distinguished
subexpression, and the fact that $D$ corresponds to a distinguished
subexpression,
that the box $b$ in $D'$ must contain a white stone,
not a black one.    (When building a distinguished subexpression
from left to right, if choosing the next simple generator $s_b$
would decrease the length of the word so far, then one must choose
$s_b$.)

Now note that the minor which Theorem \ref{p:Plucker}
associates to box $b$ in $D$ is
$\Delta_{I_b}$, where
$I_b = v^{\In} (w^{\In})^{-1}(I),$ and the minor
which the theorem associates to box $b$ in $D'$
is $\Delta_{I'_b}$, where
$I'_b = {{v'}^{\In} s_b ({w}^{\In})^{-1}(I)}$.
But now note that ${v'}^{\In} s_b = v^{\In}$.
Therefore $I_b = I'_b$.  So then by Theorem
\ref{p:Plucker},
if $A \in \mathcal{R}_{D'}$, then
$\Delta_{I_b}(A) = 0$, while by condition (2),
$\Delta_{I_b}(A) \neq 0$.
This is a contradiction.
\end{proof}

\begin{corollary}\label{cor:refine}
The Deodhar decomposition of the Grassmannian is a coarsening of the
matroid stratification:  in other words, each Deodhar component
is a union of matroid strata.
\end{corollary}
\begin{proof}
Each matroid stratum is defined by specifying that certain
Pl\"ucker coordinates are nonzero while the rest are zero.
Therefore the corollary is an immediate consequence of Theorem~\ref{th:characterize}.
\end{proof}

There is also an oriented version of Corollary \ref{cor:refine}.
To state this, we need a little preparation.
First we define the \emph{oriented matroid stratification} of
the real Grassmannian to be the decomposition into strata based
on which Pl\"ucker coordinates are $0$, positive, and negative.
Next note that from Definition
\ref{d:factorization} and Proposition
\ref{p:parameterization}, it is immediate that if
we are working over $\K = \R$, then the Deodhar
component coming from a Go-diagram $D$ has $2^r$ connected
components, where $r$ is the number of boxes in $D$ which contain
a $+$. We have the following result.

\begin{corollary}\label{cor:refine2}
Consider the decomposition of the real Grassmannian into
connected components of Deodhar components.  This is a coarsening
of the oriented matroid stratification: in other words,
each connected component of a Deodhar component is a union of
oriented matroid strata.
\end{corollary}

\begin{proof}
From Definition \ref{d:factorization} and Proposition
\ref{p:parameterization}, we see that the connected components
of a Deodhar component coming from $D$ are in bijection with the
$2^r$ ways of choosing a sign (either positive or negative)
for each of the $r$ parameters $p_{\ell}$
corresponding to the $+$ boxes of $D$.
By Theorems \ref{p:maxmin} and \ref{p:Plucker},
such a choice of signs determines the
sign pattern for the Pl\"ucker coordinates
$\Delta_{I(\lambda)}$ and
$\Delta_{I_b}$, where $b$ contains a $+$.
Conversely, suppose we know the signs for those Pl\"ucker coordinates.
Then we may algorithmically determine the signs of the $p_{\ell}$'s:
first we use Theorem \ref{p:maxmin} to determine the sign of the product of
all of the $p_{\ell}$'s; then we apply Theorem
\ref{p:Plucker} to each box $b$ containing a $+$, reading the boxes
in an order that proceeds from southeast to northwest.
(For example, by reading the rows from bottom to top, and the boxes
within each row from right to left.)
\end{proof}

\begin{remark}
Theorem~\ref{th:characterize} implicitly gives an algorithm for determining the Deodhar component and corresponding network of a point of the Grassmannian, given by a matrix representative or by a list of its Pl\"ucker coordinates.  The steps are as follows.
\begin{enumerate}
\item Find the lexicographically minimal nonzero Pl\"ucker coordinate $\Delta_I$.  Then the Go-diagram has shape $\lambda(I)$.  Fix a reading order for this shape.
\item We determine how to fill each box, working in the reading order, as follows.  First check whether the box $b$ is forced to contain a black stone.  If so, proceed to the next box.  If not, look at $\Delta_{I_b}$.  If this Pl\"ucker coordinate is zero, $b$ must contain a white stone, and if it is nonzero, $b$ must contain a $+$.
 Proceed to the next box.  This process will completely determine the Go-diagram.
\item Given the Go-diagram, we know what the underlying graph of the network must be.  To determine the weights on horizontal edges, work through them in the reading order again.  The Pl\"ucker coordinate $\Delta_{I_b}$ will only use edge weights $a_b$ (when $b$ contains a $+$) or $c_b$ (when $b$ contains a black stone) and weights $a_{b'}$ and $c_{b'}$ corresponding to boxes $b'$ which are earlier than $b$ in the reading order.  Thus, we may use the Lindstr\"om-Gessel-Viennot Lemma recursively to determine each weight $a_b$ or $c_b$.
\end{enumerate}
\end{remark}


\raggedright

\end{document}